\documentclass{conm-p-l}
\usepackage[nohug]{diagrams}
\usepackage{amscd}
\usepackage[mathscr]{eucal}
\usepackage{amsfonts}
\usepackage{amsmath}
\usepackage{amsthm}
\usepackage{amssymb}
\usepackage{latexsym}
\usepackage{tabularx,array}
\usepackage[all]{xy}
\newfont{\msam}{msam10}

\newtheorem{theorem}[]{Theorem}
\newtheorem{proposition}[]{Proposition}
\newtheorem{corollary}[]{Corollary}
\newtheorem{lemma}[]{Lemma}
\theoremstyle{definition}
\theoremstyle{Proposition}
\newtheorem{definition}[]{Definition}
\newtheorem{remark}[]{Remark}

\def\remark{\noindent\textbf{Remark.}}

\newtheorem{prop}[theorem]{Proposition}

\let\nc\newcommand

\def\ni{\noindent}
\def\bthm{\begin{theorem}}
\def\ethm{\end{theorem}}
\def\blemma{\begin{lemma}}
\def\elemma{\end{lemma}}
\def\bproof{\begin{proof}}
\def\eproof{\end{proof}}
\def\bprop{\begin{proposition}}
\def\eprop{\end{proposition}}
\def\bcor{\begin{corollary}}
\def\ecor{\end{corollary}}
\nc{\la}{\label}
\def\O{\mathcal{O}}

\def\Z{\mathbb{Z}}
\def\N{\mathbb{N}}

\def\M{\mathbb{M}}

\def\L{\boldsymbol{L}}
\def\R{\boldsymbol{R}}

\def\Com{\mathtt{Com}}

\def\Alg{\mathtt{Alg}}

\def\DGL{\mathtt{DGLA}}

\def\Bimod{\mathtt{Bimod}}
\def\cAlg{\mathtt{Comm\,Alg}}

\def\Sets{\mathtt{Sets}}
\def\DGA{\mathtt{DGA}}

\def\cDGA{\mathtt{CDGA}}

\def\DGMod{\mathtt{DG\,Mod}}
\def\DGBimod{\mathtt{DG\,Bimod}}
\def\D{{\mathscr D}}
\def\C{{\mathscr C}}

\def\Ab{{\mathscr Ab}}
\def\rtv#1{\!\!\sqrt[V]{#1}}

\def\Ho{{\mathtt{Ho}}}
\nc{\Ob}{{\rm Ob}}
\nc{\Mor}{{\rm Mor}}
\nc{\Hom}{{\rm{Hom}}}
\nc{\HOM}{\underline{\rm{Hom}}}
\nc{\DER}{\underline{\rm{Der}}}
\nc{\END}{\underline{\rm{End}}}
\nc{\Ext}{{\rm{Ext}}}
\nc{\Rep}{{\rm{Rep}}}
\nc{\DRep}{{\rm{DRep}}}
\nc{\NCRep}{\widetilde{\rm{Rep}}}
\nc{\RAct}{{\rm{RAct}}}
\nc{\bs}{\backslash}
\nc{\cn}{ \mbox{\rm cone} }
\nc{\ev}{{\tt{ev}}}
\nc{\n}{{\natural}}
\nc{\nn}{{{\natural} {\natural}}}
\nc{\colim}{{\tt{colim}}}
\nc{\B}{{\mathrm{B}}}
\nc{\Ba}{\overline{\mathrm{B}}}
\nc{\Ta}{\overline{\mathrm{T}}}
\nc{\bC}{\overline{C}}
\nc{\EXT}{\underline{\rm{Ext}}}
\nc{\TOR}{\underline{\rm{Tor}}}
\def\H{\mathrm H}
\def\tH{\mathrm{HC}}
\def\HC{\mathrm{HC}}

\def\rHC{\overline{\mathrm{HC}}}
\def\rHH{\overline{\mathrm{HH}}}

\def\CC{\mathrm{CC}}
\def\F{\mathcal F}
\def\G{\mathcal G}
\def\T{\mathrm T}
\def\ch{{\rm ch}}
\nc{\End}{{\rm{End}}}
\nc{\GL}{{\rm{GL}}}
\nc{\gl}{{\mathfrak{gl}}}
\nc{\m}{{\mathfrak{m}}}
\nc{\g}{{\mathfrak{g}}}
\nc{\PGL}{{\rm{PGL}}}
\nc{\SL}{{\rm{SL}}}
\nc{\PSL}{{\rm{PSL}}}
\nc{\ad}{{\rm{ad}}}
\nc{\Ad}{{\rm{Ad}}}
\nc{\dlim}{\varinjlim}
\nc{\plim}{\varprojlim}

\def\eR{R^{\mbox{\scriptsize{\rm{e}}}}}

\newcommand{\HH}{{\rm{HH}}}

\newcommand{\Tor}{{\rm{Tor}}}
\newcommand{\Spec}{{\rm{Spec}}}

\newcommand{\Sym}{{\rm{Sym}}}

\newcommand{\bSym}{\boldsymbol{\Lambda}}
\newcommand{\bL}{\boldsymbol{\Lambda}}

\newcommand{\id}{{\rm{Id}}}

\newcommand{\Der}{{\rm{Der}}}

\newcommand{\Tr}{{\rm{Tr}}}
\newcommand{\tTr}{\widetilde{\rm{Tr}}}

\newcommand{\tr}{{\rm{tr}}}
\newcommand{\Ker}{{\rm{Ker}}}

\newcommand{\Coker}{{\rm{Cok}}}
\newcommand{\im}{{\rm{Im}}}

\newcommand{\into}{\,\hookrightarrow\,}

\newcommand{\onto}{\,\twoheadrightarrow\,}
\newcommand{\sonto}{\,\stackrel{\sim}{\twoheadrightarrow}\,}

\newcommand{\rar}{\rightarrow}

\newcommand{\FT}{\mathcal{C}}

\newcommand{\Cyl}{\mathtt{Cyl}}
\newcommand{\Path}{\mathtt{Path}}
\newcommand{\sh}{\mathrm{Sh}}
\newenvironment{lyxlist}[1]
{\begin{list}{}
{\settowidth{\labelwidth}{#1}
 \setlength{\leftmargin}{\labelwidth}
 \addtolength{\leftmargin}{\labelsep}
 }}
{\end{list}}

\allowdisplaybreaks

\begin{document}

\title[Derived Representation Schemes]{Derived Representation Schemes and Noncommutative Geometry}
\author{Yuri Berest}
\address{Department of Mathematics,
 Cornell University, Ithaca, NY 14853-4201, USA}
\email{berest@math.cornell.edu}
\thanks{The work of Yu.~B. was partially supported by NSF Grant DMS 09-01570}
\author{Giovanni Felder}
\address{Departement Mathematik,
ETH Z\"urich,
8092 Z\"urich, Switzerland}
\email{giovanni.felder@math.ethz.ch}
\author{Ajay Ramadoss}
\address{Departement Mathematik,
ETH Z\"urich,
8092 Z\"urich, Switzerland}
\email{ajay.ramadoss@math.ethz.ch}


\begin{abstract}
Some 15 years ago M.~Kontsevich and A.~Rosenberg \cite{KR} proposed a heuristic principle
according to which the family of schemes $ \{\Rep_n(A)\} $ parametrizing the
finite-dimensional representations of a noncommutative algebra $A$ should be thought
of as a substitute or `approximation' for `$\Spec(A)$'. The idea is that every property
or noncommutative geometric structure on $ A $ should induce a corresponding geometric property or structure on $ \Rep_n(A) $ for all $n$. In recent years, many interesting structures in noncommutative geometry have originated from this idea. In practice, however, if an associative algebra $A$ possesses a
property of geometric nature (e.g., $A$ is a NC complete intersection, Cohen-Macaulay,
Calabi-Yau, etc.), it often happens that, for some $n$, the scheme $ \Rep_n(A) $ {\it fails} to have the corresponding property in the usual algebro-geometric sense.
The reason for this seems to be that the
representation functor $ \Rep_n $ is not `exact' and should be replaced by its
derived functor $ \DRep_n $ (in the sense of non-abelian homological algebra).
The higher homology of $ \DRep_n(A) $, which we call representation homology,
obstructs $ \Rep_n(A) $ from having the desired property and thus measures the
failure of the Kontsevich-Rosenberg `approximation.' In this paper, which is
mostly a survey, we prove several results confirming this intuition. We also give
a number of examples and explicit computations illustrating the theory
developed in \cite{BKR} and \cite{BR}.
\end{abstract}

\maketitle
\begin{flushright}
{\it To the memory of Alexander Rosenberg}
\end{flushright}
\section{Introduction}
If $ k $ is a field, the set of all representations of an associative  $k$-algebra $A$ in a finite-dimensional vector space
$ V $ can be given the structure of an affine $k$-scheme, called the representation scheme
$ \Rep_V(A) $. The group $ \GL_k(V) $
acts naturally on $ \Rep_V(A) $, with orbits corresponding to the isomorphism classes of representations.
If $ k $ is algebraically closed and $ A $ is finitely generated, the equivariant geometry of $ \Rep_V(A) $ is closely related to
the representation theory of $A$. This relation has been extensively studied (especially in the case of
finite-dimensional algebras) since the late 70's, and the schemes $ \Rep_V(A) $ have become a standard tool
in representation theory of  algebras (see, for example, \cite{Ga}, \cite{Bo}, \cite{Ge} and references therein).

More recently, representation schemes have come to play an important role in noncommutative  geometry. Let us recall
that in classical (commutative) algebraic geometry, there is a natural way to associate to a commutative algebra $A$
a geometric object ---  the Grothendieck prime spectrum $ \Spec(A) $. This defines a contravariant functor
from commutative algebras to affine schemes, which is  an (anti)equivalence of categories.
Attempts to extend this functor to the category of all associative algebras have been largely unsuccessful.
In \cite{KR} M. Kontsevich and A. Rosenberg proposed a heuristic principle according to which
the family of schemes $ \{\Rep_V(A) \} $ for a given  algebra $A$ should be thought of as a substitute (or ``approximation'') for
 ``$ \Spec(A) $''.  The idea is that every property or noncommutative geometric structure on $ A $  should naturally induce a
corresponding geometric property or structure on $ \Rep_V(A) $ for all $V$. This viewpoint provides a litmus test for proposed definitions
of noncommutative analogues of classical geometric notions. In recent years, many interesting structures in noncommutative
geometry have originated from this idea: NC smooth spaces  \cite{CQ, KR, LeB}, formal structures and noncommutative
thickenings of schemes \cite{Ka1, LBW}, noncommutative symplectic and bisymplectic geometry  \cite{Ko, G2, LeB1, CEG, Be, BC},
double Poisson brackets and noncommutative quasi-Hamiltonian spaces \cite{vdB, vdB1, CB, MT}.
In practice, however, the Kontsevich-Rosenberg principle works well only when $A$ is a (formally) smooth algebra, since in that case
$ \Rep_V(A) $ are smooth schemes for all $V$. To extend this principle to arbitrary algebras we proposed
in \cite{BKR} to replace $ \Rep_V(A) $ by a DG scheme $ \DRep_V(A) $, which is obtained by deriving
the classical representation functor in the sense of Quillen's homotopical algebra \cite{Q1, Q2}.  Passing from $ \Rep_V(A) $ to
$ \DRep_V(A) $ amounts, in a sense, to desingularizing $ \Rep_V(A) $, so one should expect that
$ \DRep_V(A) $ will play a role  similar to the role of $ \Rep_V(A) $ in the geometry of smooth algebras.

To explain this idea in more detail let us recall that the representation scheme $ \Rep_V(A) $ is defined as a functor on the category of commutative $k$-algebras:
\begin{equation}
\la{rep0}
\Rep_V(A):\ \cAlg_{k} \to \Sets\ ,\quad B \mapsto \Hom_{\Alg_k}(A,\, \End\,V \otimes_k B)\ .
\end{equation}
It is well known that  \eqref{rep0}  is (co)representable, and we denote the corresponding commutative algebra by
$\,A_V = k[\Rep_V(A)] \,$. Now, varying $A$ (while keeping $ V $ fixed) we can regard \eqref{rep0}  as a functor on the
category $ \Alg_k $ of associative algebras; more precisely, we define the {\it representation functor} in $ V $ by
\begin{equation}
\la{rep}
(\,\mbox{--}\,)_V :\, \Alg_k \to \cAlg_k\ , \quad A \mapsto k[\Rep_V(A)]  \ .
\end{equation}
The representation functor can be extended  to the category of differential graded (DG) algebras, $ \DGA_k $, which
has a natural model structure in the sense of \cite{Q1}. It turns out that $\,(\,\mbox{--}\,)_V\,$ defines
a left Quillen functor on $ \DGA_k $,
and hence it has a total derived functor $\, \L(\,\mbox{--}\,)_V: \Ho(\DGA_k) \to \Ho(\cDGA_k) \,$ on the homotopy category of DG algebras.  When applied to $A$,
this derived functor is represented  by a commutative DG algebra $ \DRep_V (A) $.
The homology of $\DRep_V (A)$ depends only on $A$ and $V$, with $ \H_0 [\DRep_V (A)] $
being isomorphic to $ k[\Rep_V (A)] $. Following \cite{BKR}, we will refer to $ \H_\bullet [\DRep_V (A)] $
as the {\it representation homology} of $A$ and denote it by $ \H_\bullet(A,V) $.
The action of $ \GL(V ) $ on $ k[\Rep_V (A)] $  extends naturally to $ \DRep_V (A) $, and we have an isomorphism of graded algebras
$$
\H_\bullet [\DRep_V (A)^{\GL(V )}] \cong \H_\bullet(A, V )^{\GL(V )}\ .
$$

Now, let $ \HC_\bullet(A) $ denote the cyclic homology of the algebra $A$. There is a canonical trace map
\begin{equation}
\la{trr}
\Tr_V(A)_0:\, \HC_0(A) \to k[\Rep_V (A)]^{\GL(V )}
\end{equation}
defined by taking characters of representations. One of the key results of \cite{BKR}
is the construction of the higher trace maps
\begin{equation}
\la{trr1}
\Tr_V(A)_n:\, \HC_n(A) \to \H_n(A,V)^{\GL(V)}\ ,\quad \forall\,n\ge 0\ ,
\end{equation}
extending \eqref{trr} to the full cyclic homology. It is natural to think of
\eqref{trr1}
as derived (or higher) characters of finite-dimensional representations of $A$. In accordance with Kontsevich-Rosenberg principle, various standard structures on cyclic and Hochschild homology (e.g., Bott periodicity, the Connes differential, the Gerstenhaber bracket, etc.) induce via \eqref{trr1} new interesting structures on representation homology. We illustrate this in Section~\ref{S5.33}, where we construct an analogue of Connes' periodicity exact sequence for $ \H_\bullet(A,V) $. We should mention that the idea of deriving representation schemes is certainly not new: the first construction of this kind was proposed in \cite{CK} ({\it cf.} Section~\ref{S2.3.5} below), and there are nowadays several different approaches (see, e.g., \cite{Ka}, \cite{BCHR}, \cite{TV}). However, the trace maps \eqref{trr1} seem to be new, and the relation to cyclic homology has not appeared in the earlier literature.

The aim of this paper is threefold. First, we give a detailed overview of \cite{BKR}
and \cite{BR} leaving out most technical proofs but adding motivation and necessary background
on homotopical algebra and model categories. Second, we prove several new results on
derived representation schemes refining and extending \cite{BKR}. Third, we give a number of explicit examples and computations illustrating the theory.

We would like to conclude this introduction with a general remark that clarifies the meaning of
representation homology from the point of view of noncommutative geometry. If an associative algebra
$A$  possesses a property of geometric nature (for example,  $A$ is a  NC complete intersection, Cohen-Macaulay,
Calabi-Yau, etc.), it may happen that, for some $V$, the scheme $ \Rep_V(A) $   {\it fails} to have a corresponding
property in the usual algebro-geometric sense. The reason for this seems to be that the representation functor
$ \Rep_V $ is not exact,  and it is precisely the higher homology $ \H_n(A,V) \,$, $\,n\ge 1 \,$,
that obstructs  $ \Rep_V(A) $ from having the desired property. In other words, representation
homology measures the failure of the Kontsevich-Rosenberg ``approximation.''  In Section~\ref{S6},
we prove two results confirming this intuition. First, we show that if $ A $ is a (formally) smooth algebra
then $ \H_n(A,V) = 0 \,$, $\,n\ge 1 \,$,  for all $ V $ (see Theorem~\ref{abt1}). This explains the well-known
fact that all representation schemes $ \Rep_V(A) $ of a smooth algebra are smooth.
Second, we prove that if $A$ is a noncommutative complete intersection (in the sense of \cite{Go})
then $ \Rep_V(A) $ is a complete intersection if  $ \H_n(A,V) = 0 $ for all $ n \ge 1 $ (see Theorem~\ref{RepHom}).

We now proceed with a summary of the contents of the paper.  Section~\ref{2} is a brief introduction
to Quillen's theory of model categories; in this section, we also recall some basic facts about DG algebras and
DG schemes needed for the present paper. In Section~\ref{S2}, we present our construction of
derived representation schemes and describe their basic properties. The main result of this section is
Theorem~\ref{S2T2}.  In Section~\ref{S3},  after reviewing the Feigin-Tsygan construction of
relative cyclic homology $ \HC_{\bullet}(S \bs A) $, we define canonical trace maps
$ \Tr_V(S \bs A)_\bullet:\,\HC_{\bullet -1}(S \bs A) \to \H_\bullet (S\bs A, V) $
relating the cyclic homology of an $S$-algebra $A$ to its representation homology.
In particular, for $ S = k $, we get the derived character maps \eqref{trr1}.
The main result of this section, Theorem~\ref{mcor}, describes an explicit chain
map $\, T: \CC(A) \to \DRep_V(A) \,$  that induces on homology the trace maps \eqref{trr1}.
We also draw reader's attention to Theorem~\ref{eqfun} and Corollary~\ref{corf1}
which summarize the main results of our forthcoming paper \cite{BR}.
In Section~\ref{S5}, we define and construct the abelianization of the representation
functor. The  main result of this section, Theorem~\ref{abvdb}, shows that the abelianized
representation functor is precisely (the DG extension of) Van den Bergh's functor
introduced in \cite{vdB}.  This is a new result that has not appeared in \cite{BKR}.
As a consequence, we give a simpler and more conceptual proof of Theorem~\ref{T2},
which was one of the main results of \cite{BKR}. Theorem~\ref{abvdb} also leads to
an interesting spectral sequence that clarifies the relation between representation
homology and Andr\`e-Quillen homology  (see Section~\ref{RHAQ}).
Finally, in Section~\ref{S6}, we give some examples. One notable result  is
Theorem~\ref{RepHompoly}  which says that
\begin{equation}
\la{vanhn}
\H_n(k[x,y],\,V) = 0\ ,\ \forall \, n > \dim\,V\ ,
\end{equation}
where $ k[x,y] $ is the polynomial algebra in two variables. We originally conjectured
\eqref{vanhn}  studying the  homology of $ k[x,y] $ with the help of {\tt Macaulay2}.
It came as a surprise that this vanishing result is a simple consequence of a known theorem of
Knutson \cite{Kn}.

\subsection*{Acknowledgements}{\footnotesize
The first author (Yu. B.) would like to thank K.~Iguza, A.~Martsinkovsky and G.~Todorov for inviting him to give an expository lecture at the 2012 Maurice Auslander Distinguished Lectures and International Conference.
This paper evolved from notes of that lecture and was written up during Yu.B.'s  stay at Forschungsinstitut
f\"ur Mathematik (ETH, Z\"urich) in Fall 2012. He is very grateful to this institution for its hospitality and financial support.}

\section*{Notation and Conventions}
Throughout this paper, $k$ denotes a base field of characteristic zero.
An unadorned tensor product $\, \otimes \,$ stands for the tensor product $\, \otimes_k \,$ over $k$.
An algebra means an associative $k$-algebra with $1$; the category of such algebras is denoted $ \Alg_k $.
Unless stated otherwise,  all differential graded (DG) objects (complexes, algebras, modules, $\,\ldots \,$)
are equipped with differentials of degree $-1$. The Koszul sign rule is systematically used.
For a graded vector space $V$,  we denote by $ \bL(V) $ its graded
symmetric algebra  over $k\,$: {\it i.e.}, $\, \bL(V) := \Sym_k(V_{\rm ev}) \otimes \Lambda_k(V_{\rm odd}) $,
where $ V_{\rm ev}$ and $ V_{\rm odd}$ are the even and the odd components of $V$.

\section{Model categories}
\la{2}
A model category is a category with a certain structure that allows one to do
non-abelian homological algebra (see  \cite{Q1,Q2}).  Fundamental examples
are the categories of topological spaces and simplicial sets.
However, the theory also applies to algebraic categories, including chain complexes,
differential graded algebras and differential graded modules.
In this section, we briefly recall the definition of model categories and review the results
needed for the present paper. Most of these results are well known; apart from the original works of Quillen,
proofs can be found in \cite{Hir} and \cite{Ho}.  For an excellent introduction we recommend the Dwyer-Spalinski
article \cite{DS}; for examples and applications of model categories in algebraic topology see \cite{GS} and \cite{He};
for spectacular recent applications in algebra we refer to the survey papers \cite{Ke} and \cite{S}.

\subsection{}\la{2.1}\textbf{Axioms.}
A {\it (closed) model category} is a category
$\C$ equipped with three distinguished classes of morphisms: {\it weak equivalences}
$ (\,\stackrel{\sim}{\to}\,) $, {\it fibrations} $ (\onto) $ and
{\it cofibrations} $ (\into) $. Each of these classes is closed under
composition and contains all identity maps. Morphisms that are both fibrations
and weak equivalences are called {\it acyclic fibrations} and denoted $\,\stackrel{\sim}{\onto}\,$.
Morphisms that are both cofibrations and weak equivalences are called {\it acyclic cofibrations}
and denoted $\,\stackrel{\sim}{\into}\,$.  The following five axioms are required.

\vspace{1ex}

\begin{lyxlist}{00.00.0000}
\item [{$\mathrm{MC}1$}]
$\C$ has all finite limits and colimits. In particular, $ \C $ has initial and terminal objects, which we denote
`$e$' and `$ * $', respectively. \smallskip{}

\item [{$\mathrm{MC}2$}] {\it Two-out-of-three axiom}:
If $f: X \to Y $ and $ g: Y \to Z $ are maps in $ \C $ and any
two of the three maps $f,$ $g,$ and $gf$ are weak equivalences, then so is the third. \smallskip{}

\item [{$\mathrm{MC}3$}]{\it Retract axiom}:
Each of the three distinguished classes of maps
is closed under taking retracts; by definition, $ f $ is a {\it retract} of $g$
if there is a commutative diagram
\begin{equation*}
\begin{diagram}[small, tight]
X               & \rTo  &  X'      &  \rTo  & X\\
\dTo^{f}        &       & \dTo^{g} &        & \dTo^{f} \\
Y               & \rTo  &  Y'      &  \rTo  & Y\\
\end{diagram}
\end{equation*}
such that the composition of the top and bottom rows is the identity.

\item [{$\mathrm{MC}4$}] {\it Lifting axiom}:
Suppose that
\begin{equation*}
\begin{diagram}[small, tight]
                   A &  \rTo^{} & X\\
\dInto^{}        &  \ruDotsto                      & \dOnto_{} \\
B                &  \rTo^{}                   & Y
\end{diagram}
\end{equation*}
is a square in which $ A \to B $ is a cofibration and $ X \to Y $ is a fibration.
Then, if either of the two vertical maps is a weak equivalence, there is a lifting $\,B \to X \,$ making the
diagram commute. We say that $ A \to B $ has the {\it left-lifting property} with respect to
$ X \to Y $, and $ X \to Y $ has a {\it right-lifting property} with respect to $ A \to B $.
\smallskip{}

\item [{$\mathrm{MC}5$}] {\it Factorization axiom}:
Any map $ A \to X $ in $\C$ may be factored in two ways:
$$
\mbox{\rm (i)}\ A \stackrel{\sim}{\into} B \onto X \ ,\qquad
\mbox{\rm (ii)}\ A \into Y \stackrel{\sim}{\onto} X \ .
$$
\end{lyxlist}

An object $ A \in\mathrm{Ob}(\C)$ is called {\it fibrant} if the unique morphism $\,A \to * \,$
is a fibration in $ \C$. Similarly, $\, A \in \mathrm{Ob}(\C)\,$ is {\it cofibrant} if the unique morphism $ e \to A $
is a cofibration in $ \C $. A model category $ \C $ is called {\it fibrant} (resp., {\it cofibrant})
if all objects of $ \C $ are fibrant (resp., {\it cofibrant}).

\vspace{1ex}

\noindent
\textbf{Remark.}
The notion of a model category was introduced by Quillen in \cite{Q1}.
He called such a category {\it closed} whenever any two of the three distinguished
classes of morphisms determined the third. In \cite{Q2}, Quillen characterized the closed model categories by the five axioms stated above. Nowadays, it seems generally agreed
to refer to a closed model category just as a model category. Also, in the current literature (see, e.g., \cite{Hir} and \cite{Ho}), the first and the last axioms in Quillen's list
are often stated in the stronger form: in MC1, one usually requires the existence of small (not only finite) limits and colimits, while MC5 assumes the existence of {\it functorial} factorizations.

\vspace{1ex}

\noindent
\textbf{Example.}
Let $A$ be an algebra, and let $ \Com^+(A) $ denote the category of complexes of $A$-modules that have zero terms in negative degrees. This category has a standard (projective) model structure, where the weak equivalences are the quasi-isomorphisms, the fibrations are the maps that are surjective in all positive degrees and the cofibrations are the monomorphisms whose cokernels are complexes with projective components.
The initial and the terminal objects in $ \Com^+(A) $ are the same, namely the zero complex. All objects are fibrant, while
the cofibrant objects are precisely the projective complexes ({\it i.e.}, the complexes consisting of projective modules
in each degree). A similar model structure exists on the category of complexes $ \Com^+(\mathscr{A}) $ over any
abelian category $ \mathscr{A} $ with sufficiently many projectives (see \cite{Q1}, \S \,I.1.2, Example~B).
The category $ \Com(A) $ of all (unbounded) complexes of $A$-modules also has a projective model structure with quasi-isomorphisms being the weak equivalences and the epimorphisms being the fibrations. The cofibrations in $ \Com(A) $ are monomorphisms with degreewise projective cokernels; however, unlike in $  \Com^+(A) $, not all such monomorphisms are cofibrations ({\it cf.} \cite{Ho}, Sect.~2.3).

\subsection{Natural constructions}
\la{newcon}
There are natural ways to build a new model category from a given one:

\subsubsection{}
\la{2.1.1} The axioms of a model category are self-dual: if $ \C $ is a model category, then so is its opposite
$ \C^{\rm opp}$. The classes of weak equivalences in $ \C $ and $ \C^{\rm opp} $ are the same, while
the classes of fibrations and cofibrations  are interchanged.

\subsubsection{}
\la{2.1.2}
If $ S \in \mathrm{Ob}(\C) $ is a fixed object in a model category $ \C $, then the category $ \C_S $ of arrows
$ \{S \to A \} $ starting at $S$ has a natural model structure, with a morphism $ f: A \rightarrow B $ being in a
distinguished class in $ \C_S $  if and only if $f$ is in the corresponding class in $ \C $. Dually, there
is a similar model structure on the category of arrows $ \{A \to S\} $ with target at $S$.

\subsubsection{}
\la{2.1.2m}
The category $ \Mor(\C) $ of all morphisms in a model category $ \C $ has a natural model structure, in which a morphism
$\,(\alpha, \beta):\, f \to f' \,$ given by the commutative diagram
\begin{equation*}
\begin{diagram}[small, tight]
A                 & \rTo^{\alpha}    &  A'         \\
\dTo^{f}          &                  & \dTo_{f'}   \\
B                 & \rTo^{\beta}     &  B'
\end{diagram}
\end{equation*}
is a weak equivalence (resp., a fibration) iff $ \alpha $ and $ \beta $ are weak equivalences (resp., fibrations)
in $ \C $. The morphism $ (\alpha, \beta) $ is a cofibration in $ \Mor(\C) $ iff $\, \alpha \,$ is a cofibration
and also the induced morphism $\,B \amalg_A A' \to B' \,$ is cofibration in $ \C $ ({\it cf.} \cite{R}).

\subsubsection{}
\la{2.1.3}
Let $ \D := \{a \leftarrow b \rightarrow c\} $ be the category with three objects $ \{a,b,c\} $ and the two
indicated non-identity morphisms. Given a category $ \C $, let $ \C^\D $ denotes the category of
functors $\, \D \to \C\,$. An object in $ \C^\D $ is pushout data in $ \C $:
$$
X(a) \leftarrow X(b) \rightarrow X(c)\ ,
$$
and a morphism $ \varphi: X \to Y $ is a commutative diagram
\begin{equation*}
\la{pushd}
\begin{diagram}[small, tight]
X(a)                 & \lTo    & X(b)             & \rTo             & X(c) \\
\dTo^{\varphi_a}     &         & \dTo^{\varphi_b} &                  &\dTo^{\varphi_c}  \\
Y(a)                 & \lTo    & Y(b)             & \rTo             & Y(c)
\end{diagram}
\end{equation*}
If $ \C $ is a model category, then there is a (unique) model structure on  $ \C^\D $,
where  $ \varphi $ is a weak equivalence (resp., fibration) iff
$ \varphi_a $, $\, \varphi_b $, $\, \varphi_c $ are weak equivalences (resp., fibrations)
in $ \C $. The cofibrations in $ \C^\D $ are described as the morphisms $ \varphi = (\varphi_a,\,
\varphi_b,\, \varphi_c) $, with $ \varphi_b $ being a cofibration and also the two induced maps
$\, X(a) \amalg_{X(b)} Y(b) \to Y(a) \,$, $\, X(c) \amalg_{X(b)} Y(b) \to Y(c) \,$ being cofibrations
in $ \C $.
Dually, there is a (unique) model structure on the category of pullback data
$ \C^\D $, where $ \D := \{a \rightarrow b \leftarrow c\} $.

\subsection{Homotopy category}
\la{2.2}
In an arbitrary model category, there are two different ways to define a homotopy equivalence relation.
For simplicity of exposition, we will assume that  $ \C $ is a {\bf fibrant} model category, in which case we can
use only one definition (`left' homotopy)  based on the cylinder objects.

If $ A \in \mathrm{Ob}(\C) $, a {\it cylinder} on $ A $ is an object $\Cyl(A) \in \Ob(\C) $
given together with a diagram
$$
A\amalg A \stackrel{i}{\into} \Cyl(A)\overset{\sim}{\onto}A \ ,
$$
factoring the natural map $\,(\mathrm{id},\, \mathrm{id}):\,A\amalg A\rightarrow A $.
By MC5(ii), such an object exists for all $A$ and comes together with two morphisms
$\, i_0:\, A \to \Cyl(A) \,$ and $\, i_1:\, A \to \Cyl(A) \,$, which are the restrictions of
$ i $ to the two canonical copies of $A$ in $ A \amalg A$.
In the category of topological spaces, there are natural cylinders:
$\, \Cyl(A) = A \times [0,\,1] $, with $\, i_0:\, A \to A \times [0,\,1] \,$ and
$\, i_1:\, A \to A \times [0,\,1]\,$ being the obvious embeddings. However,
in general, the cylinder objects $ \Cyl(A) $ are neither unique nor functorial in $A$.

Dually, if $X\in \mathrm{Ob}(\C)$, a {\it path object} on $X$ is an object $\Path(X)$ together
with a diagram
$$
X \stackrel{\sim}{\into} \Path(X) \stackrel{p}{\onto} X \times X\,
$$
factoring the natural map $\,(\mathrm{id},\, \mathrm{id}):\,X \rightarrow X \times X $.

If $\,f,g: A \rightarrow X\,$ are two morphisms in $\C$, a {\it homotopy}
from $f$ to $g$ is a map $\,H:\,\Cyl(A) \to X $ from a cylinder object on $A$ to $X$
such that the diagram
\begin{equation*}
\begin{diagram}[small, tight]
A & \rInto^{i_0} & \Cyl(A) & \lInto^{i_1} & A \\
  & \rdTo_{f} & \dDotsto^{H} & \ldTo_{g} & \\
  &           &   X
\end{diagram}
\end{equation*}
commutes. If such a map exists, we say that $ f $ is {\it homotopic} to $g$ and write $\,f \sim g \,$.

If $A$ is cofibrant, the homotopy relation between morphisms $\,f,g: A \rightarrow X\,$ can be
described in terms of path objects: precisely, $\,f \sim g\,$ iff there exists a map
$\,H:\,A \to \Path(X)\,$ for some path object on $X$ such that
\begin{equation*}
\begin{diagram}[small, tight]
  &           &   A\\
  & \ldTo^{f} & \dDotsto^{H} & \rdTo^{g} & \\
  X & \lOnto^{p_0} & \Path(X) & \rOnto^{p_1} & X
\end{diagram}
\end{equation*}
commutes. Also, if $A$ is cofibrant and $\,f \sim g\,$, then for any path object on $X$, there is a map
$\,H:A \to \Path(X)$ such that the above diagram commutes.

Applying $\mathrm{MC}5(\mathrm{ii})$ to the canonical morphism
$ e\rightarrow A,$ we obtain a cofibrant object $QA$ with
an acyclic fibration $ QA\overset{\sim}{\twoheadrightarrow}A.$ This
is called a {\it cofibrant resolution} of $A.$ As usual, a cofibrant resolution
is not unique, but it is unique up to homotopy equivalence: for any pair of cofibrant
resolutions $QA,$ $Q'A,$ there exist morphisms
\[ QA\underset{g}{\overset{f}{\rightleftarrows}}Q'A\]
such that $fg \sim \id $ and $gf \sim \id $.
By $\mathrm{MC}4$, for any morphism $f: A \rightarrow X $ and any cofibrant resolutions
$QA \sonto A$ and $ QX \sonto A $ there is a map $\tilde{f}: QA \to QX $ making the
following diagram commute:
\begin{equation}
\la{2.2.4}
\begin{diagram}[small, tight]
QA &  \rDotsto^{\tilde{f}} & QX\\
\dOnto       &                         & \dOnto \\
A              &  \rTo^{f}             & X
\end{diagram}
\end{equation}
We call this map a cofibrant lifting of $\,f\,$; it is uniquely determined by $f$ up to homotopy.

When $A$ and $X$ are both cofibrant objects in $ \C$, homotopy defines
an equivalence relation on $ \Hom_{\C}(A, X)$. In this case, we write
$$
\pi(A, X) :=  \Hom_{\C}(A, X)/\sim\ .
$$
The {\it homotopy category} of $ \C $ is now defined to be a category $ \Ho(\C) $
with $ \Ob(\Ho(\C)) = \Ob(\C) $ and
$$
\Hom_{\Ho(\C)}(A,\,X) := \pi(QA,\,QX)\ ,
$$
where $QA$ and $QX$ are cofibrant resolutions of $A$ and $X$.
For $A$ and $A'$  both cofibrant objects in $ \C $,  it is easy to check that
\begin{equation*}
f \sim h:\,A \to A'\ ,\quad g \sim k:\,A'\to X\quad
\Rightarrow\quad gf \sim hk:\,A \to X\ .
\end{equation*}
This ensures that the composition of morphisms in $ \Ho(\C) $ is well defined.

There is a canonical functor $\gamma:\,\C \rightarrow \Ho(\C)$
acting as the identity on objects while sending each morphism $ f \in \C $ to the homotopy
class of its lifting $ \tilde{f} \in \Ho(\C)\,$, see \eqref{2.2.4}.

\begin{theorem}
\la{Tloc}
Let $\C $ be a model category, and $ \D $ any category.
Given a functor $ F: \C \rightarrow \D $ sending weak equivalences to isomorphisms,
there is a unique  functor $ \bar{F}:\Ho(\C)\rightarrow \D $ such that $ \bar{F} \circ \gamma =F \,$.
\end{theorem}
Theorem~\ref{Tloc} shows that the category $ \Ho(\C) $ is the abstract (universal) localization of the category $ \C $ at the class $ W $ of weak equivalences. Thus $ \Ho(\C) $ depends only on $ \C $ and $ W $.
On the other hand, the model structure on $ \C $ is not
determined by $ \C $ and $ W \,$: it does depend the choice of fibrations and cofibrations in $\C$ (see \cite{Q1}, \S\,I.1.17, Example~3). The fibrations and cofibrations are needed to
control the morphisms in $ \Ho(\C)$.

\subsection{Derived functors}
\la{2.3}
Let $ F:\C \to \D$ be a functor between model categories.
A {\it (total) left derived functor} of $F$ is a functor $\,\L F: \,\Ho(\C) \to \Ho(\D) \,$ given together with a natural
transformation
\[
\L F: \,\Ho(\C) \to \Ho(\D)\ ,\qquad t:\,\L F\circ\gamma_{\C} \to \gamma_{\D} \circ F
\]
which are universal with respect to the following property: for any pair
\[
G:\Ho(\C) \rightarrow \Ho(\D),\qquad s: \, G\circ\gamma_{\C} \rightarrow \gamma_{\D} \circ F
\]
there is a unique natural transformation $\, s':\, G \rightarrow \L F\,$ such that
\[
\begin{diagram}[small, tight]
G\circ\gamma_{\C} &                        &  \rTo^{s}      &                &  \gamma_{\D} \circ F  \\
             & \rdDotsto_{s' \gamma}  &                &  \ruTo_{t} &  \\
             &                      & \L F \circ \gamma_{\C} &               &
\end{diagram}
\]
There is a dual notion of a right derived functor $\R F$ obtained by
reversing the arrows in the above definition ({\it cf.} \ref{2.1.1}).
If they exist, the functors $ \L F $ and $ \R F $ are unique up to
canonical natural equivalence.
If $ F $ sends weak equivalences to weak equivalence, then both
$ \L F $ and $ \R F $ exist and, by Theorem~\ref{Tloc},
$$
\L F = \gamma\,\bar{F} = \R F\ ,
$$
where $ \bar{F}: \Ho(\C) \to \D $ is the extension of $F$ to $ \Ho(\C)$.
In general, the functor $F$ does not extend to $ \Ho(\C) $, and
$ \L F $ and $ \R F $ should be viewed as the best possible approximations to
such an extension `from the left' and `from the right', respectively.

\subsection{The Adjunction Theorem}
\la{2.4}
One of the main results in the theory of model categories is Quillen's Adjunction Theorem.
This theorem consists of two parts: part one provides sufficient conditions for the existence
of derived functors for a pair of adjoint functors between model categories and part two
establishes a criterion for these  functors to induce an
equivalence  at the level of homotopy categories.
We will state these two parts as separate theorems.
We begin with the following  observation which is a direct consequence of basic axioms.
\blemma
\la{Qpair}
Let $ \C $ and $ \D $ be model categories. Let
\begin{equation*}
\la{Qp}
F:\, \C\rightleftarrows\D \,: G
\end{equation*}
be a pair of adjoint functors. Then the following conditions are equivalent:

$(a)$ $F$ preserves cofibrations and acyclic cofibrations,

$(b)$ $G$ preserves fibrations and acyclic fibrations,

$(c)$ $F$ preserves cofibrations and $G$ preserves fibrations.

\elemma

\noindent
A pair of functors  $ (F,G) $  satisfying the conditions of Lemma~\ref{Qpair} is called a {\it Quillen pair}; it should be
thought of as a `map' (or morphism) of model categories from $ \C $ to $ \D $.
The next theorem justifies this point of view.
\begin{theorem}
\la{Qthm}
Let  $\,F:\,\C\rightleftarrows\D :G\,$ be a Quillen pair. Then  the total derived functors
$\L F$ and $\R G$ exist and form an adjoint pair
\begin{equation}
\la{LGF}
\L F:\Ho(\C)\rightleftarrows \Ho(\D):\R G\, .
\end{equation}
The functor $\L F$ is defined by
\begin{equation}
\la{derf}
\L F(A) = \gamma\,F(QA)\ ,\quad \L F(f) = \gamma\,F(\tilde{f})\ ,
\end{equation}
where $ QA \sonto A $ is a cofibrant
resolution in $ \C $ and $ \tilde{f} $ is a lifting of $ f \,$, see \eqref{2.2.4}.
\end{theorem}

\noindent
For a detailed proof of Theorem~\ref{Qthm} we refer to \cite{DS}, Sect.~9; here,
we only mention one useful result on which this proof is based.
\blemma[K.~Brown]
\la{BrL}
If $ F:\,\C \to \D $ carries acyclic cofibrations between
cofibrant objects in $ \C $ to weak equivalences in $\D$,
then $\L F $ exists and is given by formula \eqref{derf}.
\elemma

\vspace{1ex}

\remark\
In the situation of Theorem~\ref{Qthm}, if $\D$ is a fibrant category,
then $ \R G = G $. This follows from the fact that the derived functor
$ \R G $ is defined by applying $ G $ to a fibrant resolution similar
to \eqref{derf}.

\vspace{1ex}

\noindent
\textbf{Example.}
Let $ \C^\D $ be the category of pushout data in a model category $ \C $
(see \ref{2.1.3}). The colimit construction gives a functor
$\,\colim:\,\C^\D \to \C \,$ which is left adjoint to the
diagonal (`constant diagram') functor
$$
\Delta:\,\C \to \C^\D \ , \quad A \mapsto \{A \xleftarrow{\id} A \xrightarrow{\id} A\}\ .
$$
Theorem~\ref{Qthm} applies in this situation giving the adjoint pair
$$
\L \colim:\Ho(\C^\D)\rightleftarrows \Ho(\C):\R \Delta\, .
$$
The functor $ \L \colim $ is called the {\it homotopy pushout functor}. Similarly
one defines the {\it homotopy pullback functor} $\, \R \lim \,$ which is right adjoint
to $\, \L \Delta \,$ (see \cite{DS}, Sect.~10).

\vspace{1ex}

Now, we state the second part of Quillen's Theorem.

\begin{theorem}
\la{Qthm2}
The derived functors \eqref{LGF} associated to a Quillen pair $ (F,G) $ are (mutually inverse)
equivalences of categories if and only if the following condition holds: for each cofibrant object
$ A \in \Ob(\C) $ and  fibrant object $ B \in \Ob(\D) $ a morphism $ f: A \to G(B) $ is a
weak equivalence in $ \C $ if and only if the adjoint morphism $ f^*: F(A) \to B $ is a weak
equivalence in $ \D $.
\end{theorem}

\noindent
A Quillen pair  $(F,G) $ satisfying the condition of Theorem~\ref{Qthm2}
is called a {\it Quillen equivalence}.  The fundamental example of a Quillen equivalence is the
geometric realization and the singular set functors relating the categories of simplicial sets and
topological spaces (see \cite{Q1}):
$$
|\,\mbox{--} \,| : \, {\mathtt{SSets}}  \rightleftarrows {\mathtt{Top}}\, : {\mathrm{Sing}}(\,\mbox{--}\,)\ .
$$
We give another well-known example coming from algebra. Recall that if $A$ is a DG algebra, the
category $ \DGMod(A) $ of DG modules over $A$
is abelian and has a natural  model structure, with weak equivalences being the quasi-isomorphisms.
\bprop
\la{Kelpr}
Let $ f: R \to A $ be a morphism of DG algebras. The corresponding induction and restriction functors form a Quillen pair
$$
f^*:\, \DGMod(R)  \rightleftarrows \DGMod(A) \,: f_*
$$
If $ f $ is a quasi-isomorphism, $ (f^*,  \, f_*) $ is a Quillen equivalence.
\eprop
Proposition~\ref{Kelpr} is a special case of a general result about module categories in
monoidal model categories proved in \cite{SS1} (see {\it loc. cit}, Theorem~4.3).

\subsection{Quillen homology}
\la{AQh}
For a category $ \C $, let $ \C^{\rm ab} $ denote the category of abelian objects in $ \C $.
Recall that $ A \in \Ob(\C) $ is an {\it abelian object} if the functor
$\, \Hom_{\C}(\,\mbox{--}\,,\,A) \,$  is naturally an abelian group.
Assuming that $ \C$ has enough limits, this is known to be equivalent to the `diagrammatic'
definition of an abelian group structure on $A $: {\it i.e.},
the existence of multiplication ($m: A \times A \to A $) , inverse ($ \iota:\, A \to A $) and unit
($ \ast \to A $) morphisms in $ \C $, satisfying the usual axioms of an abelian group
(see, e.g., \cite{GM}, Ch.~II, Sect.~3.10).  Note that the forgetful functor $ \, i: \C^{\rm ab} \to \C \,$  is faithful but
not necessarily full. For example,  the abelian objects in the categories $ \mathtt{Sets} $ and $ \mathtt{Groups} $
are the same: namely,  the abelian groups; however, $ \, i: \C^{\rm ab} \to \C \,$ is a full embedding only for
$ \C =  \mathtt{Groups}  $.

Now, let $ \C$ be a model category. Following Quillen (see \cite{Q1}, \S\, II.5),
we assume that  the forgetful functor $\, i: \C^{\rm ab} \to \C $ has a left adjoint $\,\Ab:\, \C \to \C^{\rm ab} $ called
{\it abelianization}, and there is a model structure on $ \C^{\rm ab} $ such that
\begin{equation}
\la{Abi}
\Ab:\, \C\rightleftarrows\C^{\rm ab} \,: i\
\end{equation}
is a Quillen pair. Then, by Theorem~\ref{Qthm}, $\, \Ab $ has a total left derived functor
$\, \L \Ab:\,\Ho(\C) \to \Ho(\C^{\rm ab}) $, which is called the {\it Quillen homology} of $ \C $.
Assume, in addition, that the model structure on $ \C^{\rm ab} $ is  stable, {\it i.e.} there is an invertible
suspension functor $ \,  \Sigma:\,\Ho(\C^{\rm ab}) \to \Ho(\C^{\rm ab}) \,$ making $  \Ho(\C^{\rm ab}) $
a triangulated category ({\it cf.} \cite{Ho}, Sect.~7.1). Then, for any $ X \in \Ob(\C) $ and
$ A \in \Ob(\C^{\rm ab}) $,  we can define the {\it Quillen cohomology of $X$ with coefficients
in $A\,$} by
$$
\H_{\C}^n(X, A) = \Hom_{\Ho(\C^{\rm ab})}(\L\Ab(X),\,\Sigma^{-n} A)\ .
$$
This construction unifies basic (co-)homology theories of spaces, groups and algebras
(see \cite{Q1}, \S~II.5).   We briefly discuss only three well-known examples related to algebras
(see \cite{Q4}).

\vspace{1ex}

\noindent
{\bf Example 1.}\
Let $ \C = \DGL_k $ be the category of DG Lie algebras over $k$.
This category has a natural model structure, with weak equivalences being quasi-isomorphisms
(see \cite{Q2}, Part~II, Sect. 5).  The abelian objects in $ \C $ are just the abelian Lie algebras ({\it i.e.}, the DG Lie algebras with zero
bracket). The category $ \C^{\rm ab} $ can thus be identified with $ \Com(k) $. The abelianization
functor $\, \Ab:  \DGL_k \to \Com(k)  \,$ is given by $\,\g \mapsto \g/[\g,\,\g] \,$. If $ \g $ is an
ordinary Lie algebra, and $ {\mathfrak L} \sonto \g $ is a cofibrant resolution of $ \g $ in $ \DGL_k$, then
\begin{equation}
\la{glie}
\H_n({\mathfrak L}/[{\mathfrak L},\,{\mathfrak L}]) \cong \H_{n+1}(\g,\,k)\ , \quad \forall\,n \ge 0\ .
\end{equation}
Thus, the Quillen homology of $ \g $ agrees with  the usual Lie algebra homology with trivial coefficients.

\vspace{1ex}

\noindent
{\bf Example 2.}
Let $  \DGA_k $ be the category of associative DG algebras\footnote{We will discuss the properties of this category as well as
its commutative counterpart in Section~\ref{3} below.}. Unlike in $ \DGL_k$, the only
abelian object in $ \DGA_k$ is the zero algebra. To get more interesting examples,
we fix an algebra $A \in \Ob(\DGA_k) $ and consider  the category $ \C := \DGA_k/A $ of algebras over $A$.
(So an object in $ \C $ is a DG algebra $B$ given together with  a DG algebra map $ B \to A $.)
In this case, it is easy to show that $ \C^{\rm ab} $ is equivalent to the (abelian) category $ \DGBimod(A) $
of DG bimodules over $A$. The equivalence is given by the semi-direct product construction
\begin{equation}
\la{semdir}
A \ltimes (\,\mbox{--}\,)\,:\  \DGBimod(A) \to \DGA_k/A \ ,
\end{equation}
assigning to a bimodule $M$ the DG algebra $ A \ltimes M $ together with the canonical projection
$ A \ltimes M \onto A $. Note that $ A \ltimes M $ is an abelian object in $ \C $ because
$\, \Hom_{\C}(B, A \ltimes M) \cong \Der_k(B, M) $,
where $ \Der_k(B,M) $ is an abelian group (in fact, a vector space) of $k$-linear derivations
$ \partial: B \to M $. On the other hand, for any $A$-bimodule $ M $,
there is a natural isomorphism
$$
\Der_k(B, M) \cong \Hom_{\DGBimod(A)}(\Omega_k^1(B/A) ,\,  M)\ ,
$$
where $ \Omega_k^1(B/A) := A \otimes_B  \Omega_k^1(B) \otimes_B A $ and $ \Omega_k^1 B $
denotes the kernel of the multiplication map $\, B \otimes B \to B \,$.
Thus, for  $ \C = \DGA_k/A $,  the Quillen pair \eqref{Abi} can be identified with
\begin{equation}
\la{nccot}
\Omega_k^1( \mbox{--}/A):\, \DGA_k/A  \rightleftarrows  \DGBimod(A)\,: A \ltimes (\,\mbox{--}\,)\ .
\end{equation}
If $A$ is an ordinary $k$-algebra, the Quillen homology of $ \C $ essentially coincides with Hochschild homology:
precisely, we have
\begin{equation}
\la{qhdga}
\H_n[\L\Omega_k^1(B/A)] =
\left\{
\begin{array}{lll}
\Omega_k^1(B/A)\ & \mbox{\rm if} &\ n = 0\\*[1ex]
\Tor^B_{n+1}(A, A)\ & \mbox{\rm if} &\ n \ge 1
\end{array}
\right.
\end{equation}
The derived abelianization functor $\,\L\Omega_k^1( A) \,$ evaluated at $ \id_A \in \C $
 is called the {\it noncommutative cotangent complex} of $A$. By \eqref{qhdga}, we simply have
$\,\L\Omega_k^1( A) \cong \Omega_k^1(A)\,$ in $ \Ho(\C) $. Similarly,
the Quillen cohomology of $A$ with coefficients in a bimodule $ M $ can be identified with
Hochschild cohomology of $ M $ (see \cite{Q4}, Proposition~3.6).

\vspace{1ex}

\noindent
{\bf Example 3.} Let $ \cDGA_k $ be the category of commutative DG $ k$-algebras.
As in the case of associative algebras, for any $ A \in \Ob(\cDGA_k) $, the semi-direct
product construction defines a  fully faithful functor
$$
A \ltimes (\,\mbox{--}\,)\,:\  \DGMod(A) \to \cDGA_k/A \  ,
$$
whose image is the subcategory of abelian objects in $ \cDGA_k/A $. The Quillen pair
\eqref{Abi} is then identified with
\begin{equation}
\la{ccot}
\Omega_{\rm com}^1( \mbox{--}/A):\, \cDGA_k/A  \rightleftarrows  \DGMod(A)\,: A \ltimes (\,\mbox{--}\,)\ .
\end{equation}
Here the abelianization functor is given by $\, \Omega_{\rm com}^1(B/A) := A \otimes_B  \Omega_{\rm com}^1(B) \,$,
where $ \Omega_{\rm com}^1(B) $ is
the module of K\"ahler differentials of the commutative $k$-algebra $B$.
The corresponding derived functor $ \L \Omega_{\rm com}^1(A) $ evaluated at
the identity morphism of $A$ is usually denoted $ \mathbb{L}_{k\bs A} $
and called the {\it cotangent complex} of $ A $. By definition, this is an object
in the homotopy category $ \Ho(\cDGA_k) $, which can be computed by the formula
$  \mathbb{L}_{k\bs A} \cong A \otimes_R  \Omega_{\rm com}^1(R) $, where
$ R \sonto A $ is a cofibrant resolution of $ A $. The homology of the cotangent
complex
$$
D_q(k\bs A) := \H_q(\mathbb{L}_{k\bs A}) \cong
\H_q[A \otimes_R  \Omega_{\rm com}^1(R)]
$$
is called the  {\it Andr\'e-Quillen homology} of $A$. More generally,
the Andr\`e-Quillen homology  with coefficients in an arbitrary module
$ M \in \DGMod(A) \,$ is defined by
\begin{equation}
\la{AnQh}
D_q(k\bs A, M) := \H_q(\mathbb{L}_{k\bs A} \otimes_A M) \ .
\end{equation}
Taking the Hom complex with $ \mathbb{L}_{k\bs A} $ instead of tensor product
defines the corresponding cohomology.  The  construction of Andr\`e-Quillen (co-)homology theory
was historically the first real application of model categories. The original paper of Quillen \cite{Q4} seems still to be the
best exposition of foundations of this theory. Many interesting examples and applications can be found in the
survey paper \cite{I}.

\subsection{Differential graded algebras}
\la{3}
By a {\it DG algebra} we mean a $\Z$-graded unital associative $k$-algebra equipped with a
differential of degree $-1$. We write $ \DGA_k $ for the
category of all such algebras and denote by $ \cDGA_k $ the full subcategory
of $ \DGA_k $ consisting of commutative DG algebras. On these categories,
there are standard model structures which we describe in the next theorem.

\begin{theorem}
\la{modax}
The categories $ \DGA_k $ and $ \cDGA_k$ have model structures in which

(i)\ the weak equivalences are the quasi-isomorphisms,

(ii)\ the fibrations are the maps which are surjective in all degrees,

(iii)\ the cofibrations are the morphisms having the left-lifting property
with respect to acyclic fibrations ({\it cf.} MC4).

\noindent
Both categories $ \DGA_k $ and $ \cDGA_k $ are fibrant, with the initial object $ k $ and
the terminal $0$.
\end{theorem}

\ni
Theorem~\ref{modax} is a special case of a general result of Hinich on model structure on
categories of algebras over an operad (see \cite{H}, Theorem~4.1.1 and Remark 4.2).
For $ \DGA_k$, a detailed proof can be found in \cite{J}. Note that the
model structure on $ \DGA_k $ is compatible with the projective model structure on the category
$ \Com_k $ of complexes. Since a DG algebra is just an algebra object (monoid) in $ \Com_k $,
Theorem~\ref{modax} follows also from \cite{SS1} (see {\it op. cit.}, Sect.~5).

It is often convenient to work with non-negatively graded DG algebras.
We denote the full subcategory of such DG algebras by $ \DGA_k^+ $ and the corresponding
subcategory of commutative DG algebras by $ \cDGA_k^+ $.
We recall that a DG algebra $\, R \in \DGA_k^+ \,$ is called {\it semi-free} if
its underlying graded algebra $ R_\# $ is free ({\it i.e.}, $ R_\# $ is
isomorphic to the tensor algebra $ T_k V $  of a graded $k$-vector space $V$).
More generally, we say that a DG algebra map $ f: A \to B $ in $ \DGA_k^+ $
is a  {\it semi-free extension} if there is an isomorphism $\,B_\# \cong A_\#  \amalg T_k V\,$
of underlying graded algebras such that the
composition of $ f_\# $ with this isomorphism is the canonical map $\,A_\# \into A_\#  \amalg T_k V \,$.
Here, $\, \amalg \,$ denotes the coproduct (free product) in the category of graded algebras over $k$.

Similarly, a commutative DG algebra $ S \in \cDGA_k^+ $ is called {\it semi-free}
if $\, S_\# \cong \bL_k V \,$ for some graded vector space $V$. A morphism
$ f: A \to B $ in $ \cDGA_k^+ $ is an (semi) {\it free extension} if $ f_\# $ is
isomorphic to the canonical map $\,A_\# \into A_\# \otimes \bL_k V  \,$.

\begin{theorem}
\la{modax1}
The categories $ \DGA_k^+ $ and $ \cDGA_k^+ $ have model structures in which

(i)\ the weak equivalences are the quasi-isomorphisms,

(ii)\ the fibrations are the maps which are surjective in all {\rm positive} degrees,

(iii)\ the cofibrations are the retracts of semi-free algebras ({\it cf.}\,MC3).

\noindent
Both categories $ \DGA_k^+ $ and $ \cDGA_k^+ $ are fibrant, with the initial object $ k $ and
the terminal $0$.
\end{theorem}

The model structure on $ \cDGA_k^+ $ described in Theorem~\ref{modax1} is a `chain' version
of a well-known model structure on the category of commutative cochain DG algebras. This
last structure plays a prominent role in rational homotopy theory and the verification of
axioms for $ \cDGA_k^+ $ can be found in many places (see, e.g., \cite{BG} or \cite{GM}, Chap.~V).
The model structure on $ \DGA_k^+ $ is also well known: a detailed proof of Theorem~\ref{modax1} for
$ \DGA_k^+ $ can be found in \cite{M}. The assumption that $ k $ has characteristic $0$ is essential
in the commutative case: without this assumption,$\, \cDGA_k^+ $ is not (Quillen) equivalent
the model category of simplicial commutative $k$-algebras. On the other hand, it is known that the model
category $ \DGA_k^+ $  is Quillen equivalent to the  model category of simplicial associative
$k$-algebras over an arbitrary commutative ring $k$ (see \cite{SS2}, Theorem~1.1).

\subsection{DG schemes}
\la{2.8}
Working with commutative DG algebras it is often convenient to use the dual geometric language of DG schemes.
In this section, we briefly recall basic definitions and facts about DG schemes needed for the
present paper. For more details, we refer to \cite{CK}, Section~2. We warn the reader that, unlike  \cite{CK},
we use the homological notation:  all our complexes and DG algebras have differentials of degree $-1$.

A {\it DG scheme} $ X = (X_0,\, \O_{X, \bullet}) $ is an ordinary $k$-scheme $ X_0 $ equipped
with a quasicoherent sheaf  $ \O_{X, \bullet} $ of non-negatively graded commutative DG algebras such
that $ \O_{X, 0} = \O_{X_0} $. A DG scheme is called {\it affine} if $ X_0$ is affine; the category of
affine DG schemes is (anti-)equivalent to $ \cDGA_k^+ $. Since $ \O_{X, \bullet} $ is non-negatively graded,
the differential $ d $ on  $ \O_{X, \bullet} $ is linear over $   \O_{X_0} $, and
$\,\H_0(\O_{X, \bullet}) = \O_{X_0}/d \O_{X,1} \,$ is the quotient of $ \O_{X_0} $. We  write
$\,\pi_0(X) := \Spec\, \H_0(\O_{X, \bullet})  \,$ and identify $ \pi_0(X) $ with a closed subscheme of $ X_0 $.

A DG scheme $ X $ is called {\it smooth} (or a {\it DG manifold}) if $ X_0 $ is a smooth variety, and
$ \O_{X, \bullet } $ is locally isomorphic (as a sheaf of graded $ \O_{X_0} $-algebras) to the graded
symmetric algebra
$$
\O_{X, \#}  =  \bL_{\O_{X_0}}(E_{\#})
$$
where $\, E_{\#} = \oplus_{i \ge 1} E_i \,$ is a graded $ \O_{X_0} $-module whose components
$ E_i $ are finite rank locally free sheaves on $ X_0 $.  (Note that we do not require $ E_\# $ to be
bounded, {\it i.e.} $ E_i $ may be nonzero for infinitely many $i$'s.)

Now, given a DG scheme $ X $ and a closed $k$-point $ x \in X_0 $, we define the {\it DG tangent space}
$\,(T_x X)_\bullet \,$ at $ x $ to be the derivation complex
\begin{equation}
\la{derc}
(T_x X)_\bullet := \DER(\O_{X, \bullet}, k_x)\ ,
\end{equation}
where $ k_x $ is the $1$-dimensional DG $\, \O_{X, \bullet}$-module corresponding
to $ x $. The homology groups of this complex are denoted
\begin{equation}
\la{dtang}
\pi_i(X, x) := \H_i(T_x X)\ ,\quad i \ge 0\ ,
\end{equation}
and called the {\it derived tangent spaces} of $ X $ at $x$. A morphism $ f: X \to Y $ of DG schemes
induces a morphism of complexes $\,(d_x f)_\bullet: \,(T_x X)_\bullet  \to (T_{y} Y)_\bullet \,$, and
hence linear maps
\begin{equation}
\la{mapd}
(d_x f)_i:\, \pi_i(X, x) \to \pi_i(Y, y)\ ,
\end{equation}
where $ y = f(x) $. Dually,  the {\it DG cotangent space} $ (T^*_x X)_\bullet  $ at a point $ x \in X_0 $
is defined by taking the complex of K\"ahler differentials:
\begin{equation*}
(T^*_x X)_\bullet := \Omega^1_{\rm com}(\O_{X, \bullet})_x = \m_x/\m_x^2 \ ,
\end{equation*}
where $ \m_x \subset \O_x $ is the maximal DG ideal corresponding to $x$.
The topological notation  \eqref{dtang} for the derived tangent spaces is justified by the following proposition,
which is analogous to the Whitehead Theorem in classical topology.
\bprop[\cite{Ka}, Proposition~1.3]
\la{Kapp}
Let $ f: X \to Y $ be a morphism of smooth DG schemes.
Then $ f $ is a quasi-isomorphism if and only if

$(1)$ $\,\pi_0(f):\,\pi_0(X) \stackrel{\sim}{\to} \pi_0(Y) \,$ is an isomorphism of schemes,

$(2)$ for every closed point $ x \in X_0 $, the differential $ d_x f $ induces linear isomorphisms
$$
\pi_i(X,x)  \stackrel{\sim}{\to} \pi_i(Y, f(x))\ , \quad \forall\, i  \ge 0 \ .
$$
\eprop
The proof of Proposition~\ref{Kapp} is based on the next lemma
which is of independent interest (see \cite{CK2}, Sect.~2.2.3).
\blemma
\la{CKlem}
Let $ X = (X_0,\, \O_{X, \bullet})  $ be a smooth DG scheme, and let
$ \hat{\O}_{X,x} := \O_{X, \bullet} \otimes_{\O_{X_0}}  \hat{\O}_{X_0,x} $
denote the complete local DG ring corresponding to a closed $k$-point $ x \in \pi_0(X) $.
Then there is a convergent spectral sequence
\begin{equation}
\la{CKsp}
E^2 = \bL^\bullet  [\H_\bullet(T^*_x X)] \ \Rightarrow\ \H_\bullet(\hat{\O}_{X,x})
\end{equation}
arising from the $ \m_x$-adic filtration on $  \hat{\O}_{X,x} $.
\elemma
Crucial to the proof of Lemma~\ref{CKlem} is the fact that $ \hat{\O}_{X,x} $
coincides with the completion of $ \O_{X,x}$ with respect to the $\m_x$-adic topology.
If $f: X \to Y $ satisfies the conditions $(1)$ and $(2)$ of Proposition~\ref{Kapp},
for any $ x \in \pi_0(X) $, the map $ \hat{f}_x: \hat{\O}_{Y,y} \to \hat{\O}_{X,x}  $
induces a quasi-isomorphism between $E^2$-terms of the spectral sequences \eqref{CKsp}
associated to the local rings $\hat{\O}_{X,x} $ and  $  \hat{\O}_{Y, y} $. Since these local
rings are complete, the Eilenberg-Moore Comparison Theorem implies that $ \hat{f}_x$ is a
quasi-isomorphism. By Krull's Theorem, the map $ f $ itself is then a quasi-isomorphism.

\section{Representation Schemes}
\la{S2}
In this section, we extend the  representation functor \eqref{rep}  to
the category of DG algebras. We show that such an extension defines a representable functor
which is actually a  left Quillen functor in the sense of Lemma~\ref{Qpair}.
A key technical tool is the universal construction of `matrix reduction', which (in the case of ordinary
associative algebras) was introduced and studied in \cite{B} and \cite{C}. The advantage of this
construction is that it produces the representing object for \eqref{rep} in a canonical form as
a result of application of some basic functors on the category of algebras.

\subsection{DG representation functors}
\la{S2.1}
Let $ S \in \DGA_k $ be a DG algebra, and let $ \DGA_S $ denote the category of DG algebras over $S$.
By definition, the objects of $ \DGA_S $ are the DG algebra maps $\,S \to A \,$ in $ \DGA_k $
and the morphisms are given by the commutative triangles
\[
\begin{diagram}[small, tight]
  &       &     S     &        & \\
  &\ldTo  &           &  \rdTo &  \\
A &       & \rTo^{f}  &        & B
\end{diagram}
\]
We will write a map $\, S \to A \,$ as $\, S \bs A \,$, or simply $A$, when we regard it
as an object in $ \DGA_S $. For $ S \in \Alg_k $, we also introduce the category  $ \Alg_S $
of ordinary $S$-algebras ({\it i.e.}, the category of morphisms $ S \to A $ in $\Alg_k$)
and identify it with a full subcategory of $ \DGA_S $ in the natural way.

Let $ (V, d_V) $ be a complex of $k$-vector spaces of finite (total) dimension, and
let $ \END\,V $ denote its graded endomorphism ring with  differential
$\, df = d_V f - (-1)^{i}f \,d_V \,$, where  $ f \in \END(V)_i\,$. Fix on $ V $ a
DG $S$-module structure, or equivalently, a DG representation $\,S \to \END\,V \,$.
This makes $ \END\,V $ a DG algebra over $S$, {\it i.e.} an object of $ \DGA_S $.
Now, given a DG algebra $ A \in \DGA_S $, an $S$-representation of $A$ in $V$ is, by definition,
a morphism $\,A \to \END\,V\,$ in $ \DGA_S $. Such representations form an affine DG scheme
which is defined as the functor on the category of commutative DG algebras:
\begin{equation}
\la{S2E1}
\Rep_V(S \bs A):\ \cDGA_{k} \to \Sets\ ,\quad C \mapsto \Hom_{\DGA_S}(A,\, \END\,V \otimes C)\ .
\end{equation}
Our proof of representability of \eqref{S2E1} is based on the following simple observation.
Denote by $ \DGA_{\End(V)} $ the category of DG algebras over $ \END\,V \,$ and consider the natural functor
\begin{equation}
\la{S2E3}
\G:\, \DGA_{k} \to \DGA_{\End(V)}\ ,\quad B \mapsto \END\,V \otimes B \ ,
\end{equation}
where $ \END\,V \otimes B $ is viewed as an object in $ \DGA_{\End(V)} $ via the
canonical map $\,\END\,V \to \END\,V \otimes B\,$.
\blemma
\la{S2L1}
The functor \eqref{S2E3} is an equivalence of categories.
\elemma
For a detailed proof we refer to  \cite{BKR}, Lemma~2.1. Here we only note that
the inverse functor  to  \eqref{S2E3}  is given by
\begin{equation}
\la{S2E2}
\G^{-1}:\, \DGA_{\End(V)} \to \DGA_k\ ,\quad (\END\,V \to A) \mapsto A^{\,\END(V)}\ ,
\end{equation}
where $ A^{\,\END(V)}\,$ is the (graded) centralizer of the image of $ \END\,V $ in $A\,$.

Next, we introduce the following functors
\begin{eqnarray}
&& \rtv{\,\mbox{--}\,}:\, \DGA_S \to \DGA_k\ ,\quad S \bs A \mapsto
(\END\,V \amalg_S A)^{\END(V)}\ ,\la{S2E9} \\*[1ex]
&& (\,\mbox{--}\,)_V:\, \DGA_S \to \cDGA_k\ ,\quad S \bs A \mapsto (\rtv{S \bs A})_\nn\ , \la{S2E10}
\end{eqnarray}
where $\,\amalg_S \,$ denotes the coproduct in the category $ \DGA_S $ and
$\,(\mbox{--})_\nn:\,\DGA_k \to \cDGA_k \,$ stands for `commutativization', {\it i.e.}
taking the quotient of a DG algebra $R$ by its two-sided commutator ideal:
$\,R_\nn := R/\langle[R, R]\rangle\,$. The following proposition is an easy consequence
of Lemma~\ref{S2L1}.
\bprop
\la{S2P1}
For any $ S\bs A \in \DGA_S $, $ B \in \DGA_k $ and $ C \in \cDGA_k $, there are natural bijections

\vspace{0.8ex}

$(a)$ $\,\Hom_{\DGA_k}(\rtv{S\bs A},\,B) \cong \Hom_{\DGA_S}(A,\,\END\,V \otimes B) \,$,

\vspace{0.8ex}

$(b)$ $\, \Hom_{\cDGA_k}((S\bs A)_V,\,C) \cong \Hom_{\DGA_S}(A,\,\END\,V \otimes C) \,$.

\eprop
\begin{proof}
The tensor functor $\, B \mapsto \END\,V \otimes B \,$ in $(a)$ can be formally written as the
composition
\begin{equation}\la{compf}
\DGA_k \xrightarrow{\G} \DGA_{\End(V)} \xrightarrow{\F} \DGA_S\ ,
\end{equation}
where $\G$ is defined by \eqref{S2E3} and $ \F $ is the restriction functor via
the given DG algebra map $ S \to \END\,V $. Both
$ \F $ and $ \G $ have natural left adjoint functors: the left adjoint of
$ \F $ is obviously the coproduct $\,A \mapsto \END\,V \amalg_S A \,$, while the left
adjoint of $ \G $ is $ \G^{-1} $, since $ \G $ is an equivalence of categories
(Lemma~\ref{S2L1}). Now, by definition, the functor $ \rtv{\,\mbox{--}\,} $ is the
composition of these left adjoint functors and hence the left adjoint to the
composition \eqref{compf}. This proves part $(a)$. Part $(b)$ follows from $(a)$ and
the obvious fact that the commutativization functor $\,(\mbox{--})_\nn:\,\DGA_k \to \cDGA_k \,$
is left adjoint to the inclusion $ \iota:\,\cDGA_k \into \DGA_k $.
\end{proof}

Part $(b)$ of Proposition~\ref{S2P1} can be restated in the following way, which shows that
$ \Rep_V(S \bs A) $ is indeed an affine DG scheme in the sense of Section~\ref{2.8}.
\bthm
\la{S2T1}
For any $ S\bs A \in \DGA_S $, the commutative DG algebra $ (S \bs A)_V  $ represents the functor \eqref{S2E1}.
\ethm

The algebras $\, \rtv{S \bs A} $, $\, (S \bs A)_V \,$ and the isomorphisms of Proposition~\ref{S2P1}
can be described explicitly. To this end, we choose a linear basis $\,\{v_i\}\,$ in $\, V \,$ consisting of homogeneous elements,
and define the elementary endomorphisms $ \{e_{ij}\} $ in $ \END\,V $ by $ e_{ij}(v_k) = \delta_{jk} v_i $.
These endomorphisms are homogeneous, the degree of $ e_{ij} $ being $\, |v_i| - |v_j|\,$,
and satisfy the obvious relations
\begin{equation}
\la{S2E5}
\sum_{i=1}^d e_{ii} = 1 \ , \qquad e_{ij}\,e_{kl} = \delta_{jk}\, e_{il}\ ,
\end{equation}
where $ d := \dim_k V $. Now, for each homogeneous element $ a \in \END\,V \amalg_S A $, we define
its `matrix' elements by
\begin{equation}
\la{S2E6}
a_{ij} := \sum_{k=1}^d\, (-1)^{(|a|+|e_{ji}|) |e_{jk}|}\, e_{ki}\,a\,e_{jk} \ ,\quad i,\,j = 1,\,2,\,\ldots\,,\,d\ .
\end{equation}
A straightforward calculation using \eqref{S2E5} shows that $\,[a_{ij}, e_{kl}] = 0 \,$ for
all $\, i, j, k, l = 1, 2,\ldots, d\,$. Since $ \{e_{ij}\} $ spans $\, \END\,V \,$, this means that
$\, a_{ij} \in  \rtv{S \bs A} $, and in fact, it is easy to see that every homogeneous
element of $ \rtv{S \bs A} $ can be written in the form \eqref{S2E6}. By Lemma~\ref{S2L1},
the map
\begin{equation}
\la{S2E8}
\psi \, :\ \END\,V \amalg_S A \ \to \ \END\,V \otimes \rtv{S \bs A}\ ,
\quad a \mapsto \sum_{i,j = 1}^d \,e_{ij} \otimes a_{ij}
\end{equation}
is a DG algebra isomorphism which is inverse to the canonical (multiplication) map
$$
\END\,V \otimes \rtv{S \bs A} \to \END\,V \amalg_S A \ .
$$
Using \eqref{S2E8}, we can now write the bijection of Proposition~\ref{S2P1}$(a)$:
$$
\Hom_{\DGA_k}(\rtv{S\bs A},\,B) \to \Hom_{\DGA_S}(A,\, \END\,V \otimes B)\ ,\quad
f \mapsto (\id \otimes f) \circ \psi|_A\ ,
$$
where $ \psi|_A $ is the composition of \eqref{S2E8} with the canonical map $ A \to \END\,V \amalg_S A \,$.
As the algebra $ (S \bs A)_V $ is, by definition, the maximal commutative quotient of  $ \rtv{S \bs A}  $, it is also
spanned by the elements \eqref{S2E6} taken modulo the commutator ideal.

\vspace{1ex}

\remark\
For ordinary $k$-algebras, Proposition~\ref{S2P1} and Theorem~\ref{S2T1} were originally proven in \cite{B} (Sect.~7)
and \cite{C} (Sect.~6). In these papers, the functor \eqref{S2E9} was called the `matrix reduction' and a
different notation was used. Our notation $\, \rtv{\mbox{--}} \,$ is borrowed from \cite{LBW}, where
\eqref{S2E9} is used for constructing noncommutative thickenings of classical representation schemes.

\subsection{Deriving the representation functor}
\la{S2.2}
As explained in Section~\ref{3}, the categories $ \DGA_k $ and $ \cDGA_k $ have natural
model structures, with weak equivalences being the quasi-isomorphisms. Furthermore, for a fixed DG algebra $S$,
the category of $ S$-algebras, $\, \DGA_S $, inherits a model structure from $ \DGA_k $ ({\it cf.} \ref{2.1.2}).
Every DG algebra  $\, S \bs A \in \DGA_S \,$ has a cofibrant resolution $\, Q(S \bs A) \sonto S \bs A \,$ in $ \DGA_S $,
which is given by a factorization $\,S \into Q \sonto A \,$ of the morphism $ S \to A $ in $ \DGA_k $.  By
Theorem~\ref{Tloc}, the homotopy  category $ \Ho(\DGA_S) $ is equivalent to the localization of
$ \DGA_S $ at the class of weak equivalences in $ \DGA_S $.
We denote the corresponding localization functor by $\,\gamma:\,\DGA_S \to \Ho(\DGA_S)\,$;
this functor acts as identity on objects while maps each morphism
$ f: S \bs A \to  S \bs B $ to the homotopy class of its cofibrant lifting
$ \tilde{f}: Q(S \bs A) \to Q(S \bs B) $  in $\DGA_S \,$, see \eqref{2.2.4}. The next theorem is one of
the main results of \cite{BKR} (see {\it loc. cit.}, Theorem~2.2).
\bthm
\la{S2T2}
$(a)$ The functors $\,(\,\mbox{--}\,)_V :\,\DGA_S \rightleftarrows \cDGA_k\,: \END\,V \otimes {\mbox{--}}\,$
form a Quillen pair.

$(b)$ $\, (\,\mbox{--}\,)_V $ has a total left derived functor defined by
$$
\L(\,\mbox{--}\,)_V:\,\Ho(\DGA_S) \to \Ho(\cDGA_k) \ ,
\quad S \bs A  \mapsto  Q(S \bs A)_V\ , \quad \gamma f \mapsto \gamma(\tilde f_V)\ .
$$

$(c)$ For any $\, S \bs A \in \DGA_S \,$ and $ B \in \cDGA_k $, there is a canonical isomorphism
\begin{equation*}
\la{S2E15}
\Hom_{\Ho(\cDGA_k)}(\L(S \bs A)_V,\, B) \cong
\Hom_{\Ho(\DGA_S)}(A,\,\END\,V \otimes B)\ .
\end{equation*}
\ethm
\bproof
By Proposition~\ref{S2P1}$(b)$, the functor $\,(\,\mbox{--}\,)_V \,$ is left adjoint to the composition
$$
\cDGA_k \stackrel{\iota}{\into} \DGA_k \xrightarrow{\END\,V \otimes\,-\,} \DGA_S \ ,
$$
which we still denote $\,\END\,V \otimes {\mbox{--}}\ $.
Both the forgetful functor $ \iota $ and the tensoring with
$ \END\,V $ over a field are exact functors on $ \Com_k $; hence, they map fibrations (the surjective morphisms in $ \DGA_k $) to fibrations and also preserve the class of weak
equivalences (the quasi-isomorphisms). It follows that $\,\END\,V \otimes {\mbox{--}}\,$ preserves fibrations as well as acyclic
fibrations. Thus, by Lemma~\ref{Qpair}, $\,(\,\mbox{--}\,)_V :\,\DGA_S \rightleftarrows \cDGA_k\,: \END\,V \otimes {\mbox{--}}\,$ is a Quillen pair. This proves part $(a)$. Part $(b)$ and $(c)$ now follow
directly from Quillen's Adjunction Theorem (see Theorem~\ref{Qthm}).
For part $(c)$, we need only to note that $\,G := \END\,V \otimes {\mbox{--}} \,$ is an exact functor in Quillen's sense,
{\it i.e.} $ \R G = G $, since $ \cDGA_k $ is a fibrant model category.
\eproof

\begin{definition}
\la{S2D1}
By Theorem~\ref{S2T2}, the assignment $\, S\bs A \mapsto Q(S\bs A)_V \,$ defines a functor
$$
\DRep_V:\ \Alg_S \to \Ho(\cDGA_k)
$$
which is independent of the choice of resolution $ Q(S\bs A) \sonto S\bs A $ in $ \DGA_S $. Abusing terminology, we call
$ \DRep_V(S \bs A) $ a relative {\it derived representation scheme} of $ A $. The homology of
$ \DRep_V(S \bs A) $ is a graded commutative algebra, which depends only on $ S \bs A$ and $V$. We write
\begin{equation}
\la{S2E13}
\H_\bullet(S \bs A,\,V) := \H_\bullet[\DRep_V(S\bs A)]
\end{equation}
and refer to \eqref{S2E13} as {\it representation homology} of $S \bs A$ with coefficients in $V$. In the absolute case
when $ S = k $, we simplify the notation writing $ \DRep_V(A) := \DRep_V(k \bs A) $ and
$ \H_\bullet(A,\,V) := \H_\bullet(k \bs A,\,V) $.
\end{definition}

\vspace{1ex}

\noindent
We now make a few remarks related to Theorem~\ref{S2T2}.

\vspace{1ex}

\subsubsection{}
\la{S2.2.2}
For any cofibrant resolutions $\,p: Q(S\bs A) \sonto S\bs A \,$ and $\,p': Q'(S\bs A) \sonto S\bs A \,$
of a given $ S \bs A \in \DGA_S $, there is a quasi-isomorphism $\,f_V:\,Q(S\bs A)_V
\stackrel{\sim}{\to} Q'(S\bs A)_V \,$ in $ \cDGA_k $. Indeed, by \ref{2.2.4}, the identity map on $A$
lifts to a morphism $ f: Q(S\bs A) \stackrel{\sim}{\to} Q'(S\bs A) $ such that $\,p'\,f = p\,$.
This morphism is automatically a weak equivalence in $ \DGA_S$, so $ \gamma f $ is an
isomorphism in $ \Ho(\DGA_S) $. It follows that $ \L(\gamma f)_V $ is an isomorphism in
$ \Ho(\cDGA_k)$. But $ Q(S\bs A) $ and $ Q'(S\bs A) $ are both cofibrant objects, so $ \L(\gamma f)_V =
\gamma(f_V) $ in $ \Ho(\cDGA_k) $. Thus $ f_V $ is a quasi-isomorphism in $\cDGA_k$.

\subsubsection{}
\la{S2.2.3}
The analogue of Theorem~\ref{S2T2} holds for the pair of functors
$\,
\rtv{\,\mbox{--}\,} :\ \DGA_S \rightleftarrows \DGA_k\,: \ \END\,V \otimes {\mbox{--}}
\,$,
which are adjoint to each other by Proposition~\ref{S2P1}$(a)$. Thus, $\, \rtv{\,\mbox{--}\,} \,$ has
the left derived functor
\begin{equation*}
\la{S2E15'}
\L \rtv{\,\mbox{--}\,}:\, \Ho(\DGA_S) \to \Ho(\DGA_k) \ ,\quad \L \rtv{S \bs A} := \rtv{Q(S \bs A)}\ ,
\end{equation*}
which is left adjoint to $\,\END\,V \otimes {\mbox{--}}\,$ on the homotopy category $ \Ho(\DGA_k) $.

\subsubsection{}
\la{S2.2.4}
If $V$ is a complex concentrated in degree $0$, the functors
$ (\,\mbox{--}\,)_V $ and $\,\End\,V \otimes \mbox{--} \,$ restrict to the category of
{\it non-negatively} graded DG algebras and still form the adjoint pair
\begin{equation*}
\la{plus}
(\,\mbox{--}\,)_V :\,\DGA^+_S \rightleftarrows \cDGA^+_k\,: \End\,V \otimes {\mbox{--}}\quad .
\end{equation*}
The categories $ \DGA_S^+ $ and $ \cDGA_k^+ $ have natural model structures (see Theorem~\ref{modax1}), for which
all the above results, including  Theorem~\ref{S2T2}, hold, with proofs being identical to the unbounded case.

\subsubsection{}
\la{S2.2.5}
The representation functor \eqref{rep} naturally extends to the category $ \mathtt{SAlg}_k $ of simplicial
$k$-algebras, and one can also use the model structure on this last category to construct the derived functors of
\eqref{rep}. However, for any commutative ring $ k $,  the model category $ \mathtt{SAlg}_k $ is known to be
is Quillen equivalent to the model category $ \DGA^+_k $ (see \cite{SS2}, Theorem~1.1).  Also, if
$ k $ is a field of characteristic zero (as we always assume in this paper), the corresponding categories of
commutative algebras $ \mathtt{SComm\,Alg}_k $ and $ \cDGA_k^+ $  are Quillen equivalent
(see \cite{Q2}, Remark on p. 223). Thus, at least when $V$ is a complex concentrated in degree $0$,
the derived representation functors $ \DRep_V $ constructed using simplicial and DG resolutions are
naturally equivalent.

\subsection{Basic properties of $ \DRep_V(S \bs A)$}
\la{S2.3}

\subsubsection{}
\la{S2.3.1}
We begin by clarifying how the functor $ \DRep_V $ depends on $V$.
Let $ \DGMod(S) $ be the category of DG modules over $S$, and let $V$ and $W$ be two modules in
$ \DGMod(S) $ each of which has finite dimension over $k$.

\bprop[\cite{BKR}, Proposition~2.3]
\la{S2P3}
If $\,V$ and $W$ are quasi-isomorphic in $\DGMod(S)$, the corresponding derived functors
$ \L (\,\mbox{--}\,)_V $ and $ \L (\,\mbox{--}\,)_W :  \Ho(\DGA_S) \to \Ho(\cDGA_k) $ are naturally equivalent.
\eprop
The proof of this proposition is based on the following lemma, which is probably known to the experts.
\blemma
\la{S2L3}
Let $V$ and $W$ be two bounded DG modules over $S$, and assume that there is a quasi-isomorphism $\,f: V \stackrel{\sim}{\to} W \,$
in $\DGMod(S)$. Then the DG algebras $ \END\,V $ and $ \END\,W $ are weakly equivalent in $ \DGA_S $, {\it i.e.} isomorphic in $ \Ho(\DGA_S) $.
\elemma

As an immediate consequence of Proposition~\ref{S2P3}, we get
\begin{corollary}
\la{S2C1}
If $ V $ and $ W $ are quasi-isomorphic $S$-modules, then
$\,\DRep_V(S \bs A) \cong \DRep_W(S \bs A) \,$ for any algebra
$\, S \bs A\in \Alg_S $. In particular, $\,\H_\bullet(S \bs A,V)
\cong \H_\bullet(S \bs A,W)\,$ as graded algebras.
\end{corollary}

\subsubsection{Base change}
\la{S2.4}
Let $\, R \xrightarrow{g} S \xrightarrow{f} A\,$ be morphisms in $\DGA_k$. Fix a DG representation
$\,\varrho:\, S \to \END\,V $, and let $\,\varrho_R := \varrho \circ g\,$. Using $ \varrho $
and $ \varrho_R $, define the representation functors $\,(S \bs\,\mbox{--}\,)_V:\,\DGA_S \to \cDGA_k\,$,
and $\, (R\bs\,\mbox{--}\,)_V :\,\DGA_R \to \cDGA_k \,$ and consider the corresponding derived functors
$ \L(S\bs\,\mbox{--}\,)_V $ and $ \L(R\bs\,\mbox{--}\,)_V $.

\begin{theorem}[\cite{BKR}, Theorem~2.3]
\la{hopushout}
$(a)$ The commutative diagram
\begin{equation}
\la{drepd11}
\begin{diagram}[small, tight]
(R\bs S)_V              &  \rTo^{\,(R\bs f)_V\,}            & (R\bs A)_V\\
\dTo^{(R\bs \varrho)_V} &                                   & \dTo \\
k                       &  \rInto                           & (S\bs A)_V
\end{diagram}
\end{equation}
is a cocartesian square in $ \cDGA_k $.

$(b)$  There is a homotopy commutative diagram
\begin{equation}
\la{drepd2}
\begin{diagram}[small, tight]
\L(R\bs S)_V           &  \rTo^{\,\L(f)_V\,}        & \L(R\bs A)_V\\
\dTo^{\L(\varrho)_V} &                    & \dTo \\
k                &  \rInto            & \L(S\bs A)_V
\end{diagram}
\end{equation}
which is a cocartesian square in $ \Ho(\cDGA_k) $.
\end{theorem}

Let us state the main corollary of Theorem~\ref{hopushout}, which may be viewed
as an alternative definition of $ \DRep_V(S\bs A) $. It shows that our
construction of relative $ \DRep_V $ is a `correct' one from homotopical point of view
({\it cf.} \cite{Q2}, Part~I, 2.8).
\begin{corollary}
\la{hlimr}
For any $ S \bs A \in \Alg_S $, $\,\DRep_V(S\bs A) \,$ is a homotopy cofibre of the natural map
$ \DRep_V(S) \to \DRep_V(A) $, {\it i.e.}
\begin{equation*}
\begin{diagram}[small, tight]
\DRep_V(S)          &  \rTo        &  \DRep_V(A)\\
\dTo &                    & \dTo \\
k                &  \rInto            & \DRep_V(S\bs A)
\end{diagram}
\end{equation*}
is a pushout diagram in $ \Ho(\cDGA_k^+) $.
\end{corollary}

The above result suggests that the homology of $ \DRep_V(S\bs A) $ should be related to
the homology of $ \DRep_V(S) $ and $ \DRep_V(A) $ through a standard spectral sequence
associated to a fibration. To simplify matters we will assume that $V$ is a $0$-complex
and work in the category $ \DGA_k^+$ of non-negatively graded DG algebras
({\it cf.} Remark~\ref{S2.2.4}).
\begin{corollary}
\la{EMsps}
Given $\, R \xrightarrow{} S \xrightarrow{} A\,$ in $\DGA^+_k$ and a representation $ S \to \End(V) $,
there is an Eilenberg-Moore spectral sequence with
$$
E^2_{\ast,\,\ast} = \Tor_{\ast,\,\ast}^{\H_\bullet(R\bs S, V)}(k,\,\H_\bullet(R\bs A, V))
$$
converging to $ \H_{\bullet}(S\bs A, V) $.
\end{corollary}

\subsubsection{}
\la{S2.3.3}
The next result shows that $ \DRep_V(S \bs A) $ is indeed the `higher' derived functor of
the classical representation scheme $ \Rep_V(S\bs A) $ in the sense of homological algebra.
\bthm
\la{S2T4}
Let $ S \in \Alg_k $ and $ V $ concentrated in degree $0$. Then, for any $ S \bs A \in \Alg_S $,
$$
\H_0(S\bs A,V) \cong (S\bs A)_V
$$
where $(S\bs A)_V$ is a commutative algebra representing $ \Rep_V(S \bs A) $.
\ethm

Theorem~\ref{S2T4} implies, in particular, that $ \DRep_V(S\bs A) $ is trivial whenever $ \Rep_V(S \bs A) $ is trivial.
Indeed, if $ \Rep_V(S \bs A) $ is empty, then $ (S\bs A)_V = 0 $. By Theorem~\ref{S2T4}, this means
that $\,1 = 0\,$ in $ \H_\bullet[\DRep_V(S\bs A)] $, hence $ \H_\bullet[\DRep_V(S\bs A)] $ is the zero
algebra. This, in turn, means that $ \DRep_V(S\bs A) $ is acyclic and hence $ \DRep_V(S \bs A) = 0 $
in $\Ho(\cDGA^+_k) $ as well.

\vspace{1ex}

\noindent
{\bf Example.}\ Take the first Weyl algebra $\, A_1(k) := k \langle x, y \rangle/(xy-yx-1) \,$.  Since $k$ has
characteristic zero, $A_1(k)$ has no (nonzero) finite-dimensional modules. So $ \Rep_V[A_1(k)] $ is empty and
$ \DRep_V[A_1(k)] = 0 $ for all $ V \not= 0 $. Note that, even if we allow $ V $ to be a chain complex,
we still get $\, \DRep_V[A_1(k)] = 0 \,$, by Proposition~\ref{S2P3}.

\vspace{1ex}

\subsection{The invariant subfunctor}
\la{S2.3.4}
We will keep the assumption that $V$ is a $0$-complex and assume, in addition, that $S=k$.
Let $ \GL(V) \subset \End(V) $ denote, as usual, the group of invertible endomorphisms of $V$.
Consider the right action of $ \GL(V) $ on $ \End(V) $ by conjugation, $\,\alpha \mapsto g^{-1} \alpha g \,$,
and extend it naturally to the functor $\,\End\,V \otimes \mbox{--} \, : \,\cDGA_k \to \DGA_k \,$.
Through the adjunction of Proposition~\ref{S2P1}$(b)$, this right action induces a (left) action on the representation functor $\, (\,\mbox{--}\,)_V: \DGA_k \to \cDGA_k \,$, so we
can define its invariant subfunctor
\begin{equation}
\la{S2E16}
(\,\mbox{--}\,)_V^{\GL}\,:\ \DGA_k \to \cDGA_k\ , \quad A \mapsto A_V^{\GL(V)}\ .
\end{equation}
Unlike  $ (\,\mbox{--}\,)_V\,$, the functor \eqref{S2E16} does not seem to
have a right adjoint, so it is not a left Quillen functor. The Quillen
Adjunction Theorem does not apply in this case. Still, using
Brown's Lemma~\ref{BrL}, we prove
\bthm[\cite{BKR}, Theorem 2.6]
\la{S2P4}
$(a)$ The functor \eqref{S2E16} has a total left derived functor
$$
\L(\,\mbox{--}\,)_V^{\GL}:\,\Ho(\DGA_k) \to \Ho(\cDGA_k)\ .
$$

$(b)$ For any $ A \in \DGA_k $, there is a natural isomorphism of graded algebras
$$
\H_\bullet[\L(A)_V^{\GL}] \cong \H_\bullet(A, V)^{\GL(V)}\ .
$$
\ethm

\noindent
If $\,A \in \Alg_k \,$, abusing notation we will sometimes write
$ \DRep_V(A)^{\GL} $ instead of $ \L(A)_V^{\GL} $.

\subsection{The Ciocan-Fontanine-Kapranov construction}
\la{S2.3.5} For an ordinary $k$-algebra $A$ and a $k$-vector space $V$, Ciocan-Fontanine and
Kapranov introduced a derived affine scheme, $ \RAct(A, V) $, which they called the {\it derived space of actions of $A$}
(see \cite{CK}, Sect.~3.3). Although the construction of $ \RAct(A, V) $ is quite different from our construction
of $ \DRep_V(A) $, Proposition~3.5.2 of \cite{CK} shows that, for a certain {\it specific} resolution of $A$,
the DG algebra $ k[\RAct(A, V)] $ satisfies the adjunction of Proposition~\ref{S2P1}$(b)$.
Since $ k[\RAct(A, V)] $ and $ \DRep_V(A) $ are independent of the choice of resolution, we conclude
\begin{theorem}
\la{S2T44}
If $\,A \in \Alg_k \,$ and $V$ is a $0$-complex,
then $\, k[\RAct(A,\,V)] \cong \DRep_V(A) \,$ in $ \Ho(\cDGA^+_k) $.
\end{theorem}
The fact that $ k[\RAct(A, V)] $ is independent of resolutions was proved in
\cite{CK} by a fairly involved calculation using spectral sequences. Strictly speaking, this
calculation does not show that $ \RAct(\,\mbox{--}\,, V) $ is a Quillen derived functor.
In combination with Theorem~\ref{S2T44}, our main Theorem~\ref{S2T2} can thus be viewed as a strengthening of \cite{CK} --- it implies that $ \RAct(A, V) $ is indeed a (right) Quillen derived functor on the category of DG schemes.

\subsection{Explicit presentation}
\la{S2.5}
Let $ A \in \Alg_k $. Given an semi-free resolution $ R \sonto A $ in $\DGA^+_k$, the DG algebra $ R_V $ can be
described explicitly. To this end, we extend a construction of Le Bruyn and van de Weyer (see \cite{LBW}, Theorem~4.1)
to the case of DG algebras. Assume, for simplicity, that $ V = k^d $. Let $\{x^\alpha\}_{\alpha \in I}$ be a set of
homogeneous generators of a semi-free DG algebra $R$, and let $ d_R $ be its differential. Consider a free graded algebra
$\tilde{R}$ on generators $\,\{x^{\alpha}_{ij}\,:\, 1\leq i,j\leq d\, ,\, \alpha \in I\}\,$, where
$\,|x^{\alpha}_{ij}|=|x^{\alpha}|\,$ for all $ i,j $. Forming matrices $ X^\alpha :=
\| x^{\alpha}_{ij}\| $ from these generators, we define the algebra map
$$
\pi:\, R \to \M_d(\tilde{R})\ , \quad  x^\alpha \mapsto X^\alpha \ ,
$$
where $ \M_d(\tilde{R}) $ denotes the ring of $ (d \times d)$-matrices with entries in $\tilde{R}$.
Then, letting $\, \tilde{d}(x^{\alpha}_{ij}) := \| \pi(d x^{\alpha}) \|_{ij}\,$, we define a differential
$ \tilde{d} $ on generators of $ \tilde{R} $ and extend it to the whole of $ \tilde{R}$ by linearity
and the Leibniz rule. This makes $ \tilde{R} $ a DG algebra. The commutativization of
$\tilde{R} $ is a free (graded) commutative algebra generated by (the images of) $\,x^\alpha_{ij}\,$ and the differential
on $ \tilde{R}_\nn $ is induced by $ \tilde{d}$.
\bthm[\cite{BKR}, Theorem~2.8]
\la{comp}
There is an isomorphism of DG algebras $\,  \rtv{R} \cong \tilde{R} \,$. Consequently, $\,R_V \cong \tilde{R}_\nn\,$.
\ethm
Using Theorem~\ref{comp}, one can construct a finite presentation for $ R_V $ (and hence an explicit
model for $ \DRep_V(A) $) whenever a finite semi-free resolution $ R \to A $ is available. We will apply
this theorem in Section~\ref{S6}, where we study representation homology  for three classes of algebras:
noncommutative complete intersections, Koszul and Calabi-Yau algebras, which have canonical
`small' resolutions.
\section{Cyclic Homology and Higher Trace Maps}
\la{S3}
In this section, we construct canonical trace maps
$\,\Tr_V(S \bs A)_n:\,\HC_{n-1}(S \bs A) \to \H_n(S \bs A, V)\,$ relating the cyclic homology of an $S$-algebra
$ A \in \Alg_S $ to its representation homology. In the case when $ S = k $ and $V$ is concentrated in degree
$ 0 $, these maps can be viewed as {\it derived characters} of finite-dimensional representations of $A$.

\subsection{Relative cyclic homology}
\la{S4.1}
We begin by recalling the Feigin-Tsygan construction of cyclic homology as a non-abelian derived functor on the category
of algebras (see \cite{FT, FT1}). To the best of our knowledge, this construction does not appear in standard textbooks
on cyclic homology (like, e.g., \cite{L} or \cite{W}).  One reason for this is perhaps that while the idea of Feigin and
Tsygan is very simple and natural, the proofs in \cite{FT, FT1} are obtained by means of spectral
sequences and are fairly indirect.   In  \cite{BKR}, we develop a more conceptual (and in fact,
slightly more general) approach and give proofs using simple model-categorical arguments. What follows is a
brief summary of this approach: for details, we refer to \cite{BKR}, Section~3.

If $ A $ is a DG algebra, we write $\, A_\natural := A/[A,A] \,$, where $ [A,A] $ is the commutator
subspace of $ A $. The assignment $\, A \mapsto A_\natural \,$ is obviously a functor from $ \DGA_k $ to
the category of complexes
$ \Com(k)$: thus, a morphism of DG algebras $ f: S \to A $ induces a morphism of complexes
$\,f_\natural: S_\natural \to A_\natural \,$.  Fixing $ S \in \DGA_k $, we now define the functor
\begin{equation}
\la{cycd}
\FT:\ \DGA_S \rar \Com(k)\ ,\quad (S \xrightarrow{f} A) \mapsto \cn(f_\natural)\ ,
\end{equation}
where `$ \cn $' refers to the mapping c\^{o}ne in $ \Com(k) $.

The category $ \Com(k) $ has a natural model structure with
quasi-isomorphisms being the weak equivalence and the  epimorphisms being the fibrations.
The corresponding homotopy category $ \Ho(\Com(k)) $ is isomorphic to the (unbounded)
derived category $ \D(k) := \D(\Com\,k) $ ({\it cf.} Theorem~\ref{Tloc}).
\bthm
\la{FTT}
The functor \eqref{cycd} has a total left derived functor $\,\L\FT:\,\Ho(\DGA_S) \to \D(k)\,$
given by
$$
\L\FT(S \bs A) = \cn(S_\n \to Q(S\bs A)_\n)\ ,
$$
where $\, S \to Q(S\bs A)\,$ is a cofibrant resolution of $ S \to A $ in $ \DGA_S $.
\ethm

Theorem~\ref{FTT} implies that the homology of $\, \L \FT(S \bs A)\,$ depends only on
the morphism $ S \to A $. Thus, we may give the following

\vspace{1ex}

\begin{definition}
\la{S3D1}
The (relative) {\it cyclic homology} of $\, S \bs A \in \DGA_S \,$ is defined by
\begin{equation}
\la{hcyc}
\HC_{n-1}(S \bs A) := \H_{n}[\L \FT(S \bs A)] = \H_{n} [\cn(S_\n \to Q(S\bs A)_\n)]\ .
\end{equation}
\end{definition}

\vspace{1ex}

If $ S \to A $ is a map of ordinary algebras and $ S \stackrel{i}{\into} QA \sonto A $ is a
cofibrant resolution of $ S \to A $ such that $ i $ is a semi-free extension in $ \DGA^+_S $,
then the induced map $\, i_\natural:\,S_\n \into (QA)_\n \,$ is injective, and
\begin{equation}
\la{cyc1}
\FT(S \bs QA) = \cn(i_\n) \cong (QA)_\natural/S_\natural \cong QA/([QA, QA] + i(S)) \ .
\end{equation}
In this special form, the functor $ \FT $ was originally introduced by Feigin and Tsygan in \cite{FT} (see also \cite{FT1});
they proved that the homology groups \eqref{hcyc} are independent of the choice of resolution using spectral sequences.
Theorem~\ref{FTT} is not explicitly stated in \cite{FT, FT1}, although it is implicit in several calculations.
We  emphasize that, in the case when $ S $ and $A$ are ordinary algebras, our definition of relative cyclic homology
\eqref{hcyc} agrees with the Feigin-Tsygan one.

\vspace{1ex}

One of the key properties of relative cyclic homology is the existence of a long exact sequence
for composition of algebra maps. Precisely,
\bthm[\cite{FT}, Theorem~2]
\la{S4C1}
Given DG algebra maps $ R \xrightarrow{} S \xrightarrow{} A $,
there is an exact sequence in cyclic homology
\begin{equation}
\la{ex2}
\ldots \to \HC_n(R \bs S) \to \HC_n(R \bs A) \to \HC_n(S \bs A) \to \HC_{n-1}(R \bs S) \to \ldots
\end{equation}
\ethm
In fact, the long exact sequence  \eqref{ex2} arises from the distinguished triangle in $ \D(k) \,$:
\begin{equation}
\la{isct}
\L\FT(R \bs S) \to \L\FT(R \bs A) \to \L\FT(S \bs A) \to
\L\FT(R \bs S)[1]\ ,
\end{equation}
the construction of \eqref{isct} is given in \cite{BKR}, Theorem~3.3.

\vspace{1ex}

If $A$ is an ordinary algebra over a field of characteristic zero, its {\it cyclic homology} $ \HC_\bullet(A) $
is usually defined  as the homology of the cyclic complex $ \CC(A) $ ({\it cf.} \cite{L}, Sect.~2.1.4):
\begin{equation}
\la{ccc}
\CC_n(A) := A^{\otimes (n+1)}/\im(\id-t_n)\ , \quad b_n:\,\CC_n(A) \to \CC_{n-1} (A)\ ,
\end{equation}
where $ b_n $ is induced by the standard Hochschild differential and
$ t_n $ is the cyclic operator defining an action of $ \Z/(n+1)\Z $ on $ A^{\otimes (n+1)} $:
\begin{equation}
\la{cop}
t_n:\, A^{\otimes (n+1)} \to A^{\otimes (n+1)}\ , \quad (a_0,\, a_1, \,\ldots \,, \,a_n) \mapsto
(-1)^{n}(a_n,\, a_0, \,\ldots \,, \,a_{n-1})\ .
\end{equation}
%
%$$
%b_n(a_0,\, a_1, \,\ldots \,, \,a_n) = \sum_{i=0}^{n-1}\,(-1)^i\, (a_0,\,\ldots \,,\,a_i a_{i+1}\,,\,\ldots\,,\,a_n)
%+ (-1)^n (a_n a_0,\,\ldots \,, \,a_{n-1})\ .
%$$
%
The complex $ \CC(A) $ contains the canonical subcomplex $ \CC(k) $; the homology of the corresponding
quotient complex $\, \rHC_\bullet(A) := \H_\bullet[\CC(A)/\CC(k)]\,$
is called the {\it reduced cyclic homology} of $A$. Both $ \HC(A) $
and $ \rHC(A) $ are special cases of relative cyclic homology in the sense of Definition~\eqref{hcyc}.
Precisely, we have the following result (due to Feigin and Tsygan \cite{FT}).
\bprop
\la{S4P1}
For any $k$-algebra $A$, there are canonical isomorphisms

\vspace{1ex}

$(a)$\ $\,\HC_n(A) \cong \HC_n(A \bs 0) \,$ for all $\,n\ge 0\,$,

\vspace{1ex}

$(b)$\ $\,\rHC_n(A) \cong \HC_{n-1}(k \bs A) \,$ for all $\,n\ge 1\,$.

\eprop
\bproof
$(a)$\ For any (unital) algebra $A$, the DG algebra $  A\langle x \rangle := A\, \amalg \,k \langle x \rangle$ coincides with
the bar construction of $A$ and hence is acyclic. The canonical morphism $\, A \into A\langle x \rangle \,$ provides then
a cofibrant resolution of $ A \to 0 $ in $ \DGA_A $. In this case, we can identify
$$
\L\FT(A\bs 0) \cong \cn(A_\n \to A\langle x \rangle_\n) \cong
A\langle x \rangle/(A + [A\langle x \rangle,\,A\langle x \rangle])
\cong \CC(A)[1]\ ,
$$
where the last isomorphism (in degree $ n>0$) is given by
$\,
a_1 \,x\, a_2\, x\, \ldots\, a_n\, x \leftrightarrow a_1 \otimes a_2 \otimes \ldots \otimes a_n\,$.
On the level of homology, this induces isomorphisms
$ \HC_{n-1}(A \bs 0) \cong \H_n(\CC(A)[1]) = \HC_{n-1}(A) $.

$(b)$ With above identification, the triangle \eqref{isct} associated to the canonical maps $\,k \to A \to 0 \,$ yields
$$
\L\FT(k\bs A) \cong
\cn(\L\FT(k\bs 0) \to \L\FT(A \bs 0))[-1] \cong
\cn[\CC(k) \to \CC(A)]\ .
$$
Whence $ \HC_{n-1}(k \bs A) \cong \rHC_n(A) $ for all $ n \ge 1 $.
\eproof

As a consequence of Theorem~\ref{S4C1} and Proposition~\ref{S4P1}$(a)$, we get the fundamental exact sequence associated
to an algebra map $\, S \to A \,$:
\begin{equation*}
\la{ex1}
\ldots \to \HC_n(S\bs A) \to \HC_n(S) \to \HC_n(A) \to \HC_{n-1}(S\bs A) \to \ldots \to \HC_0(S) \to \HC_0(A) \to 0\ .
\end{equation*}
In particular, if we take $ S = k $ and use the isomorphism of Proposition~\ref{S4P1}$(b)$, then \eqref{ex1} becomes
\begin{equation}
\la{rseq}
\ldots \to \HC_n(k) \to \HC_n(A) \to \rHC_n(A) \to \HC_{n-1}(k) \to \ldots \to \HC_0(A) \to \rHC_0(A) \to 0\ .
\end{equation}

\vspace{1ex}

\remark\ The isomorphism of Proposition~\ref{S4P1}$(a)$ justifies the shift of indexing in our definition \eqref{hcyc} of
relative cyclic homology. In \cite{FT, FT1}, cyclic homology
is referred to as an {\it  additive $K$-theory}, and a different notation is used. The relation between
the Feigin-Tsygan notation and our notation is $\,  K_{n}^+(A,S) = \HC_{n-1}(S \bs A)$ for all $\,n \ge 1\,$.

\vspace{1ex}

\subsection{Trace maps}
\la{S4.1e}
Let $V$ be a complex of $k$-vector spaces of total dimension $d$. The natural map
$\, k \into \END(V) \onto \END(V)_\n \,$ is an isomorphism of complexes, which we can use to
identify $\, \END(V)_\n = k \,$. This defines a canonical (super) trace map
$\,\Tr_V:\,\END\,V \to k \,$ on the DG algebra $ \END\,V $. Explicitly, $ \Tr_V $
is given by
$$
\Tr_V(f) = \sum_{i=1}^d (-1)^{|v_i|} f_{ii}\ ,
$$
where $ \{v_i\} $ is a homogeneous basis in $V$ and $ \| f_{ij} \| $
is the matrix representing $\,f \in \END\,V \,$ in this basis.

Now, fix $\, S \in \DGA_k \,$ and a DG algebra map
$\,\varrho:\,S \to \END\,V \,$ making $V$ a DG module over $S$.
For an $S$-algebra $ A \in \DGA_S $, consider the (relative) DG representation
scheme $ \Rep_V(S \bs A) $, and let $\,\pi_V:\,A \to \END\,V \otimes (S \bs A)_V \,$
denote the universal representation of $A$ corresponding to the identity map
in the adjunction of Proposition~\ref{S2P1}$(b)$. Consider the morphism of complexes
\begin{equation}
\la{E1S4}
A \xrightarrow{\pi_V} \END\,V \otimes (S \bs A)_V \xrightarrow{\Tr_V \otimes \id } (S \bs A)_V\ .
\end{equation}
Since $ \pi_V $ is a map of $S$-algebras, and the $S$-algebra structure on $ \END\,V \otimes (S \bs A)_V $
is of the form $ \varrho \otimes \id $, \eqref{E1S4} induces a map
$\, \Tr_V \circ \pi_V: \, A_\n \to (S \bs A)_V $, which fits in the commutative diagram
\begin{equation}
\la{D1S4}
\begin{diagram}[tight, small]
  S_\n                  &  \rTo^{}  & A_\n  \\
\dTo^{\Tr_V\circ \varrho} &           & \dTo_{\Tr_V \circ \pi_V} \\
  k                     &  \rTo^{}  & (S \bs A)_V
\end{diagram}
\end{equation}
This, in turn, induces a morphism of complexes
\begin{equation}
\la{E2S4}
\cn(S_\n \to A_\n) \to \overline{(S \bs A)}_V\ ,
\end{equation}
where we write $\, \bar{R} = R/k\cdot 1_R \,$ for a unital DG algebra $R$.
The family of morphisms \eqref{E2S4} defines a natural transformation of functors
from $ \DGA_S $ to $ \Com(k)\,$:
\begin{equation}
\la{E3S4}
\Tr_V:\, \FT \to \overline{(\mbox{---})}_V\ .
\end{equation}
The next lemma is a formal consequence of Theorem~\ref{S2T2} and Theorem~\ref{FTT}.
\blemma
\la{L1S4}
$ \Tr_V $ induces a natural transformation
$\, \L \FT \to \overline{\L(\mbox{---})}_V\,$ of functors $ \,\Ho(\DGA_S) \to \D(k)$.
\elemma
For any (non-acyclic) unital DG algebra $ R $, we have
\begin{equation}
\la{hvan}
\H_{n}(\bar{R}) \cong
\left\{
\begin{array}
{lcl}
\overline{\H_0(R)} & , & n = 0 \\*[1ex]
\H_{n}(R)           & , & n \not= 0
\end{array}
\right.
\end{equation}
This is immediate from the long homology sequence arising from
$ 0 \to k \to  R \to \bar{R} \to 0 $.
Hence, if $ A \in \Alg_S $ is an ordinary algebra, applying the natural transformation of
Lemma~\ref{L1S4} to $ S \bs A $ and using \eqref{hvan}, we can define
\begin{equation}
\la{char}
\Tr_V(S \bs A)_n:\ \HC_{n-1}(S\bs A) \to \H_{n}(S \bs A, V)\ ,\ n \ge 1\ .
\end{equation}
Assembled together, these trace maps define a homomorphism of
graded commutative algebras
\begin{equation}
\la{lchar}
\bL \Tr_V(S \bs A)_\bullet :\ \bL(\HC(S\bs A)[1]) \to \H_{\bullet}(S \bs A, V)
\ ,
\end{equation}
where $ \bL $ denotes the graded symmetric algebra of a graded $k$-vector space $W$.

We examine the trace maps \eqref{char} and \eqref{lchar} in the special case
when $S = k$ and $V$ is a single vector space concentrated in degree $0$. In this case, by
Proposition~\ref{S4P1}, the maps \eqref{char} relate
the reduced cyclic homology of $A$ to the (absolute) representation
homology:
\begin{equation}
\la{char1}
\Tr_V(A)_n:\ \rHC_{n}(A) \to \H_{n}(A, V)\ ,\ n \ge 1\ .
\end{equation}
Now, for each $n$, there is a natural map $ \HC_n(A) \to \rHC_n(A) $ induced by the projection of complexes
$ \CC(A) \onto \overline{\CC}(A) $, {\it cf.} \eqref{rseq}. Combining this map with \eqref{char1}, we get
\begin{equation}
\la{char2}
\Tr_V(A)_n:\ \HC_{n}(A) \to \H_{n}(A, V)\ , \ n\ge 0\ .
\end{equation}
Notice that \eqref{char2} is defined for all $\, n $, including $ n = 0 $. In the latter case,
$\, \H_0(A, V) \cong A_V \,$ (by Theorem~\ref{S2T4}), and $\, \Tr_V(A)_0:\, A_\n \to A_V\,$ is
the usual trace induced by $\,A \xrightarrow{\pi_V} A_V \otimes \End_k V \xrightarrow{\id \otimes \Tr} A_V\,$.

The linear maps \eqref{char2} define an algebra homomorphism
\begin{equation}
\la{lchar1}
\bL \Tr_V(A)_\bullet :\ \bL[\HC(A)] \to \H_{\bullet}(A, V)\ .
\end{equation}
Since, for $n\ge 1$, \eqref{char2} factor through reduced cyclic homology, \eqref{lchar1} induces
\begin{equation}
\la{trhomr}
\overline{\Tr}_V(A)_\bullet :\ \Sym(\HC_0(A)) \otimes \bL(\rHC_{\ge 1}(A)) \to \H_\bullet(A,\,V)\ ,
\end{equation}
where $ \rHC_{\ge 1}(A) := \bigoplus_{n \ge 1} \rHC_{n}(A) $.
\bprop
The image of the maps \eqref{lchar1} and \eqref{trhomr} is contained in $ \H_\bullet(A, V)^{\GL(V)}$.
\eprop

Our next goal is to construct an explicit morphism of complexes $\, T:\,\CC(A) \to R_V \,$ that induces
the trace maps \eqref{char2}.  Recall that if $ R \in \DGA_k $ is a DG algebra, its (reduced)
{\it bar construction} $ \Ba(R) $ is a (noncounital) DG coalgebra, which is a universal model
for twisting cochains with values in $R$ (see \cite{HMS}, Chap.~II). Explicitly, $ \Ba(R) $ can be
identified with  the tensor coalgebra
$\, \Ta(R[1]) = \bigoplus_{n \ge 1} R[1]^{\otimes n} \,$, the universal twisting cochain being
the canonical map $\, \hat{\theta}: \, \Ba(R) \to R $ of degree $-1$.

Now, let $ \pi: R \sonto A $ be a semi-free resolution of an algebra $A$ in $ \DGA_k^+$.
By functoriality of the bar construction, the map $ \pi $ extends to a surjective quasi-isomorphism
of DG coalgebras $ \Ba(R) \sonto \Ba(A)$, which we still denote by $ \pi $. This quasi-isomorphism has a
section $ f: \, \Ba(A) \into \Ba(R) $ in the category of DG coalgebras, that is uniquely determined
by the twisting cochain $ \theta_\pi := \hat{\theta}\,f:\, \Ba(A) \to R $.
The components $\,f_n:\,A^{\otimes n} \to R_{n-1}\,$, $\,n \ge 1\,$, of $ \theta_\pi $ satisfy the
Maurer-Cartan equations
\begin{eqnarray}
&& \pi\, f_1 = \id_A \la{rr1} \\
&& d_R f_2 =  f_1 m_A - m_R(f_1 \otimes f_1) \la{rr2} \\
%&& d_R f_3 =  f_2(m_A \otimes \id_A - \id_A \otimes m_A) + m_R(f_2 \otimes f_1 - f_1 \otimes f_2) \la{rr3} \\
&& d_R f_n =  \sum_{i=1}^{n-1} (-1)^{i-1} f_{n-1}(\id_A^{\otimes(i-1)} \otimes m_A \otimes \id_A^{\otimes (n-1-i)})
+ \sum_{i=1}^{n-1} (-1)^{i} m_R(f_i \otimes f_{n-i})\ , \quad n\ge 2\ . \la{rrn}
\end{eqnarray}
where $ m_A $ and $ m_R $ denote the multiplication maps of $A$ and $R$, respectively.
Giving the maps $ \,f_n:\,A^{\otimes n} \to R_{n-1}\,$ is equivalent to
giving a quasi-isomorphism of $A_\infty$-algebras $ f: A \to R $, which induces the inverse of $ \pi $ on
the level of homology. The existence of such a quasi-isomorphism is a well-known result in the theory of
$A_\infty$-algebras (see \cite{K}, Theorem~3.3). Since $ \pi: R \to A $ is a homomorphism of {\it unital}
algebras, we may assume that $ f $  is a (strictly) unital homomorphism of  $A_\infty$-algebras: this means
that, in addition to  \eqref{rr1}-\eqref{rrn}, we have the relations ({\it cf.} \cite{K}, Sect.~3.3)
\begin{equation}
\la{rru}
f_1(1) = 1 \quad ,\qquad f_n(a_1, a_2, \ldots, a_n) = 0 \ , \quad n\ge 2\ ,
\end{equation}
whenever one of the $ a_i$'s equals $1$.

To state the main result of this section, we fix a $k$-vector space $ V $ of (finite) dimension $d$
and, for the given semi-free resolution
$\, R \to A \,$, consider the DG algebra $ R_V = (\!\sqrt[V]{R})_\nn $. As explained in Remark following
Proposition~\ref{S2P1}, the elements of $ R_V $ can be written in the `matrix' form as the images of
$\, a_{ij} \in \sqrt[V]{R} \, $, see \eqref{S2E6}, under the commutativization map
$ \sqrt[V]{R} \onto R_V $.  With this notation, we have

\bthm[\cite{BKR}, Theorem~4.2]
\la{mcor}
The trace maps \eqref{char2} are induced by the morphism
of complexes $ T:\,\CC(A) \to R_V $, whose $n$-th graded component $\,
T_n:\,A^{\otimes (n+1)}/\im(1-t_n) \to (R_V)_n \,$ is given by
\begin{equation}
\la{trf}
T_{n}(a_1, a_2, \ldots, a_{n+1}) = \sum_{i = 1}^{d}\, \sum_{k \in \Z_{n+1}}
(-1)^{nk}  f_{n+1}(a_{1+k}, a_{2+k}, \ldots, a_{n+1+k})_{ii}\ ,
\end{equation}
where $ (f_1,\, f_2,\, \ldots) $ are defined by the relations \eqref{rr1}--\eqref{rrn} and \eqref{rru}.
\ethm
For $n=0$, it is easy to see that \eqref{trf} induces
$$
\Tr_V(A)_0:\,A \to \H_0(A, V) = A_V\ ,\quad a \mapsto \sum_{i = 1}^{d} \, a_{ii}\ ,
$$
which is the usual trace map on $ \Rep_V(A) $.

We can also write an explicit formula for the first trace  $\,\Tr_V(A)_1: \HC_1(A) \to \H_1(A,V) $. For this, we fix a
section $\,f_1: A \to R_0 \,$ satisfying \eqref{rr1}, and let $ \omega:\,A \otimes A \to R_0 $ denote its `curvature':
$$
\omega(a,b) := f_1(a b) - f_1(a) f_1(b)\ ,\quad a,b \in A \ .
$$
Notice that, by \eqref{rr1}, $\,\im\,\omega \subseteq \Ker\,\pi\,$. On the other hand,
$\,\Ker\,\pi = dR_1 \cong R_1/dR_2\,$, since $ R $ is acyclic in positive degrees. Thus,
identifying $\,\Ker\,\pi = R_1/dR_2\,$ via the differential on $R$, we get a map
$\, \tilde{\omega}:\,A \otimes A \to R_1/dR_2 \,$ such that
$\,d\tilde{\omega} = \omega \,$. Using this  map, we
define
\begin{equation}
\la{chern}
\ch_2:\, \CC_1(A) \to R_1/dR_2\ ,\quad
(a,b) \mapsto [\tilde{\omega}(a,b) - \tilde{\omega}(b,a)]\ \mbox{mod}\, dR_2\ .
\end{equation}
Since $\,\tilde{\omega} \equiv f_2\,(\mbox{mod}\, dR_2)\,$, {\it cf.} \eqref{rr2},
it follows from \eqref{trf} that $\,\Tr_V(A)_1 \,$ is induced by the map
\begin{equation}
\la{chern1}
 \Tr_V(A)_1:\, (a,b) \mapsto \sum_{i=1}^d\, \ch_2(a,b)_{ii} \ .
\end{equation}

\vspace{1ex}

\begin{remark}
The notation `$ \ch_2 $' for \eqref{chern} is justified by
the fact that this map coincides with the second Chern character in Quillen's Chern-Weil theory
of algebra cochains (see \cite{Q3}). It would be interesting to see whether the higher traces
$ \Tr_V(A)_n $ can be expressed in terms of higher Quillen-Chern characters.
\end{remark}

\vspace{1ex}

\subsection{Relation to Lie algebra homology}
\la{S4.5}
There is a close relation between representation homology and the
homology of matrix Lie algebras. To describe this relation we first recall
a celebrated result of Loday-Quillen \cite{LQ} and Tsygan \cite{T} which
was historically at the  origin of cyclic homology theory.

For a fixed $k$-algebra $A$ and finite-dimensional vector space $ W $,
let $ \H_\bullet(\gl_W(A); k) $ denote the homology of the Lie algebra $
\gl_W(A) := \mbox{Lie}(\End\,W \otimes A) $.
The Loday-Quillen-Tsygan Theorem  states that there are natural maps
\begin{equation}
\la{liehom}
\H_{n+1}(\gl_W(A); k) \to \HC_{n}(A) \ ,\quad \forall\,n \ge 0\ ,
\end{equation}
which, in the limit $\,\gl_W(A) \to \gl_\infty(A)\,$, induce an isomorphism of graded Hopf algebras
\begin{equation}
\la{liehom5}
\H_\bullet(\gl_{\infty}(A); k) \stackrel{\sim}{\to} \bL(\HC(A)[1])\ .
\end{equation}
Explicitly, the maps \eqref{liehom} are induced by the  morphisms of complexes
\begin{equation*}
\la{liem1}
\Lambda^{\bullet + 1}\,\gl_W(A) \xrightarrow{\vartheta_\bullet} \CC_\bullet(\End\,W \otimes A)
\xrightarrow{\tr_\bullet} \CC_\bullet(A)\ ,
\end{equation*}
where $\,\vartheta_\bullet\, $ is defined by
$$
\vartheta_n(\xi_0 \wedge \xi_1 \wedge \ldots \wedge \xi_n) = \sum_{\sigma \in S_n} \,
\mbox{\rm sgn}(\sigma)\,(\xi_0,\,\xi_{\sigma(1)}, \ldots , \xi_{\sigma(n)})\ ,
$$
and $ \tr_\bullet $ is given by the generalized trace maps
$\,\tr_n:\, (\End\,W \otimes A)^{\otimes (n+1)} \to A^{\otimes (n+1)}
\,$ (see \cite{L}, 10.2.3).

It turns out that there is a natural map relating the Lie algebra homology $ \H_{\bullet}(\gl_W(A); k) $
to representation homology of $A$.
To construct this map we will realize the Lie algebra homology as Quillen homology of the category
$ \DGL^+_k $  of non-negatively graded DG Lie algebras ({\bf cf.} Example~1 in Section~\ref{AQh}).
This category has a natural model structure, which is compatible with
the model structure on $ \DGA_k $ via the forgetful functor $ \mbox{Lie}:\,\DGA^+_k \to \DGL^+_k $
(see \cite{Q2}, Part II, Sect. 5). Fix a cofibrant resolution $\,\alpha: R \sonto A $ of $A$
in $\DGA^+_k$. Then, for each finite-dimensional vector space $ W $, tensoring $ \alpha $ by
$ \End(W) $ yields an acyclic fibration in $ \DGA_k^+ $ which, in turn, yields (via the forgetful functor)
an acyclic fibration in $ \DGL_k^+\,$:
$$
\tilde{\alpha}:\ \gl_W(R) \sonto \gl_W(A)\ .
$$
Now, let $\,\beta: \, {\mathfrak L}_W \sonto  \gl_W(A) \,$ be a cofibrant resolution of $ \gl_W(A) $
in  $ \DGL_k^+ $. Since  $ \tilde{\alpha} $ is an acyclic fibration, $\, \beta \,$ lifts through
$ \tilde{\alpha} $ giving a quasi-isomorphism
$\, \tilde{\beta}: {\mathfrak L}_W \sonto  \gl_W(R)  \,$.
Combining this quasi-isomorphism with traces induces the map of complexes
\begin{equation}
\la{ul2}
{\mathfrak L}_W /[{\mathfrak L}_W, {\mathfrak L}_W]  \xrightarrow{\tilde \beta} \gl_W(R)/[ \gl_W(R), \gl_W(R)]
= (\End\, W \otimes R)_\n \cong R_\n  \xrightarrow{\Tr_V} R_V\ .
\end{equation}
Now, for any cofibrant resolution
$ {\mathfrak L} \sonto \g $ of  in $ \DGL_k^+ $, the complex $ {\mathfrak L} /[{\mathfrak L}, {\mathfrak L}]  $ computes the Lie algebra homology
of $ \g $ with trivial coefficients, see \eqref{glie}. Thus, for any $V$ and $ W $, \eqref{ul2} induces the maps
\begin{equation}
\la{ul3}
\H_{n+1}(\gl_W(A); k) \to \H_n(A, V)\ , \quad n \ge 0\ .
\end{equation}
Letting $ W = k^d $ and taking the inductive limit (as $\, d \to \infty $), we identify
$$
 \varinjlim\,\H_{\bullet}(\gl_W(A);\, k) \cong
\H_\bullet(\varinjlim\,\gl_W(A) ;\,k) = \H_{\bullet}(\gl_\infty(A); k)\ .
$$
With this indentification, \eqref{ul3}  induces the maps
\begin{equation}
\la{liem}
\H_{n+1}(\gl_\infty(A); k) \to \H_n(A, V) \  ,\quad  \forall\,n \geq 0\ .
\end{equation}
\bthm[\cite{BKR}, Theorem~4.3]
\la{LQTT}
For each $ n \ge 0 $, the maps \eqref{ul3} and \eqref{liem} factor through the Loday-Quillen-Tsygan map
\eqref{liehom}. The induced maps are precisely the trace maps \eqref{char2}.
\ethm
Note that for $n=0$, the map \eqref{ul3} is simply the composition of obvious traces
\begin{equation*}
\la{liem0}
\H_{1}(\gl_W(A); k) \cong \gl_W(A)/[\gl_W(A),\, \gl_W(A)] = (\End\,W \otimes A)_\n \stackrel{\sim}{\to} A_\n \to A_V\ ,
\end{equation*}
so the claim of Theorem~\ref{LQTT} is immediate in this case.

\vspace{1ex}

\begin{remark}
The homology of a Lie algebra with trivial coefficients has a natural coalgebra structure ({\it cf.} \cite[10.1.3]{L}). One can show that the degree $(-1) $ map
$\, \tau: \H_{\bullet}(\gl_W(A); k) \to \H_\bullet(A, V)^{\GL(V)} $ defined by \eqref{ul3} is
a twisting cochain with respect to the coalgebra structure on $ \H_{\bullet}(\gl_W(A); k) $. In the stable limit (see Section~\ref{stabi} below), $\,\tau\,$
becomes an {\it acyclic} twisting cochain, which means that  the Lie algebra homology of $ \gl_\infty(A) $ is {\it Koszul dual}\, to the stable representation homology of $A$. For a precise statement of this result and its implications we refer the reader to \cite[Section~5]{BR} (see, in particular, {\it op. cit.}, Theorem~5.2). \end{remark}

\subsection{Stabilization theorem}
\la{stabi}
If $A$ is an ordinary algebra, a fundamental theorem of Procesi \cite{P} implies that
the traces of elements of $A$ generate the algebra $ A_V^{\GL(V)} $; in other words,
the algebra map
\begin{equation}
\la{proc}
\Sym[\Tr_V(A)_0]:\ \Sym[\HC_0(A)] \to A_V^{\GL(V)}
\end{equation}
is surjective for all $ V$.  A natural question is whether this result extends to higher traces:
namely, is the full trace map
\begin{equation}
\la{E19S4}
\bL \Tr_V(A)_\bullet :\ \bL[\HC(A)] \to \H_{\bullet}(A, V)^{\GL(V)}
\end{equation}
surjective?  We address this question in the forthcoming paper \cite{BR}, where
by analogy with matrix Lie algebras (see \cite{T, LQ}) we approach it in two
steps. First, we `stabilize' the family of maps
\eqref{E19S4} passing to an infinite-dimensional limit $ \dim_k V \to \infty $ and prove
that \eqref{E19S4} becomes an isomorphism in that limit. Then, for a finite-dimensional $V$,
we construct obstructions to $ \H_{\bullet}(A, V)^{\GL(V)} $ attaining its `stable limit'.
These obstructions arise as homology of a complex that measures the failure of \eqref{E19S4}
being surjective. Thus, the answer to the above question is negative.
A simple counterexample will be given in Section~\ref{dualnm} below.

We conclude this section by briefly explaining the stabilization procedure of \cite{BR}.
We will work with unital DG algebras $A$ which are {\it augmented} over $k$. We recall that the
category of such DG algebras is naturally equivalent to the category of non-unital DG algebras,
with $ A $ corresponding to its augmentation ideal $ \bar{A} $. We identify these two categories
and denote them by $ \DGA_{k/k} $. Further, to simplify the notation we take $ V = k^d $ and
identify $\, \End\,V = \M_d(k) \,$, $\, \GL(V) = \GL_k(d) \,$; in addition, for $ V = k^d $,
we will write  $ A_V $ as $ A_d $.
Bordering a matrix in $ \M_d(k) $ by $0$'s on the right and on the bottom gives an embedding
$\,\M_d(k) \into \M_{d+1}(k) \,$ of non-unital algebras. As a result, for each $ B \in \cDGA_k $,
we get a map of sets
\begin{equation}
\la{isoun0}
\Hom_{\DGA_{k/k}}(\bar{A},\,\M_d(B)) \to \Hom_{\DGA_{k/k}}(\bar{A},\,\M_{d+1}(B))
\end{equation}
defining a natural transformation of functors from $ \cDGA_k $ to $ \Sets $.
Since $B$'s are unital and $ A $ is augmented, the restriction maps
\begin{equation}
\la{isoun1}
\Hom_{\DGA_k}(A,\,\M_d(B)) \stackrel{\sim}{\to} \Hom_{\DGA_{k/k}}(\bar{A},\,\M_d(B)) \ ,\quad
\varphi \mapsto \varphi|_{\bar{A}}
\end{equation}
are isomorphisms for all $ d \in \N $. Combining \eqref{isoun0} and \eqref{isoun1},
we thus have natural transformations
\begin{equation}
\la{isoun}
\Hom_{\DGA_k}(A,\,\M_d(\,\mbox{--}\,)) \to \Hom_{\DGA_k}(A,\,\M_{d+1}(\,\mbox{--}\,))\ .
\end{equation}
By standard adjunction, \eqref{isoun} yield an inverse system of
morphisms $\,\{ \mu_{d+1, d}: A_{d+1} \to A_d \} \,$ in $ \cDGA_k $. Taking the limit of this
system, we define
$$
A_{{\infty}} := \varprojlim_{d\,\in\,\mathbb N} A_d \ .
$$
Next, we recall that the group $ \GL(d) $ acts naturally on $ A_d $,
and it is easy to check that $\,\mu_{d+1, d}: A_{d+1} \to A_d\,$ maps the subalgebra
$ A_{d+1}^{\GL} $ of $\GL $-invariants in $ A_{d+1} $ to the subalgebra $ A_d^{\GL}$ of
$\GL $-invariants in $A_d$. Defining $ \GL(\infty) := \varinjlim\, \GL(d) $ through
the standard inclusions $ \GL(d) \into \GL(d+1) $, we extend the actions of $ \GL(d) $
on $ A_d $ to an action of $ \GL(\infty) $ on $ A_{\infty} $ and let $ A^{\GL(\infty)}_{{\infty}} $
denote the corresponding invariant subalgebra. Then one can prove (see~\cite{T-TT})
\begin{equation}
\la{isolim}
A^{\GL(\infty)}_{{\infty}} \cong \varprojlim_{d\,\in\,\mathbb N} A^{\GL(d)}_d \ .
\end{equation}
This isomorphism  allows us to equip $ A^{\GL(\infty)}_{{\infty}} $ with a natural topology:
namely, we put first the discrete topology on each $ A^{\GL(d)}_d $ and equip
$\,\prod_{d \in \N} A^{\GL(d)}_d \,$ with the product topology; then, identifying
$ A^{\GL(\infty)}_{{\infty}} $ with a subspace in
$\,\prod_{d \in \N} A^{\GL(d)}_d \,$ via \eqref{isolim}, we put on
$ A^{\GL(\infty)}_{{\infty}} $ the induced topology. The
corresponding topological DG algebra will be denoted $ A^{\GL}_{{\infty}} $.

Now, for each $ d \in \N $, we have the commutative diagram
\[
\begin{diagram}[small, tight]
  &       &     \FT(A)     &        & \\
  &\ldTo^{\Tr_{d+1}(A)_\bullet}  &           &  \rdTo^{\Tr_{d}(A)_\bullet} &  \\
A_{d+1}^{\GL} &       & \rTo^{\mu_{d+1, d}}  &        &  A_d^{\GL}
\end{diagram}
\]
where $\, \FT(A) \,$ is the cyclic functor restricted to $ \DGA_{k/k} $ ({\it cf.} Section~\ref{S4.1}).
Hence, by the universal property of inverse limits, there is a morphism of complexes $\, \Tr_{\infty}(A)_\bullet :\,
\FT(A) \to A^{\GL}_{{\infty}}\,$ that factors $ \Tr_{d}(A)_\bullet $ for each $ d \in \N $. We extend this morphism
to a homomorphism of commutative DG algebras:
\begin{equation}
\la{trinf}
\Tr_{\infty}(A)_\bullet :\ \bL[\FT(A)] \to A^{\GL}_{{\infty}}\ .
\end{equation}

The following lemma is one of the key technical results of \cite{BR} (see {\it loc. cit.}, Lemma~3.1).
\blemma
\la{dense}
The map \eqref{trinf} is {\rm topologically} surjective: {\it i.e.}, its image is dense in $ A^{\GL}_{{\infty}} $.
\elemma

\vspace{1ex}

\noindent
Letting $ A_\infty^{\Tr}  $ denote the image of \eqref{trinf}, we
define the functor
\begin{equation}
\la{trfun}
(\,\mbox{--}\,)^{\Tr}_{\infty}\,:\ \DGA_{k/k} \to \cDGA_k\ ,\quad A \mapsto A_\infty^{\Tr}\ .
\end{equation}
The algebra maps \eqref{trinf} then give a morphism of functors
\begin{equation}
\la{morfun}
\Tr_{\infty}(\,\mbox{--}\,)_\bullet :\ \bL[\FT(\,\mbox{--}\,)] \to (\,\mbox{--}\,)^{\Tr}_{\infty}\ .
\end{equation}

Now, to state the main result of \cite{BR} we recall that
the category of augmented DG algebras $ \DGA_{k/k} $ has a natural model structure induced from
$ \DGA_k $. We also recall the derived Feigin-Tsygan functor
$\, \L\FT(\,\mbox{--}\,):\, \Ho(\DGA_{k/k}) \to \Ho(\cDGA_k)\, $ inducing the isomorphism of
Proposition~\ref{S4P1}$(b)$.
\bthm[\cite{BR}, Theorem~4.2]
\la{eqfun}
$(a)$ The functor \eqref{trfun} has a total left derived functor
$\,\L(\,\mbox{--}\,)^{\Tr}_{\infty}\,:\ \Ho(\DGA_{k/k}) \to \Ho(\cDGA_k)\,$.

$(b)$ The morphism \eqref{morfun} induces an isomorphism of functors
$$
\Tr_{\infty}(\,\mbox{--}\,)_\bullet :\  \bL[\L\FT(\,\mbox{--}\,)] \stackrel{\sim}{\to}
\L (\,\mbox{--}\,)^{\Tr}_{\infty}\ .
$$
\ethm
\noindent
By definition, $\, \L(\,\mbox{--}\,)^{\Tr}_{\infty} $ is given by
$\,
\L(A)^{\Tr}_{\infty} = (QA)^{\Tr}_{\infty}\,$,
where $ QA $ is a cofibrant resolution of $A$ in $ \DGA_{k/k}\,$. For an ordinary augmented $k$-algebra $A \in \Alg_{k/k} $, we set
$$
\DRep_\infty(A)^\Tr :=  (QA)^{\Tr}_{\infty} \ .
$$
By part $(a)$ of Theorem~\ref{eqfun}, $ \DRep_\infty(A)^\Tr $ is well defined. On the other hand, part $(b)$ implies
\begin{corollary}
\la{corf1}
For any $ A \in \Alg_{k/k} $, $\, \Tr_{\infty}(A)_\bullet $ induces an isomorphism of graded commutative algebras
\begin{equation}
\la{funhc}
\bL[\rHC(A)] \cong \H_\bullet[\DRep_\infty(A)^\Tr]\ .
\end{equation}
\end{corollary}

In fact, one can show that $ \H_\bullet[\DRep_\infty(A)^\Tr] $
has a natural structure of a graded Hopf algebra,
and the isomorphism of Corollary~\ref{corf1} is actually an
isomorphism of Hopf algebras. This isomorphism is analogous
to the Loday-Quillen-Tsygan isomorphism \eqref{liehom5} computing the
stable homology of matrix Lie algebras $ \gl_n(A)$ in terms of
cyclic homology. Heuristically, it implies that the cyclic homology
of an augmented algebra is determined by its representation homology.

\section{Abelianization of the Representation Functor}
\la{S5}
``If homotopical algebra is thought of as `nonlinear' or `non-additive' homological algebra, then
it is natural to ask what is the `linearization' or `abelianization' of this situation'' (Quillen, \cite{Q1},
\S~II.5). In Section~\ref{AQh}, following Quillen, we defined the abelianization of a model category
$ \C $ as the category $ \C^{\rm ab} $ of abelian group objects in $ \C $.  As a next step, one should
ask for abelianization of a functor $ F: \C \to \D $ between model categories. We formalize this notion
in Section~\ref{AbFun} below, and then apply it to our representation functor
$\, (\mbox{---})_V:\, \DGA_k \to \cDGA_k \,$. As a result, for a given algebra $ A \in \DGA_k $,
we get an additive left Quillen functor
\begin{equation}
\la{E1}
(\mbox{---})^{\rm ab}_V:\ \DGBimod(A) \to \DGMod(A_V)\ ,
\end{equation}
relating the category of DG bimodules over $ A $ to DG modules over $ A_V $. In the case of ordinary algebras, this functor was introduced by M. Van den Bergh \cite{vdB}. He found that \eqref{E1} plays a special role in noncommutative geometry of smooth algebras, transforming noncommutative objects on $A$ to classical geometric objects on
$ \Rep_V(A) $.  Passing from $ \Rep_V(A) $ to $ \DRep_V(A) $, we constructed in \cite{BKR} the derived functor of \eqref{E1}  and showed that it plays a similar role in the geometry of arbitrary (not necessarily smooth) algebras. The original definition
of \eqref{E1} in \cite{vdB} is given by an explicit but somewhat {\it ad hoc} construction ({\it cf.} \eqref{E111} below).
Characterizing Van den Bergh's functor as abelianization of the representation functor provides a conceptual explanation of the results of \cite{vdB} and \cite{BKR}.  At the derived level, this also leads to a new spectral sequence relating representation homology to Andr\`e-Quillen homology (see Section~\ref{RHAQ} below).

\subsection{Abelianization as a Kan extension}
\la{AbFun}
Let $\,F: \C \to \D \,$ be a right exact ({\it i.e.}, compatible with finite colimits) functor between model categories.
As in Section~\ref{AQh}, we assume that $ \C^{\rm ab} $ and $ \D^{\rm ab}$ are abelian categories with enough projectives
and the abelianization functors $\, \Ab_\C :\, \C \to \C^{\rm ab}  \,$ and  $\, \Ab_\D :\, \D \to \D^{\rm ab}  \,$ exist and form Quillen pairs, see \eqref{Abi}. In general, $ F $ may
not descend to an additive functor $\,F^{\rm ab}: \C^{\rm ab} \to \D^{\rm ab} \,$ that would complete the commutative diagram
\begin{equation}
\la{dab1}
\begin{diagram}[small, tight]
\C               & \rTo^{F}    &  \D        \\
\dTo^{\Ab_\C}          &                  & \dTo_{\Ab_\D}   \\
  \C^{\rm ab}                & \rDotsto^{F^{\rm ab}}     &   \D^{\rm ab}
\end{diagram}
\end{equation}
Following a standard categorical approach (see \cite[Chapter~X]{ML}), we remedy
this problem in two steps. First, we define the `best left approximation' to
$ F^{\rm ab} $ (which we call the left abelianization) as a right Kan extension of
$\, \Ab_\D \circ F \,$ along $ \Ab_\C $.
Precisely, the {\it left abelianization} of $ F $ is a right exact additive functor
$\, F_l^{\rm ab}: \C^{\rm ab} \to \D^{\rm ab} \,$ together with a natural transformation $\,t : F_l^{\rm ab} \circ  \Ab_\C \to \Ab_\D \circ F\,$ satisfying the following universal property:
\begin{quote}
For any pair $ (G, s) $ consisting of a  right exact additive functor $\, G: \C^{\rm ab} \to \D^{\rm ab} \,$
and a natural transformation $\, s: G  \circ  \Ab_\C \to \Ab_\D \circ F \,$, there is a unique natural
transformation $\, s' : G \to F_l^{\rm ab} \,$ such that the following diagram commutes:
\begin{equation}
\la{dab2}
\begin{diagram}[small, tight]
G\circ\Ab_{\C} &                        &  \rTo^{s}      &                &  \Ab_{\D} \circ F  \\
             & \rdDotsto_{s' \, \Ab_\C}\quad  &                &  \ruTo_{t} &  \\
             &                      & F_l^{\rm ab} \circ \Ab_{\C} &               &
\end{diagram}
\end{equation}
\end{quote}
Next, we say that $ F $ is {\it abelianizable} if $\, F_l^{\rm ab} \,$ exists,
and the corresponding natural transformation $\,t : F_l^{\rm ab} \circ  \Ab_\C \to \Ab_\D \circ F\,$
is a natural equivalence. In this case, we drop the subscript in $\, F_l^{\rm ab} \,$
and call  $ F^{\rm ab} $ the {\it abelianization} of $ F $. As usual, the above universal property guarantees
that when it exists, the functor $\,F^{\rm ab} : \C^{\rm ab} \to \D^{\rm ab} \,$ is unique up to a canonical isomorphism.

%The following lemma shows that abelianization is compatible with usual derived functors.
%
%\blemma
%\la{derab}
%Suppose that $\, F \,$ is abelianizable and has a total left derived functor
%$ \L F: \Ho(\C) \to \Ho(\D) $. Then $ F^{\rm ab} $ has a total left derived functor
%$ \L F^{\rm ab}: \Ho(\C^{\rm ab}) \to \Ho(\D^{\rm ab})  $, and there is a canonical
%isomorphism of functors
%
%$$
%\L F^{\rm ab} \circ \L \Ab_{\C} \cong \L\Ab_{\D } \circ \L F\ .
%$$
%\elemma
%
%Note that the existence of the derived functor $\L F^{\rm ab} $ is just a consequence of our assumptions
%on $ \C^{\rm ab}$ and $\, F^{\rm ab}: \C^{\rm ab} \to \D^{\rm ab} \,$, while the existence of
%the above canonical isomorphism is a consequence of universal properties.

\subsection{The Van den Bergh functor}
\la{S5.1}
In this section, we assume for simplicity that $S = k$ and $ V $ is concentrated in degree $0$.
Given $\, R \in \DGA_k \,$, let $\,\pi_V:\,R \to \End\,V \otimes \rtv{R} \,$ denote the universal DG algebra homomorphism,
see Proposition~\ref{S2P1}$(a)$. The complex
$\,\rtv{R} \otimes V \,$ is naturally a left DG module over $ \End\,V \otimes \rtv{R}  $ and right DG module
over $ \rtv{R} \,$, so restricting the left action via $ \pi $ we can regard  $\,\rtv{R} \otimes V \,$ as a DG
bimodule over $R$ and $\rtv{R}$. Similarly, we can make $\,V^* \otimes \rtv{R}$ a $\,\rtv{R}$-$R$-bimodule.
Using these bimodules, we define the functor
\begin{equation}
\la{rtvm}
\rtv{-}:\ \DGBimod(R) \to \DGBimod(\rtv{R}) \ , \quad
M \mapsto (V^* \otimes \sqrt[V]{R}) \otimes_R M \otimes_R (\sqrt[V]{R} \otimes V) \ .
\end{equation}
Now, recall that $\,R_V := (\rtv{R})_\nn\,$ is a commutative DGA. Using the natural projection
$\,\rtv{R} \to R_V \,$, we regard $R_V$ as a DG bimodule over $ \rtv{R} $ and define
\begin{equation}
\la{abb}
(\mbox{---})_\nn:\ \DGBimod(\rtv{R}) \to \DGMod(R_V) \ , \quad
M \mapsto M_\nn := M \otimes_{(\rtv{R})^{\rm e}} R_V  \ .
\end{equation}
Combining \eqref{rtvm} and \eqref{abb}, we get the functor
\begin{equation}
\la{E11}
(\mbox{---})^{\rm ab}_V:\ \DGBimod(R) \to \DGMod(R_V)\ , \quad M \mapsto
M^{\rm ab}_V := (\rtv{M})_\nn \ .
\end{equation}
It is easy to check that, for any $ M \in \DGBimod(R) $, there is a canonical isomorphism of $R_V$-modules
\begin{equation}
\la{E111}
M^{\rm ab}_V = M \otimes_{R^{\rm e}} (\End\,V \otimes R_V)\ .
\end{equation}
Thus,  \eqref{E11} is indeed a DG extension of Van den Bergh's functor defined
in \cite{vdB}, Section~3.3.

The next lemma is analogous to Proposition~\ref{S2P1} for DG algebras. We recall that, if $ R $ is a DG algebra and
$M$,$\,N$ are DG modules over $R$, the  morphism complex $\, \HOM_R(M,\,N) \,$ is a complex of vector spaces with
$n$-th graded component consisting of all $R$-linear maps $\,f:\, M \to N\,$ of degree $ n $ and
the $n$-th differential given by $\,d(f) = d_N \circ f - (-1)^n f \circ d_M\,$.
\blemma[\cite{BKR}, Lemma~5.1]
\la{L2}
For any $\, M \in \DGBimod(R)$, $\,N \in \DGBimod(\rtv{R}) \,$ and
$\, L \in \DGMod(R_V) \,$, there are canonical isomorphisms of complexes

\vspace{0.8ex}

$(a)$ $\,\HOM_{(\rtv{R})^{\rm e}}(\rtv{M},\,N)
\cong \HOM_{\eR}(M,\,\End\,V \otimes N) \,$,

\vspace{0.8ex}

$(b)$ $\, \HOM_{R_V}(M^{\rm ab}_V,\,L) \cong \HOM_{\eR}(M,\,\End\,V \otimes L) \,$.

\elemma

\vspace{1ex}

\noindent
{\bf Example.}
Let $\, \Omega^1 R \, $ denote the kernel of the multiplication map $\,R \otimes R \to R\,$ of a DG algebra $R$.
This is naturally a DG bimodule over $ R $, which, as in the case of ordinary algebras, represents the complex
of graded derivations $\,\DER (R,\,M) \,$, {\it i.e.}  $\, \DER (R,\,M) \cong \HOM_{\eR}(\Omega^{1} R,\,M) \,$
for any $\, M \in \DGBimod(R) \,$ (see, e.g., \cite{Q3}, Sect.~3). Lemma~\ref{L2} then implies canonical isomorphisms
\begin{equation}
\la{omvdb}
\rtv{\Omega^1 R} \cong \Omega^1(\rtv{R}) \quad , \quad
(\Omega^1 R)^{\rm ab}_V \cong  \Omega_{\rm com}^1(R_V)\ .
 \end{equation}
To prove \eqref{omvdb} it suffices to check that $ \Omega^1(\rtv{R}) $ and $  \Omega_{\rm com}^1(R_V) $
satisfy the adjunctions of Lemma~\ref{L2} and then appeal to Yoneda's Lemma.  We leave this as an exercise
to the reader.

\vspace{1ex}

We are now in position to state the main theorem of this section. This theorem justifies, in particular,
our notation for the functor \eqref{E11}.
\bthm
\la{abvdb}
The functor \eqref{E11} is the abelianization of the representation functor \eqref{S2E10}.
\ethm
\bproof
Given a DG algebra $ R \in \DGA_k $, we set $ \C := \DGA_k/R $ and $ \D := \cDGA_k/R_V$.
Then, as in Section~\ref{AQh} (see Example 2 and Example 3), we can identify
$ \C^{\rm ab} = \DGBimod(R) $ and $ \D^{\rm ab} = \DGMod(R_V) $. Under this identification,
the abelianization functors $ \Ab_\C $ and $ \Ab_\D $ become
$$
\Omega^1( \mbox{--}/R):\, \DGA_k/R \, \to \, \DGBimod(R)\,,\quad
B \mapsto R \otimes_B \Omega^1(B) \otimes_B R\ ,
$$
$$
\Omega_{\rm com}^1( \mbox{--}/R_V):\, \cDGA_k/R_V  \,\to \, \DGMod(R_V)\,,\quad
B  \mapsto R_V \otimes_B \Omega_{\rm com}^1(B) \ ,
$$
where $ \Omega^1(B) $ and $ \Omega^1_{\rm com}(B) $ are the modules of noncommutative and
commutative (K\"ahler) differentials, respectively. We prove Theorem~\ref{abvdb} in two steps.
First, we show that for the functor \eqref{E11}, there is a canonical natural equivalence
\begin{equation}
\la{ciso}
t:\   (\mbox{---})^{\rm ab}_V \circ \Omega^1( \mbox{--}/R)\, \stackrel{\sim}{\to}\,  \Omega_{\rm com}^1( \mbox{--}/R_V) \circ
(\mbox{---})_V
\end{equation}
which makes \eqref{dab1} a commutative diagram. Then, we verify the universal property stated in Section~\ref{AbFun}.

To establish \eqref{ciso} we will use the Yoneda Lemma. For any $ B \in \DGA_k/R $ and $ L \in \DGMod(R_V) $,
Lemma~\ref{L2} together with \eqref{omvdb} gives natural isomorphisms:
\begin{eqnarray*}
\Hom_{R_V}(\Omega^1(B/R)^{\rm ab}_V,\,L)
&\cong& \Hom_{R^{\rm e}}(\Omega^1(B/R),\, \End(V) \otimes L) \\
&\cong& \Hom_{B^{\rm e}}(\Omega^1(B),\, \End(V) \otimes L)   \\
&\cong& \Hom_{B^{\rm e}}(\rtv{\Omega^1(B)},\, L)  \\
&\cong& \Hom_{B_V}((\rtv{\Omega^1(B)})_\nn,\, L)  \\
&\stackrel{{\rm df}}{=} & \Hom_{B_V}(\Omega^1(B)^{\rm ab}_V,\, L)  \\
&\cong & \Hom_{B_V}(\Omega_{\rm com}^1(B_V),\, L) \\
&\cong&  \Hom_{R_V}(R_V \otimes_{B_V} \Omega_{\rm com}^1(B_V),\, L)  \ .
\end{eqnarray*}
Hence,  $ \Omega^1(B/R)^{\rm ab}_V $ is canonically isomorphic to $ R_V \otimes_{B_V} \Omega_{\rm com}^1(B_V) $, which is
equivalent to \eqref{ciso}.

To verify  the universal property for abelianization we will use the functorial isomorphism
\begin{equation}
\la{fiso}
M = \Coker [\,\Omega^1(R) \to \Omega^1(T_R M/R)\,]
\end{equation}
where $ T_R M $ is the tensor algebra of $ M $ equipped with the canonical projection
$ T_R M \onto R $. This isomorphism follows from the standard cotangent
sequence for the tensor algebra $ T = T_R M $
$$
T \otimes_R \Omega_k^1(R) \otimes_R T \to \Omega_k^1(T) \to T \otimes_R M \otimes_R T \to 0\ ,
$$
which is proved, for example, in \cite{CQ} (see {\it loc. cit.}, Corollary~2.10).

Now, given a right exact additive functor $\, G: \DGBimod(R) \to \DGMod(R_V) \,$
with natural transformation
\begin{equation*}
\la{giso}
s:\   G \circ \Omega^1( \mbox{--}/R)\, \to\,  \Omega_{\rm com}^1( \mbox{--}/R_V) \circ  (\mbox{---})_V
\end{equation*}
we compose $ s $ with the inverse of \eqref{ciso}  and use \eqref{fiso} to define the $B_V$-module maps
$$
G(M) = \Coker[\,G \circ \Omega^1(R) \to G \circ \Omega^1(T_R M/R)\,]\,
\xrightarrow{s_M'} \,
\Coker[\, (\mbox{--})^{\rm ab}_V \circ  \Omega^1(R) \to (\mbox{--})^{\rm ab}_V \circ  \Omega^1(T_R M/R)\,] = M_V^{\rm ab}
$$
The maps $ s_M' $ define a natural transformation $ s': G \to (\mbox{---})^{\rm ab}_V $
making \eqref{dab2} commutative. This proves the required universal property and finishes the proof of the
theorem.
\eproof

Now, as in the case of DG algebras ({\it cf.} Theorem~\ref{S2T2}), Lemma~\ref{L2} easily implies
\bthm
\la{S5T2}
$(a)$ The functors $\,(\,\mbox{--}\,)^{\rm ab}_V :\,\DGBimod(R)
\rightleftarrows \DGMod(R_V)\,: \End\,V \otimes {\mbox{--}}\,$
form a Quillen pair.

$(b)$ $\, (\,\mbox{--}\,)^{\rm ab}_V $ has a total left derived functor
$$
\L(\,\mbox{--}\,)^{\rm ab}_V:\,\D(\DGBimod\, R) \to \D(\DGMod\, R_V)
$$
which is left adjoint to the exact functor $\, \End\,V \otimes {\mbox{--}} \,$.
\ethm

Now, for ordinary algebras,  the derived Van den Bergh functor can be defined
using  a standard procedure in differential homological algebra
({\it cf.} \cite{HMS}, \cite{FHT}). Given $ A \in \Alg_k $ and a complex $M$ of
bimodules over $A$,  we first choose a semi-free resolution $\,f: R \to A\,$ in $\, \DGA_k \,$
and consider $ M $ as a DG bimodule over $ R $ via $f$. Then, we choose a semi-free resolution
$\,F(R,M) \to M \,$  in the category $ \DGBimod(R) $ and apply to $F(R,M)$
the functor \eqref{E11}.  Combining Theorem~\ref{S5T2} with
Proposition~\ref{Kelpr} in Section~2.4, we get
\begin{corollary}
\la{T2}
Let $ A \in \Alg_k $, and let $ M $ be a complex of bimodules over $A$.
The assignment $\,M \mapsto F(R,M)_V\,$ induces a well-defined functor
$$
\L(\mbox{---})^{\rm ab}_V:\ \D(\Bimod\,A) \to \D(\DGMod\,R_V)\ ,
$$
which is independent of the choice of the resolutions $ R \to A $ and $ F \to M $ up to auto-equivalence of
$\, \D(\DGMod\,R_V) $ inducing the identity on homology.
\end{corollary}
This result can be also verified directly, using polynomial homotopies  (see \cite{BKR}).

\vspace{1ex}

\begin{definition}
\la{HB}
For $ M \in \DGBimod(A) $, we call $\, \H_\bullet(M, V) := \H_\bullet[\L(M)^{\rm ab}_V]\,$
the {\it representation homology of the bimodule} $M$ with coefficients in $V$. If $ M \in \Bimod(A) $ is
an ordinary bimodule viewed  as a complex in $\D(\Bimod\,A)$ concentrated in degree $0$,  then
$\, \H_0(M, V) \cong M^{\rm ab}_V \ ,$.
\end{definition}

\vspace{1ex}

We now give some applications of Theorem~\ref{S5T2}.

\subsection{Derived tangent spaces}
\la{derts}
First, we compute the derived tangent spaces $ \pi_i(\DRep_V(A),\,\varrho) $ for
$ \DRep_V(A) $ viewed as an affine DG scheme (see Section~\ref{2.8} for notation and terminology).
Let $\,\varrho:\,A \to \End\,V\,$ be a fixed representation of $A$. Choose a cofibrant resolution $ R \sonto A $,
and let $\,\varrho_V:\,R_V \to k \,$ be the DG algebra homomorphism corresponding to the representation
$\,\varrho:\,R \to A \to \End\,V$.  Now, for any DG bimodule $M$, there is a canonical map of complexes induced by
the functor \eqref{E11}:
\begin{equation*}
\la{trns}
\DER(R,M) \cong \HOM_{\eR}(\Omega^{1} R,\,M) \xrightarrow{(\mbox{--})^{\rm ab}_V}
\HOM_{R_V}(\Omega^{1}(R_V),\,M_V) \cong \DER(R_V,M_V)\ .
\end{equation*}
We claim that for $ M = \End\,V $ viewed as a DG bimodule via $\, \varrho\,$, this map is an isomorphism.
Indeed,
\begin{eqnarray*}
\DER(R_V, \, k) &\cong& \HOM_{R_V}(\Omega^1(R_V),\,k) \\
                 &\cong& \HOM_{R_V}((\Omega^1 R)_V,\,k)\qquad\ \, [\,\mbox{see}\,\eqref{omvdb}\,] \\
                 &\cong& \HOM_{\eR} (\Omega^1 R ,\,\End\,V) \qquad [\,\mbox{see\,Lemma~\ref{L2}$(b)$}\,]\\
                 &\cong& \DER(R,\,\End\,V)\ .
\end{eqnarray*}
This implies
$$
 \pi_{\bullet}(\DRep_V(A),\,\varrho) := \H_\bullet[\DER(R_V, \, k)] \cong \H_\bullet[\DER(R,\,\End\,V)]\ .
$$
The following proposition is now a direct consequence of \cite{BP}, Lemma~4.2.1 and Lemma~4.3.2.
\bprop
\la{P2}
There are canonical isomorphisms
$$
\pi_{i}(\DRep_V(A),\,\varrho)\cong \left\{
\begin{array}{lll}
\Der(A,\,\End\,V)\ & \mbox{\rm if} &\ i = 0\\*[1ex]
\HH^{i+1}(A,\,\End\,V)\ & \mbox{\rm if} &\ i \ge 1
\end{array}
\right.
$$
where $ \HH^{\bullet}(A,\,\End\,V) $ denotes the Hochschild cohomology of the representation $ \varrho: A \to \End\,V $.
\eprop

As explained in Section~\ref{S2.3.5}, in the case when $V$ is a single vector space concentrated in degree $0$
$ \Rep_V(R) $ is isomorphic to the DG scheme $ \RAct(R,\,V) $ constructed in \cite{CK}.
This implies that $\,\pi_{\bullet}(\DRep_V(A),\,\varrho)\,$ should be isomorphic to $\,\pi_{\bullet}( \RAct(R,\,V),\,\varrho)\,$,
which is indeed the case, as one can easily see by comparing our Proposition~\ref{P2} to \cite{CK},
Proposition~3.5.4$(b)$.

\subsection{Periodicity and the Connes differential}
\la{S5.33}
One of the most fundamental properties of cyclic homology
is Connes' periodicity exact sequence ({\it cf.} \cite{L}, 2.2.13):
\begin{equation}
\la{ISB}
\ldots \to \rHH_n(A) \xrightarrow{I} \rHC_n(A) \xrightarrow{S}
\rHC_{n-2}(A) \xrightarrow{B} \rHH_{n-1}(A) \to \ldots
\end{equation}
This sequence involves two important operations on cyclic homology:
the periodicity operator $S$ and the Connes differential $B$.
It  turns out that $S$ and $B$ induce (via trace maps)
some natural operations on representation homology, and there is a
periodicity exact sequence for  $ \H_\bullet(A,V) $ similar to \eqref{ISB}.
We briefly describe this construction below referring the reader
to \cite{BKR}, Section~5.4, for details and proofs.

We begin by constructing the abelianized version of the trace maps
\eqref{char2}.  Recall, if $M$ is a bimodule over a DG algebra $A$, a {\it trace} on $M$
is a map of complexes $ \tau: M \to N $ vanishing on the commutator subspace
$ [A,M] \subseteq M$. Every trace on $M$ factors through the canonical
projection $ M \onto M_\n := M/[A,M] $, which is thus the universal trace.
Given a finite-dimensional vector space $V$, let $ \pi_V(M) $ denote the canonical map
corresponding to $ \id_{M_V} $ under the isomorphism of Lemma~\ref{L2}. The map of complexes
\begin{equation}
\la{S5E1}
\Tr_V(M)\,:\ M \xrightarrow{\pi_V(M)} \End\,V \otimes M^{\rm ab}_V \xrightarrow{\Tr_V \otimes \id} M^{\rm ab}_V\ ,
\end{equation}
is then obviously a trace, which is functorial in $M$. Thus \eqref{S5E1} defines a morphism of functors
\begin{equation}
\la{S5E2}
\Tr_V:\ (\,\mbox{--}\,)_\n \to (\,\mbox{--}\,)^{\rm ab}_V\ .
\end{equation}
As in the case of DG algebras, we have the following result.
\blemma
\la{dtrm}
\eqref{S5E2} induces a morphism of functors $\, \D(\DGBimod\,A) \to \D(k)\,$:
\begin{equation}
\la{S5E3}
\Tr_V:\ \L(\,\mbox{--}\,)_\n \to \L(\,\mbox{--}\,)^{\rm ab}_V\ ,
\end{equation}
where $ \L(\,\mbox{--}\,)^{\rm ab}_V $ is the derived representation functor introduced in Theorem~\ref{T2}.
\elemma
To describe \eqref{S5E3} on  $ M \in \DGBimod(A) $ explicitly we choose an
semi-free resolution $ p:\,R \sonto A $, regard $M$ as a bimodule over $R$ via $p$
and choose a semi-free resolution of $ F(R,M) \sonto M $ in $ \DGBimod(R)$. Then
\eqref{S5E3} is induced by the map \eqref{S5E1} with $ M $ replaced by $ F(R, M) $:
\begin{equation}
\la{S5E4}
\Tr_V(M):\ F(R, M)_\n \to F(R,M)^{\rm ab}_V\ .
\end{equation}
Note that, if $ A \in \Alg_k $ and $ M \in \Bimod(A) $, then
$\,\H_n[F(R, M)_\n] \cong \HH_n(A,M) \,$ for all $ n \ge 0 $,
so \eqref{S5E4} induces the trace maps on Hochschild homology:
\begin{equation}
\la{S5E5}
\Tr_V(M)_n:\ \HH_n(A,M) \to \H_n(M, V)\ ,\quad \forall\,n \ge 0\ ,
\end{equation}
where $ \H_n(M, V) := \H_n[\L(M)^{\rm ab}_V] $ is the representation homology of $M$ in sense of
Definition~\ref{HB}.

Now, given an algebra $ A \in \Alg_k $, fix an semi-free resolution $\,p: R \sonto A \,$ in $ \DGA_k^+ $ and
consider the commutative DG algebra $ R_V $. Let   $ \Omega_{\rm com}^1(R_V) $ be  the DG module
of  K\"ahler differentials of $ R_V $, and let $ \partial_V:\, R_V \to \Omega_{\rm com}^1(R_V) $ denote the
universal derivation (the de Rham differential) on $R_V$ .  By Theorem~\ref{comp}, $ R_V $ is isomorphic to a (graded) polynomial algebra.
Hence $ \Ker(\partial_V) \cong k $ for all $V$. On the the hand, the cokernel of $ \partial_V $ is a nontrivial
complex which is not, in general,  acyclic in positive degrees. We denote this complex by
$\,  \Omega_{\rm com}^1(R_V)/\partial R_V \,$, and for each integer $ n \ge 0 $, define
\begin{equation}
\la{Totv1e}
\tH_n(A,V) := \bar{\H}_n(A,V)\,\oplus\,
\H_{n-1}[\Omega_{\rm com}^1(R_V)/\partial R_V] \,\oplus\, \H_{n-3}[\Omega_{\rm com}^1(R_V)/\partial R_V]\,\oplus\,\ldots
\end{equation}
Note that the (reduced) representation homology $ \bar{\H}_n(A,V) $ appears as a direct summand of $ \tH_n(A,V) $.
It turns out that there are canonical maps
\begin{equation*}
\la{trex}
\tTr_V(A)_n:\, \rHC_n(A) \to \tH_n(A,V)\ ,\quad \forall\,n \ge 0\ ,
\end{equation*}
lifting  the traces \eqref{char2} to $ \bar{\H}_n(A,V) $.  Moreover, for all $ n \ge 0$,  one can construct natural
maps $ S_V:  \tH_n(A,\,V) \to  \tH_{n-2}(A,V)\,$ and $ B_V: \bar{\H}_n(A,\,V) \to  \H_{n}(\Omega^1 A, V) $
making commutative the following diagrams ({\it cf.} \cite{BKR}, Theorem~5.2):
\begin{equation}
\la{S5D3}
\begin{diagram}[small, tight]
\rHC_n(A)   &  \rTo^{S} & \rHC_{n-2}(A) \\
\dTo^{\tTr_V}&                 & \dTo_{\tTr_V} \\
\tH_n(A,\,V) &  \rTo^{S_V}     & \tH_{n-2}(A,V)
\end{diagram}
\qquad \qquad \quad
\begin{diagram}[small, tight]
\rHC_n(A)   &  \rTo^{B} & \overline{\HH}_{n+1}(A) \\
\dTo^{\Tr_V}&                 & \dTo_{\Tr_V} \\
\bar{\H}_n(A,\,V) &  \rTo^{B_V}     & \H_{n}(\Omega^1 A, V)
\end{diagram}
\end{equation}
(The rightmost trace in the second diagram is defined as in \eqref{S5E5} for $ M = \Omega^1 A $.)
Finally, there exists a long exact sequence
\begin{equation*}\la{chom}
\ldots \to \bar{\H}_n(A,V) \oplus \H_{n-1}(\Omega^1 A, V) \xrightarrow{I_V} \tH_n(A,V) \xrightarrow{S_V} \tH_{n-2}(A,V)
\xrightarrow{B_V} \bar{\H}_{n-1}(A,V) \oplus \H_{n-2}(\Omega^1 A, V) \to \ldots
\end{equation*}
which is related to the Connes periodicity sequence \eqref{ISB} by the trace mps in \eqref{S5D3}.
It is suggestive to call $ \tH_\bullet(A,V) $ the {\it cyclic representation homology} of $A$.

\subsection{Relation to Andr\`e-Quillen homology}
\la{RHAQ}
Recall that the Andr\`e-Quillen homology of a commutative algebra $C$ with coefficients in a module $M$ is denoted
$\, D_\bullet(k \bs C, M) \,$ (see Section~\ref{AQh}, Example~3).

Now, fix $ A \in \DGA_k $, and let $\,\pi: R \sonto A \,$ be a semi-free resolution of $A$. Assume that, for some $ V $, the canonical map induced by $ \pi\,$:
\begin{equation}
\la{rab}
\Omega_{\rm com}^1(R_V) \stackrel{\sim}{\to} A_V \otimes_{R_V}
\Omega_{\rm com}^1(R_V)\ \mbox{\rm is a quasi-isomorphism.}
\end{equation}
Then, there is a homological spectral sequence
\begin{equation}
\la{sppp}
E^2_{pq} = D_p(k\bs A_V,\,\H_q(A, V))\  \Rightarrow \  
\H_n(\Omega^1 A, V)
\end{equation}
converging the repreresentation homology of $ \Omega^1 A $.
Indeed, applying $ \L \Omega_{\rm com}^1( \mbox{--}/A_V) \circ  \L(\mbox{---})_V  $ to the DG algebra $R$,
we have isomorphisms in the derived category of DG $R_V$-modules:
\begin{equation*}
\la{lciso}
\L \Omega_{\rm com}^1(R_V/A_V)\, \cong\, 
\Omega_{\rm com}^1(R_V/A_V) \, := \, A_V \otimes_{R_V}
\Omega_{\rm com}^1(R_V) \,\cong\, \Omega_{\rm com}^1(R_V) \,\cong\,
(\Omega^1 R)_V^{\rm ab}\,\cong\,\L(\Omega^1 A)^{\rm ab}_V\ .
\end{equation*}
Here, the first isomorphism follows from the fact $ R_V $
is semi-free in $ \cDGA_k $ whenever $ R $ is semi-free in $ \DGA_k $ ({\it cf.} Theorem~\ref{comp}), the second isomorphism is a consequence of \eqref{rab} and the third isomorphism is given in \eqref{omvdb} and the last again follows from the fact that
$R$ is semi-free so that $ \pi: R \sonto A $ induces a semi-free resolution $ \Omega^1 R \sonto \Omega^1 A $ in the category of
$R$-bimodules. Hence, we have the Grothendieck spectral sequence
$$
\L_p \Omega_{\rm com}^1( \mbox{--}/A_V) \circ  \L_q(A)_V \  \Rightarrow  \  \L_{p+q}(\Omega^1 A)^{\rm ab}_V
$$
which is precisely \eqref{sppp}. We conclude with the following

\vspace{1ex}

\noindent
{\bf Example.} Let $A  $ be a formally smooth algebra in $ \Alg_k $ (see Section~\ref{SAl} below). Assume that $A$ has a
semi-free resolution $ R \sonto A $ that is finitely generated in each degree. Then, by (the proof of) Theorem~\ref{abt1}, we have a quasi-isomorphism $ R_V \sonto A_V $ which implies \eqref{rab}.
Hence, the spectral sequence \eqref{sppp} exists in this case.
Now, we actually have $\,\H_q(A,V) = 0 \,$ for all $ q > 0 $, while $\, \H_0(A,V) = A_V \,$.
On the other hand, if $A$ is formally smooth in $ \Alg_k$, then $A_V$ is formally smooth in the category of commutative $k$-algebras.
This implies that $\, D_p(k\bs A,\,\mbox{---}\,) = 0 \,$ for all $ p > 0 $ (see \cite{L}, Theorem~3.5.6). Thus, 
the spectral sequence  \eqref{sppp} collapses, giving isomorphisms
$\, \H_0(\Omega^1 A, V) \cong \Omega_{\rm com}^1(A_V) \,$ and $ \H_n(\Omega^1 A, V) = 0 $ for all $ n > 0 $.

\section{Examples}
\la{S6}
In this section, we will give a number of examples and explicit computations.
We will focus on two classes of algebras: noncommutative complete intersections
and Koszul algebras for which there are known `small' canonical resolutions.
We begin with a particularly simple class of algebras that are models
for smooth spaces in noncommutative geometry (see \cite{KR}).

\subsection{Smooth algebras}
\la{SAl}
Recall that a $k$-algebra $A $ is called {\it formally smooth} (or {\it quasi-free}) if either of the following equivalent conditions holds (see \cite{CQ, KR}):
\begin{enumerate}
\item
$A$ has cohomological dimension $ \leq 1 $ with respect to Hochschild cohomology.
\item
The universal bimodule $ \Omega_k^1 A $ of derivations is a projective bimodule.
\item
$A$ satisfies the lifting property with respect to nilpotent extensions in $ \Alg_k \,$:
{\it i.e.} for every algebra homomoprhism $ f: A \to B/I $,
where $\, I \vartriangleleft B \,$ is a nilpotent ideal, there is an algebra
homomorphism $ \tilde{f}: A \to B $ inducing $f $.
\end{enumerate}
A formally smooth algebra is called {\it smooth} if it is finitely generated.
It is easy to see that a formally smooth algebra is necessarily hereditary
(\cite{CQ}, Proposition~6.1),  but a hereditary algebra may not be formally smooth
(e.g., the Weyl algebra $ A_1(k)$).
Here are some well-known examples of smooth algebras:
\begin{itemize}
\item
Finite-dimensional separable algebras.
\item
Finitely generated free algebras.
\item
Path algebras of (finite) quivers.
\item
The coordinate rings of  smooth affine curves.
\item
If $ G $ is a f.g. discrete group, its group algebra $ k G $  is smooth iff
$G$ is virtually free ({\it i.e.}, $G$ contains a free
subgroup of finite index), see \cite{LeB}.
\end{itemize}
The class of formally smooth algebras is
closed under some natural constructions: for example,
coproducts and (universal) localizations of formally smooth
algebras are formally smooth.

The key property of (formally) smooth algebras is given by
the following  well-known theorem (see, e.g,, \cite{G}, Proposition~19.1.4).
\begin{theorem}
 \la{abt0}
If $ A $ is a (formally) smooth algebra, then $ \Rep_V(A)$ is
a (formally) smooth scheme for every finite-dimensional
vector space $V$.
\end{theorem}
In other words, Theorem~\ref{abt0} says that the representation
functor $ \Rep_V $ preserves (formal) smoothness.
This can be explained by the following
\begin{theorem}
\la{abt1}
Let $A$ be a formally smooth algebra. Assume that $A$ has a
semi-free resolution in $ \DGA_k^+ $ that is finitely generated in each degree. Then, for any finite-dimensional vector space $V$,
\begin{equation}
\la{vanrep}
\H_n(A,V) = 0 \ , \quad \forall\, n > 0 \ .
\end{equation}
\end{theorem}

\vspace{1ex}

\begin{remark} It is natural to ask whether the vanishing condition  \eqref{vanrep} characterizes formally smooth algebras: that is,
does \eqref{vanrep} imply that $ A $ is formally smooth? The answer to this question is `no.'
A counterexample will be given in Section~\ref{dqcom}.
\end{remark}

\begin{proof}
The proof of Theorem~\ref{abt1} is based on Proposition~\ref{Kapp}
of  Section~\ref{2.8}.  We will use the notation and terminology
introduced in that section.  Let $p\,:\,R \sonto A$ be a semifree resolution
of $A$ in $ \DGA_k^+ $ that is finitely generated in each degree.
Then, $R_V$ defines a smooth affine DG scheme which, abusing notation,
we denote  $\, \DRep_V(A) \,$. By Theorem~\ref{S2T4},
$\, \pi_0( \DRep_V(A))\,\cong \,\mathrm{Rep}_V(A)$.
On the other hand, $ \pi_0(\Spec(A_V))= \Spec(A_V)$
 is indeed the same as $\Rep_V(A) $, and the latter scheme
is smooth by Theorem~\ref{abt0}.
Furthermore, by Proposition~\ref{P2},
for any $ \rho \,\in\, \Rep_V(A) $,
$$
\pi_i(\DRep_V(A), \rho)\, \cong\, \HH^{i+1}(A, \End \,V)\,,\,\,\,\, \forall \,i \geq 1\,\text{.}
$$
Since $A$ is formally smooth, it follows that
$\pi_i(\DRep_V(A), \rho) = 0 $  for all $i \geq 1$. Thus, the differential
$dp_{\rho}$ is a quasi-isomorphism of tangent spaces
$ \T_{\rho}\DRep_V(A)_\bullet $ to  $ \T_{\rho}\Rep_V(A) $ for each representation
$\rho$ of $A$ in $\End\,V$.  Now,  from Proposition~\ref{Kapp}
it follows that $R_V$ is quasi-isomorphic to $A_V$ (via $p_V$).
Since $\mathrm{H}_n(A,V) \,\cong\,\mathrm{H}_n(R_V)$ for all $n$,
the desired result follows.
\end{proof}

We call an algebra $A$ {\it representation cofibrant} if $\mathrm{H}_n(A,V)$
vanishes for all {\it positive} $n$ and for each finite-dimensional
$k$-vector space $V$. The following result is analogous to
the fact that a resolution by acyclic sheaves suffices to compute sheaf cohomology.

\begin{proposition}
\la{abt2}
 Let $B\,\in\,\DGA^+_k$. Suppose $S \rar B$ is a resolution of $B$ by a DG algebra $S$ that
is an extension of a {\it representation cofibrant} algebra
$A$ by an honest cofibration, then, for any finite-dimensional vector space $V$,
 $$ \mathrm{H}_n(S_V)\,\cong\, \mathrm{H}_n(B, V)\,\text{.}$$
\end{proposition}

\begin{proof}
 Let $A \rar S$ be the given cofibration. Consider a cofibrant resolution $R \rar A$ of $A$ and note that the composite map
$R \rar A \rar S$ makes $S$ an object in $\DGA_R$. Let $R \backslash T$ be a cofibrant resolution of $R \backslash S$ in $\DGA_R$.
Consider the pushout $U\,:=\,A \amalg_R T$ in $\DGA^+_k$. We claim that $U$ is quasi-isomorphic to $S$
 (via the natural map $U \rar S$ arising out f the universal property of $U$)
, and hence, to $B$. Indeed, since the model category $\DGA^+_k$
is proper ({\it cf.}~\cite[Proposition B.3]{BKR}), the morphism $T \rar U$ (coming from the pushout diagram) is a quasi-isomorphism.
Since the resolution $T \rar S$ is equal to
the composition $T \rar U \rar S$, $U \rar S$ is indeed a quasi-isomorphism. Further, since $A \rar U$ is the pushout of a cofibration, it is a cofibration.
Thus, $p\,:\,U \rar S$ is a quasi-isomorphism between cofibrant objects in $\DGA^+_A$. Since $T \rar S$ is a fibration, so is $U \rar S$.
Thus, one obtains a homotopy inverse $i\,:\,S \rar U$ of $p$ in $\DGA^+_A$. By~\cite[Proposition B.2]{BKR},  $ip$ is homotopic to the identity via an
M-homotopy (while $pi=\mathrm{Id}_S$). Thus, $i_V$ and $p_V$ are quasi-isomorphisms.
It therefore suffices to check that $$\mathrm{H}_n(U_V)\,\cong\, \mathrm{H}_n(B,V) \,\text{.}$$ By definition,
$\mathrm{H}_n(B,V) \cong \mathrm{H}_n(T_V)$. Since the functor $(\mbox{--})_V$ preserves cofibrations and pushout diagrams,
$$ U_V\,\cong\, A_V \amalg_{R_V} T_V $$
in $\cDGA^+_k$. Since $R_V \rar A_V$ is a quasi-isomorphism (as $A$ is representation cofibrant), and since the model category
$\cDGA^+_k$ is proper,  $T_V \rar U_V$ is a quasi-isomorphism in $\cDGA^+_k$. This proves the desired result.
\end{proof}

\subsection{Noncommutative complete intersections}
Let $ F \in \Alg_k $ be a smooth algebra (e.g., the tensor algebra of a finite-dimensional
vector space), and let $ J $ be a finitely generated 2-sided ideal of $ F $.

\vspace{1ex}

\begin{definition}
The algebra $ A = F/J $ (or the pair $ J \subseteq F $) is called a {\it noncommutative complete intersection} (for short, NCCI) if $ J/J^2 $ is a projective bimodule over $A$.
\end{definition}

\vspace{1ex}

This class of algebras has been studied, under different names, by different authors
(see, e.g., \cite{AH, A, GSh, Go, EG}).
In the present paper, we will use the notation and terminology of \cite{EG}.
As in \cite{EG}, we will work with graded connected algebras equipped with a non-negative polynomial grading. Such an algebra $A$ can be presented as the quotient of a free algebra generated by a finite set of homogeneous variables by the two-sided ideal generated by a finite collection of homogeneous relations. In other words, we may write
\begin{equation}
\la{nci}
A = T_k\,V/\langle j(L) \rangle
\end{equation}
where $V$ is a positively graded $k$-vector space of finite total dimension and $L$ is a finite-dimensional positively graded $k$-vector space equipped with
an homomorphism $\,j :\,L \rar T_k\,V\,$ of graded $k$-vector spaces (which can be chosen to be an embedding). Following \cite{EG}, we refer to the triple $(V,L,j)$ as  {\it presentation data} for $A$. It is easy to show that an algebra $A$ of the form \eqref{nci} is NCCI
if and only if it has cohomological dimension $ \leq 2 $ with respect to Hochschild
cohomology  (see \cite[Theorem~3.1.1]{EG}). The class of (graded) noncommutative complete intersections is thus a natural extension of the class of smooth algebras.

Associated to the data $\,(V,L,j)\,$ there is a non-negatively graded DG algebra defined as follows.
Place $V$ in homological  degree $0$ and place $L$ in homological degree $1$ to obtain the $k$-vector space $V \oplus L[1]$ (which is graded homologically as well as polynomially). Then define the bigraded algebra $ T_k(V \oplus L[1])$ and put on it a (unique) differential $d$ such that
$$ d(l)\,=\,j(l)\,\in T_k\,V $$ for all $l$ in $L[1]$.  The resulting DG algebra is denoted $\sh(A,(V,L,j))$ and called
the {\it Shafarevich complex}\footnote{This complex was originally introduced in \cite{GSh} in connection with the famous Golod-Shafarevich Theorem. We recommend
the survey paper \cite{Pi}, where this connection as well as many other
interesting applications of the Shafarevich complex are discussed.}
corresponding to  $(V,L,j)$. Note that $\mathrm{H}_0(\sh(A,(V,L,j)))\,\cong\,A$.
\begin{theorem}[see \cite{A, Go, EG}]
A (graded connected) algebra $A$ is NCCI iff it has presentation data
$(V,L,j)$ such that the associated Shafarevich complex $\sh(A, (V,\,L,\,j))$ is acyclic in all positive degrees.
\end{theorem}

Using the Shafarevich complex, we can study the representation homology of NCCI algebras.
To avoid confusion with the data $\,(V,L,j)\,$,  we will consider representations of $A$ on a vector space
$ k^n $, $\, n \ge 1 $. The corresponding representation functors will then be denoted by
$ \Rep_n(A) $ and $\, \DRep_n(A)$ instead of  $ \Rep_{k^n}(A) $ and $\, \DRep_{k^n}(A) $.

Recall that a given DG algebra  $R \in \DGA^+_k$ has a universal DG representation
$ \pi_n\,:\,R \to  \M_n(R_n)$ defined for each $ n \ge 1 $. For a matrix $M$ with entries in a DG-algebra $S$, we denote the entry in row $i$ and column $j$ by $M_{ij}$. For notational brevity, we shall denote the vector space
$X \otimes \M_n(k)$ by $X_n$ for any $k$-vector space $X$. Let $j_n\,:\,L_n \rar (T_k\,V)_n$ denote the map
$$l_{pq}\,:=\, l \otimes \mathrm{e}_{pq}  \mapsto (\pi_n(j(l)))_{pq}\,\text{.}$$
Further, recall that for a finitely generated (polynomially graded) commutative algebra $k$-algebra $B$, given finite dimensional (polynomially graded) vector spaces $W\,,S$ and a homomorphism $f\,:\,S \rar \bSym(W)$ such that $B\,=\,  \bSym(W)/(f(S))$, one can construct the Koszul complex $\mathrm{K}(B,(W,S,f))\,:=\,\bSym(W \oplus S[1])$ equipped with the homological differential mapping each $s \in S$ to $f(s)$.
It turns out that the representation functor transforms Shafarevich complexes to Koszul complexes.
Indeed, with our notation, the following lemma is an immediate consequence  of Theorem~\ref{comp}.

\begin{lemma} \la{ShNCC}
Let $(V,L,j)$ be presentation data for $A$. Then,
$$(\sh(A,(V,L,j)))_n\,\cong\, \mathrm{K}(A_n,\,(V_n\,,L_n\,,j_n))\,\text{.}$$
\end{lemma}
This lemma suggests that we should indeed view a Shafarevich complex as a noncommutative Koszul complex.
The next theorem shows that the representation homology of NCCI algebras is {\it rigid} in the sense of
Auslander-Buchsbaum (see \cite{AB}).
\begin{theorem}
\la{Rigid}
 If $A$ is a NCCI algebra, then $\mathrm{H}_q(A, k^n)=0$ implies that $\mathrm{H}_p(A, k^n)=0$ for all $\, p \geq q \,$.
\end{theorem}
\begin{proof}
This follows from Lemma~\ref{ShNCC} and the rigidity of the usual Koszul complexes (see \cite{AB}, Proposition~2.6).
\end{proof}

The following theorem gives a natural interpretation for
the 1-st representation homology of NCCI algebras:
namely, $  \H_1(A, k^n) $ is an obstruction for the classical representation scheme $ \Rep_n(A) $ to be a complete interesection.
\begin{theorem}
 \la{RepHom}
Let $A$ be a NCCI algebra.  Assume that $ \H_1(A, k^n)=0 $ for some $ n \ge 1 $.
Then $ \Rep_n(A) $ is a complete intersection.
\end{theorem}
\begin{proof}
Suppose that $(V,L,j)$ is presentation data for $A$ making $\sh(A,(V,L,j))$ acyclic in positive degree. Then,
the Koszul complex $\mathrm{K}\,:=\,\mathrm{K}(A_n,(V_n,L_n,j_n))$ represents $\mathrm{DRep}_n(A)$ in $\Ho(\cDGA_k)$
(by Lemma~\ref{ShNCC}). Suppose that $\mathrm{H}_1(A,n)=0$. Then, by~\cite[Proposition 2.6]{AB}, the
Koszul complex $\mathrm{K}$ is acyclic in all positive degrees. Since $k[\mathrm{Rep}_n(A)] \,\cong\,\mathrm{H}_0(\mathrm{K})$,
$\mathrm{Rep}_n(A)$ is a complete intersection.
\end{proof}

Under extra (mild) assumptions, the vanishing of  $ \H_1(A, k^n) $ is not only sufficient but also necessary for  $ \Rep_n(A) $ to be a complete intersection. More precisely, we have
\begin{theorem}
\la{RepHom1}
Let $A$ be a NCCI algebra given with presentation data $(V,L,j)$  such that $\,\sh(A,(V,L,j))$ is acyclic in positive degrees.

$(a)$  If  $ \Rep_n(A)$ is a complete intersection in $\,\Rep_n(T_k\,V)\,$ of dimension $\,n^2(\mathrm{dim}_k\,V -\mathrm{dim}_k\,L)\,$, then
$$
\H_p (A, k^n)=0\ ,\quad \forall\, p > 0 \ .
$$

$(b)$  More generally, if  $\Rep_n(A)$ is a complete intersection in $ \Rep_n(T_k\,V)$, then
$\mathrm{H}_q(A, k^n)$ is a free module over $\mathrm{H}_0(A,k^n)$ of rank $\binom{p}{q}$, where $ p\,:=\,\dim\,\mathrm{Rep}_n(A)- n^2(\dim_k\, V-\dim_k\,L)$.
\end{theorem}
\begin{proof}
If $\mathrm{Rep}_n(A)$ is a complete intersection in $\mathrm{Rep}_n(T_k\,V)$ implies that the Koszul complex $\mathrm{K}$ is acyclic in positive degrees. Since
$\mathrm{H}_q(A,k^n)\,\cong\,\mathrm{H}_q(\mathrm{K})$, choose a homogenous basis of $L_n$ and choose a minimal set $S$ from this homogenous basis
such that its image under $j_n$ generates the ideal $I_n$ defining $\mathrm{Rep}_n(A)$ in $\mathrm{Rep}_n(T_k\,V)$. The $k$-linear span
of $S$ is a graded subspace $L^{o}$ of $L_n$, and $\dim_k \,\mathrm{Rep}_n(A)\,=\,n^2.\dim_k \,V - \dim_k \,L^{o}$. For any complement
$L^{\perp}$ of $L^o$ in $L$, $j(L^{\perp})$ is contained in the ideal $I_n$. It follows from~\cite{Ei} that $\mathrm{K}$ is quasi-isomorphic to
$\mathrm{K}(A_n,(V_n,L^o,j_n)) \otimes \bSym(L^{\perp}[1])$. Since $\mathrm{Rep}_n(A)$ is a complete intersection, $\mathrm{K}(A_n,(V_n,L^o,j_n))$
is acyclic in positive degrees and has $0$-th homology $k[\mathrm{Rep}_n(A)]$. Thus, $\mathrm{H}_q(A, k^n) \cong k[\mathrm{Rep}_n(A)] \otimes \Lambda^q L^{\perp}$ as
(polynomially graded) vector spaces. Finally, note that the number $p$ in the statement of (c) is precisely $\dim_k\,L^{\perp}$. This proves
$(b)$, of which $(a)$ is a special case.
\end{proof}

Let $A$ be a NCCI algebra with presentation data $(V,L,j)$ as in Theorem~\ref{RepHom1}. Set $R\,:=\,\sh(A,(V,L,j))$ and denote the
summand of polynomial degree $r$ in $R_p$ by $ R_p{r} $ (with square brackets being reserved for denoting shifts in homological degree).
Since $ R $ is acyclic in homological degrees $ p > 0 $, the map
$ j: L \to T_k V $ is injective (see \cite{Pi}, Theorem~2.4).
Consider the graded subspace $ L_0\,:=\,j^{-1}([T_k\,V,\,T_k\,V])$ of $L$. The embedding $L_0 \hookrightarrow R_1$ induces a linear map
$$ \phi\,:\, L_0 \rar \mathrm{H}_1(\bar{R}_{\natural}) \cong \overline{\HC}_1(A) \,\text{.}$$
 Consider the restriction of the
map $\mathrm{Tr}_n\,:\,R \rar R_n$ to $L_0$. Clearly, $\mathrm{Tr}_n\,|_{L_0}$ is injective. We may therefore,
identify $L_0$ with its image under $\mathrm{Tr}_n$ and choose a direct sum decomposition
$$L_n\,\cong\, L_0 \oplus L_0^{\perp} $$ as graded $k$-vector spaces. The following proposition now follows
from Lemma~\ref{ShNCC}.

\begin{prop} \la{RepNCC}
With above notation, there is an isomorphism of DG algebras
$$ R_n \,\cong\, \mathrm{K}(A_n, (V_n, L_0^{\perp}, j_n)) \otimes \bSym(L_0[1]) \,\text{.}$$
Consequently,
$$\mathrm{H}_{\bullet}(A, k^n) \,\cong\, \mathrm{H}_{\mathrm{Koszul},\bullet}(A_n, (V_n,L_0^{\perp}, j_n)) \otimes \bSym(L_0[1])\,\text{.}$$
\end{prop}

When the graded vector space $L$ is concentrated in a single degree and when $n>1$, one can further show (using the 2nd
Fundamental Theorem of Invariant Theory) that
the images of any basis of $L_0^{\perp}$ form a minimal generating set for the ideal defining $A_n$ in $(T_k\,V)_n$.
Hence, in this case, the Koszul homology $\mathrm{H}_{\mathrm{Koszul},\bullet}(A_n, (V_n,L_0^{\perp}, j_n))$ is literally the Koszul homology for the embedding $\mathrm{Rep}_n(A) \hookrightarrow \mathrm{Rep}_n(T_k\,V)$ of schemes.

\subsubsection{Derived commuting schemes} Let $\, A = k[x,y] \,$ be the polynomial algebra of two variables. For $ n \ge 1 $, the
representation scheme $ \Rep_n(A) $ is called the {\it $n$-th commuting scheme}. We write $ A_n = k[x,y]_n $ for the
corresponding commutative algebra.  It is not known whether $\Rep_n(A) $ is a reduced scheme in general but it is known
that the underlying variety is irreducible for all $ n \,$ (see \cite{Ger}). The following result is a consequence of a
deep theorem of A.~Knutson \cite{Kn}.

\begin{theorem}
\la{RepHompoly}
 $\H_p(A,k^n)\,=\,0\, $ for all $\,p>n$.
\end{theorem}
\begin{proof}
The obvious presentation $ A =  k\langle x,y\rangle/(xy-yx) $
with natural polynomial grading ($ \deg(x)= \deg(y) = 1 $) shows that $A$ is actually a NCCI algebra ({\it cf.} \cite{Pi}, Proposition~2.20). Indeed, for $\,V := k.x \oplus k.y\,$, $\,L :=k.t$ (with $t$ in polynomial degree $2$) and $j(t)\,:=\,xy-yx$, the Shafarevich complex is isomorphic to
the DG algebra $ R\,:=\,k \langle x,y,t\,:\, dt=xy-yx \rangle $ which is acyclic in positive degrees. Thus, $R_n\,\cong\, k[x_{ij},y_{ij},t_{ij}\,|\,1 \leq i,j \leq n]$ with variables $t_{ij}$ in degree $1$ and differential determined by the formula
$$
dT\,=\,[X,Y]\ ,
$$
where $ X :=(x_{ij})\,$,$ \,Y:=(y_{ij})\,$, $\, T:=(t_{ij})\,\in \M_n(R_n)\,\text{.}$
By~\cite[Theorem 1]{Kn}, the $ (n^2-n) $ elements $\, \{dt_{ij}\,, 1 \leq i\neq j \leq n\} \,$ form a regular sequence in $\,k[x_{ij},y_{ij}]\,$.
It follows from~\cite[Corollary 17.12]{Ei} that $\H_p(R_n)\,=\,0$ for all $p>n$.
\end{proof}

\vspace{1ex}

\noindent
{\bf Example}\,($n=1$).
It is easy to see that
$$
\H_{\bullet}(k[x,y], k) \,\cong\, k[x,y] \otimes \bSym(k.t)
$$
where $t$ has degree $1$. Hence, $\,\H_{\bullet}(k[x,y], k)\cong k[x,y] \oplus k[x,y].t\,$ is a rank 2 free module over $ k[x,y]$.

\vspace{1ex}

This simple example shows that $ \DRep_V(A) $ does depend on the algebra $A$, and not only on the affine scheme $ \Rep_V(A)$.
Indeed, comparing $ k[x,y]$ to the free algebra $ k\langle x,y \rangle $, we see that $ \Rep_1(k[x,y]) = \Rep_1(k\langle x,y \rangle)$ but
$ \H_1(k[x,y], k) \not\cong \H_1(k\langle x,y \rangle, k) $ because
$ \H_1(k\langle x,y \rangle, k) = 0 $, by Theorem~\ref{abt1}.

\vspace{1ex}

\noindent
{\bf Example}\,($n=2$).
The algebra $\H_{\bullet}(k[x,y], k^2)$ is more complicated.  Let $\mathfrak{g}\,:=\, \mathrm{span}_k\{ \xi, \tau,\eta\}$,
where the variables $\,\xi,\,\tau,\,\eta\,$ are in homological degree $1$.  Then, there is an isomorphism of graded algebras
\begin{equation}
\la{reph2}
\mathrm{H}_{\bullet}(k[x,y], k^2)\,\cong\, (k[x,y]_2 \otimes \Lambda_k \mathfrak{g})/\mathbf{I}
\end{equation}
where the ideal $\mathbf{I}$ is generated by the following relations\\

(1)\quad $x_{12}\,\eta - y_{12}\,\xi \,=\, (x_{12}y_{11}-y_{12}x_{11})\,\tau$\\

(2)\quad $x_{21}\,\eta - y_{21}\,\xi \,=\, (x_{21}y_{22}-y_{21}x_{22})\,\tau$\\

(3)\quad $(x_{11}-x_{22})\,\eta- (y_{11}-y_{22})\,\xi \,=\, (x_{11}y_{22}- y_{11}x_{22})\,\tau$\\

(4)\quad $\xi \eta\,=\, y_{11}\,(\xi\tau) - x_{11}\,(\eta\tau) \, =\, y_{22}\,(\xi\tau)-x_{22}\,(\eta\tau) $\\

\noindent
Thus, as a $\mathrm{H}_0$-module, $\mathrm{H}_{\bullet}(k[x,y],k^2)\,\cong\,H_0 \oplus H_1 \oplus H_2\,$, where
\begin{eqnarray*}
H_0 &\cong&  k[x,y]_2\\*[2ex]
H_1 &\cong& (H_0 \cdot \xi \oplus H_0 \cdot \tau \oplus H_0 \cdot \eta)/(\text{relations (1)-(3)})\\*[2ex]
H_2 \,&\cong&\, (H_0 \cdot \xi\tau \oplus H_0 \cdot \eta\tau )/(x_{12}\eta\tau-y_{12}\xi\tau\,,\,x_{21}\eta\tau-y_{21}\xi\tau\,,\, (x_{11}-x_{22})\eta - (y_{11}-y_{22})\xi)
\end{eqnarray*}
The above presentation of $\H_{\bullet}(k[x,y], k^2)$  was obtained with an assistance of {\tt Macaulay2}.

\vspace{1ex}

Recall that for $ A = k[x,y] $, the cyclic homology ${\rHC}_i(A) = 0 $ for $i>1$, while $ \HC_0(A)=A $ and
$\overline{\HC}_1(A)= \Omega^1A/dA $. With these identifications, for all $\, n\, $, the $0$-th trace
$ \mathrm{Tr}(A)_0 : \, \HC_0(A) \rar \mathrm{H}_0(A,k^n)$ is obviously given by the formula
$\, x^l y^m \mapsto \mathrm{Tr}(X^l Y^m)$, while the $1$-st trace
$\,\mathrm{Tr}(A)_1\,:\,\overline{\HC}_1(A) \rar \H_1(A, k^n) $ is expressed by (see \cite[Example 4.1]{BKR}):
$$
\mathrm{Tr}_1(x^l y^m dx)\,=\, \sum_{i=0}^{m-1} \, \mathrm{Tr}\,(X^l Y^{i}T Y^{m-1-i})\ , \qquad
\mathrm{Tr}_1(x^l y^m dy) \,=\, - \,\sum_{j=0}^{l-1}\,\Tr(X^j T X^{l-1-j} Y^m)\ .
$$
Now, for $ n = 2 $, the  generators $ \xi $, $\eta$ and $\tau $ in \eqref{reph2} correspond  to the classes of the elements
$\, \Tr(XT) $, $\, \Tr(YT) $ and $\,\Tr(T)$. Thus, we see that
$$
\tau = \mathrm{Tr}_1(ydx) \ ,\quad  \xi\,=\, \mathrm{Tr}_1(xydx)\ , \quad  \eta\,=\,\mathrm{Tr}_1(xydy)\ .
$$
It follows that  $ \H_{\bullet}(k[x,y], k^2) $ is generated (as an algebra over $\mathrm{H}_0$) by invariant traces of degree $1$.

\vspace{1ex}

\subsubsection{Derived $q$-commuting schemes}
\la{dqcom}
For a parameter $ q \in k^* $, define  $ A\,:=\, k \langle x, y\rangle/\langle xy-qyx \rangle $. By~\cite[Proposition 2.20]{Pi},
this algebra is a NCCI whose Shafarevich resolution is given by $  k\langle x,y,t\,|\,dt = xy-qyx\rangle $ . In the case when
$q$ is {\it not} a root of unity,  \cite[Proposition 5.3.1]{EG} shows that $ \mathrm{Rep}_n(A) $ is a complete intersection of
dimension $ n^2 $ for all $ n \ge 1 $. By Theorem~\ref{RepHom1}, we conclude that $ \H_p(A, k^n ) = 0 $ for all $ p > 0 $,
{\it i.e.} $A$ is a representation cofibrant algebra. However, by~\cite[Theorem 5.3]{Di}, the global dimension of $ A $ is equal to $2$.
Hence, $A$ is not formally smooth.

More generally,  for parameters $\, q_1,\,q_2,\,\ldots\,, q_{m-1} \in k^* $, we can define ({\it cf.} \cite[Example 5.3.3]{EG})
$$
A = k \langle x_1,\ldots,x_m \rangle/\langle \mathrm{ad}_{q_1}(x_1)\, \ldots \, \mathrm{ad}_{q_{m-1}}(x_{m-1})x_m \rangle
$$
where $ \mathrm{ad}_q(x)(y)\,:=\, xy-qyx$. If all $ q_i$'s are not roots of unity, then $ \mathrm{Rep}_n(A) $ is
a complete intersection of dimension $n^2(m-1)$. Again, Theorem~\ref{RepHom1} implies the vanishing of the
higher representation homology of $A$, while \cite[Theorem 5.3]{Di} shows that $A$ is not formally smooth.

\subsection{Koszul algebras}
\la{dualnm}
For any Koszul algebra $A$ with quadratic linear relations, there is a canonical semifree
resolution given in terms of the cobar construction of the dual
coalgebra $ (A^{!})^*$ (see \cite[Chapter 3]{LV}).  We illustrate the use of this resolution in three examples.

\subsubsection{Dual numbers}
Let  $ A\,:=\,k[x]/(x^2) $ be the ring of dual numbers. This is a quadratic algebra which is Koszul dual to the tensor algebra $ T_k V$
of a one-dimensional vector space $ V $. It has a minimal semi-free resolution of the form $\, R\,:=\, k\langle x, t_1, t_2, t_3, \ldots \rangle\,$
where $ \deg(x)=0 $  and $ \deg(t_p) = p $, and the differential is given by
$$
d t_p\,=\, xt_{p-1}-t_1t_{p-2} + \ldots +{(-1)}^{p-1}t_{p-1}x \ ,\quad p \ge 1 \ .
$$
By Theorem~\ref{comp}, $ \mathrm{H}_{\bullet}(A, k^n) $  is then the homology of the commutative DG algebra
$$
R_n\,:=\, k [x_{ij}, (t_1)_{ij},  (t_2)_{ij}, \, \ldots \ |\ 1 \leq i,j \leq n] \ ,
$$
whose  differential  in the matrix notation is given by
$$
dT_p \,=\, XT_{p-1}-T_1T_{p-2}+ \ldots +{(-1)}^{p-1}T_{p-1}X \ .
$$
For $ n  = 1 $,  using  {\tt Macaulay2}, we find  \\

$\qquad \mathrm{H}_0(A,k)\,\cong\,A$\\

$\qquad \mathrm{H}_1(A,k)\,=\, 0$\\

$\qquad \mathrm{H}_2(A,k)\,\cong\, A \cdot t_2$\\

$\qquad \mathrm{H}_3(A,k) \,\cong\, A \cdot (xt_3-2t_1t_2)$\\

$\qquad \mathrm{H}_4(A,k)\,\cong\, A \cdot t_2^2\,\oplus\, A \cdot t_4$\\

$\qquad \mathrm{H}_5(A,k)\,\cong\, A \cdot (-2t_1t_2^2+xt_2t_3) \,\oplus \,A \cdot (-t_2t_3-4t_1t_4+ 2xt_5)$\\

$\qquad \mathrm{H}_6(A,k)\,\cong\, A \cdot t_2t_4\, \oplus \,A \cdot t_6$\\

$\qquad \mathrm{H}_7(A,k)\,\cong\, A \cdot (-t_2^2t_3-4t_1t_2t_4+2xt_2t_5) \,\oplus \, A \cdot (-t_3t_4-2t_1t_6+xt_7)\,$\\

$\qquad \mathrm{H}_8(A,k)\,\cong\, A \cdot t_2 t_6 \oplus A \cdot t_4^2 \oplus A \cdot
t_8 \,$\\

$\qquad  .\ .\ .\ .\ .\ .\ .\ .\ .\ .\ .\ .\ .\ .\ .\ .\ .\ .\ .\ .\ .\ .\ .\ .\ .\ .\   $\\

\noindent
The (reduced) cyclic homology of $ A $ is given by (see, e.g., \cite{LQ}, Section~4.3):
\begin{equation*}
\rHC_n(A) =
\left\{
\begin{array}{lll}
0 \ & \mbox{\rm if} &\ n= 2p+1 \\*[1ex]
k. x^{\otimes (2p+1)}\ & \mbox{\rm if} &\ n = 2p
\end{array}
\right.
\end{equation*}
The odd traces $ \Tr(A)_{2p+1} $ thus vanish, while the even ones are given by
$$
\Tr(A)_{2p}\,:\,\overline{\HC}_{2p}(A) \rar \H_{2p}(A, k)\ ,\quad x^{\otimes (2p+1)} \mapsto t_{2p}\ .
$$

This example shows that the algebra map
$\,
\bL \Tr(A)_\bullet :\ \bL[\HC(A)] \to \H_{\bullet}(A, V)^{\GL(V)}
$
is not surjective  in general, {\it i.e.} the Procesi Theorem \cite{P} fails for higher traces ({\it cf.} Section~\ref{stabi}).
Note also that, unlike in the case of NCCI algebras, the representation homology of $ A $ is not rigid in the sense that
$ H_1 = 0 $ does not force the vanishing of higher homology.

\subsubsection{Polynomials in three variables}
Let $ A = k[x,y,z] $ be the polynomial ring in three variables. It has a minimal Koszul resolution of the form
$ R = k \langle x,y,z; \xi, \theta, \lambda; t\rangle $, where the generators $ x, y, z $ have degree $0$;
$ \xi, \theta, \lambda $ have degree $1$ and $t$ has degree $2$. The differential on $R $ is defined by
$$
d\xi = [y,z] \ ,\quad d\theta = [z,x] \ ,\quad d\lambda = [x,y] \ ,\quad dt = [x, \xi] + [y, \theta] + [z, \lambda]\ .
$$
For $ V = k^n $, Theorem~\ref{comp}  implies that
$$
R_n \cong k[x_{ij},\,y_{ij},\,z_{ij};\, \xi_{ij},\, \theta_{ij},\, \lambda_{ij};\,t_{ij}]\ ,
$$
where the generators $ x_{ij} \,$,$\, y_{ij} \,$,$ \,z_{ij} \,$ have degree zero,
$\,\xi_{ij} \,$, $\,\theta_{ij}\, $, $\,\lambda_{ij} \,$ have degree $ 1$,
and $\, t_{ij} $ have degree $ 2 $. Using the matrix notation $ X = \|x_{ij}\| \,$,
$\, Y = \| y_{ij}\|\,$, etc., we can write the differential on $\,R_n\,$ in the form
\begin{equation*}
d\Xi = [Y,Z] \ ,\quad d\Theta = [Z,X] \ ,\quad d\Lambda = [X,Y]\ ,\quad
dT = [X, \Xi] + [Y, \Theta] + [Z, \Lambda]\ .
\end{equation*}
For $ n=1 $, it is easy to see that the homology of $ R_n $ is just a graded symmetric algebra
generated by the classes of $ x,y,z, \xi, \theta, \lambda, t \,$. Thus,
$$
\H_\bullet(A, k) \cong \bSym(x, y, z, \xi, \theta, \lambda, t)
$$
This example shows that, unlike in the case of two variables, the representation homology of the polynomial algebra
$ k[x,y,z] $ is not bounded.

\subsubsection{Universal enveloping algebras}
Let $ A = U(\mathfrak{sl}_{2}) $ be the universal enveloping algebra of the Lie algebra $ \mathfrak{sl}_{2}(k) $.
As in previous example, $A$ has a minimal resolution of the form $ R = k \langle x,y,z; \xi, \theta, \lambda; t\rangle $
with generators $ x, y, z $ of degree $0$; $ \xi, \theta, \lambda $ of degree $1$ and $t$ of degree $2$. The differential on $R $
is defined by
$$
d\xi = [y,z] + x\ ,\quad d\theta = [z,x] + y\ ,\quad d\lambda = [x,y] + z\ ,\quad dt = [x, \xi] + [y, \theta] + [z, \lambda]\ .
$$
For $ V = k^n $, the corresponding algebra $ R_n $ has differential
\begin{equation*}
d\Xi = [Y,Z] + X\ ,\quad d\Theta = [Z,X] + Y\ ,\quad d\Lambda = [X,Y] + Z\ ,\quad
dT = [X, \Xi] + [Y, \Theta] + [Z, \Lambda]\ .
\end{equation*}
For $ n=1 $,  it is easy to see that the homology of $ R_n $ is just the polynomial algebra
generated by one variable $ t $ of degree $2$. Hence $\, \H_\bullet(A,k) \cong k[t] \,$.


\begin{thebibliography}{G}
%
\bibitem[AH]{AH}
J. F. Adams and P. J. Hilton,  \textit{On the chain algebra of a loop space},
Comment. Math. Helvetici \textbf{30} (1956), 305--330.

\bibitem[A]{A}
D.~Anick, \textit{Non-commutative graded algebras and their Hilbert series},
J. Algebra  \textbf{78} (1982),  120--140.
%
\bibitem[AB]{AB}
M.~Auslander and D.~A.~Buchsbaum,
\textit{Codimension and multiplicity}, Ann. of Math. \textbf{68}  (1958), 625--657.
%
\bibitem[BP]{BP}
H.-J. Baues and T.~Pirashvili, \textit{Comparison of MacLane, Shukla and Hochschild cohomology}, J. Reine Angew. Math. \textbf{598} (2006), 25--69.
%
\bibitem[BCHR]{BCHR}
K.~Behrend, I.~Ciocan-Fontanine, J.~Hwang and M.~Rose,
\textit{The derived moduli space of stable sheaves},
{\tt arxiv:1004.1884}.
%
\bibitem[Be]{Be}
Yu.~Berest, \textit{Calogero-Moser spaces over algebraic curves},
Selecta Math. \textbf{14} (2009), 373--396.
%
\bibitem[BC]{BC}
Yu.~Berest and O.~Chalykh, \textit{Ideals of rings of differential operators on
algebraic curves} (with an appendix by G.~Wilson),
J. Pure Appl. Alg. \textbf{216} (2012),  1493--1527.
%
\bibitem[BKR]{BKR}
Yu.~Berest, G.~Khachatryan and A.~Ramadoss,
\textit{Derived representation schemes and cyclic homology}, {\tt arXiv:1112.1449}.
%
\bibitem[BCER]{BCER}
Yu.~Berest,  X.~Chen, F.~Eshmatov and A.~Ramadoss,
\textit{Noncommutative Poisson structures, derived representation
schemes and Calabi-Yau algebras}, Contemp. Math. \textbf{583} (2012), 219--246.
%
\bibitem[BR]{BR}
Yu.~Berest and A.~Ramadoss,
\textit{Stable representation homology and Koszul duality}, {\tt arXiv:1304.1753}.
%
\bibitem[B]{B}
G. M. Bergman,
\emph{Coproducts and some universal ring constructions},
Trans. Amer. Math. Soc. \textbf{200} (1974), 33--88.
%
\bibitem[Bo]{Bo}
K.~Bongartz, \textit{Some geometric aspects of representation theory}
in \textit{ Algebras and modules I},
CMS Conf. Proc., \textbf{23}, Amer. Math. Soc., Providence, RI, 1998, pp. 1--27
%
\bibitem[BG]{BG}
A. K. Bousfield and V.K.A.M. Gugenheim,
\textit{On PL De Rham Theory and
Rational Homotopy Type},  Memoirs Amer. Math. Soc. \textbf{179}, 1976.
%
\bibitem[CF]{CF}
A. Catteneo and G. Felder,  \textit{Relative formality theorem and
quantisation of coisotropic submanifolds},
Adv. Math. \textbf{208}(2) (2007),  521--548.
%
\bibitem[CK]{CK}
I. Ciocan-Fontanine and M. Kapranov,
\emph{Derived Quot schemes}, Ann. Sci. ENS \textbf{34} (2001), 403--440.
%
\bibitem[CK2]{CK2}
I. Ciocan-Fontanine and M. Kapranov,
\emph{Virtual fundamental classes via dg-manifolds},
Geometry and Topology \textbf{13} (2009),  1779--1804.
%
\bibitem[C]{C}
P. M. Cohn, \textit{The affine scheme of a general ring},
Lecture Notes in Math. \textbf{753}, Springer-Verlag, Berlin, 1979,
pp. 197--211.
%
\bibitem[CEG]{CEG}
W.~Crawley-Boevey, P.~Etingof and V.~Ginzburg,
\textit{Noncommutative geometry and quiver algebras}, Adv. Math.
\textbf{209} (2007), 274--336.
%
\bibitem[CB]{CB} W.~Crawley-Boevey, {Poisson structures on moduli spaces of
representations}, J. Algebra \textbf{325} (2011), 205--215.
%
\bibitem[CQ]{CQ}
J.~Cuntz and D. Quillen, \textit{Algebra extensions and nonsingularity},
J. Amer. Math. Soc. \textbf{8}(2) (1995), 251--289.
%
\bibitem[CQ1]{CQ1}
J.~Cuntz and D. Quillen, \textit{Cyclic homology and nonsingularity},
J. Amer. Math. Soc. \textbf{8}(2) (1995), 373--442.
%
\bibitem[Di]{Di}
W.~Dicks, \textit{On the cohomology of one relator associative algebras},
J. Algebra \textbf{97}(1) (1985), 79--100.
%
%\bibitem[DK]{DK}
%V.~Dotsenko and A.~Khoroshkin, \textit{Free resolutons via Gr\"{o}bner bases},
%preprint {\tt arXiv:0912.4985}.
%
\bibitem[DS]{DS} W. G. Dwyer and J. Spalinski,
{\it Homotopy theories and model categories} in \textit{Handbook of Algebraic Topology},
Elsevier, 1995, pp. 73--126.
%
\bibitem[E]{Ei}
D.~Eisenbud,
\textit{Commutative Algebra.  With a View Towards Algebraic Geometry},
Graduate Texts in Mathematics \textbf{150}, Springer-Verlag, New York, 1995.
%
\bibitem[EG]{EG}
P.~Etingof and V.~Ginzburg,
\textit{Noncommutative complete intersections and matrix integrals},
Pure Appl. Math. Q. \textbf{3}(1) (2007), 107--151.
%
\bibitem[FT]{FT} B.~Feigin and B.~Tsygan,
{\it Additive K-theory and crystalline cohomology},
Funct. Anal. Appl. \textbf{19}(2) (1985), 124-–132.
%
\bibitem[FT1]{FT1}
B.~L.~Feigin and B.~L.~Tsygan,
\textit{Additive $K$-theory}, Lecture Notes in Math. \textbf{1289},
Springer, Berlin, 1987, pp. 67--209.
%
\bibitem[FHT]{FHT}
Y.~Felix, S.~Halperin and J.-C.~Thomas,
{\it Differential graded algebras in topology} in
{\it Handbook of Algebraic Topology}, Elsevier, 1995, pp. 829--865.
%
\bibitem[Ga]{Ga}
P. Gabriel, \textit{Finite representation type is open}, in {\it Representations of Algebras},
Lecture Notes in Math. \textbf{488}, Springer, Berlin, 1975, pp. 132--155.
%
\bibitem[Ge]{Ge}
Ch.~Geiss, \textit{Geometric methods in representation theory of finite-dimensional algebras}
in \textit{Representation theory of algebras and related topics}, CMS Conf. Proc.,
\textbf{19}, Amer. Math. Soc., Providence, RI, 1996, pp. 53--63.
%
\bibitem[Ger]{Ger}
M.~Gerstenhaber,
 \textit{On dominance and varieties of commuting matrices},
Ann. of Math. \textbf{73} (1961), 324--348.
%
\bibitem[GM]{GM}
S.~Gelfand and Yu. Manin,
\textit{Methods of Homological Algebra}, Springer, Berlin, 2000.
%
\bibitem[G]{G}
V.~Ginzburg, \textit{Lectures on Noncommutative Geometry},  {\tt arXiv:math.AG/0506603}.
%
\bibitem[G1]{G1}
V.~Ginzburg, \textit{Calabi-Yau algebras}, {\tt arXiv:math.AG/0612139}.
%
\bibitem[G2]{G2} V.~Ginzburg,
\textit{Noncommutative symplectic geometry, quiver varieties and operads},
Math. Res. Lett. \textbf{8} (2001), 377--400.
%
\bibitem[GS]{GS}
P.~Goerss and K.~Schemmerhorn,
\textit{Model categories and simplicial methods} ,
Contemp. Math. \textbf{436} (2007),  3--49.
%
\bibitem[GSh]{GSh}
E. S. Golod and I. R. Shafarevich, \textit{On the class field tower},
Izv. Akad. Nauk SSSR Ser. Mat. \textbf{28} (1964), 261--272.
%
\bibitem[Go]{Go}
E. S. Golod, \textit{Homology of the Shafarevich complex, and
noncommutative complete intersections}.
Fundam. Prikl. Mat. \textbf{5}(1) (1999),  85-–95.
%
\bibitem[He]{He}
K.~Hess, \textit{Model categories in algebraic topology},
Appl. Categ. Structures \textbf{10}(3) (2002),  195-–220.
%
\bibitem[H]{H}
V. Hinich,
\textit{Homological algebra of homotopy algebras},
Comm. Algebra \textbf{25} (1997), 3291--3323.
%
\bibitem[Hir]{Hir}
P. Hirschhorn,
\textit{Model Categories and their Localizations},
Mathematical Surveys and Monographs \textbf{99}, Amer. Math. Soc., 2009.
%
\bibitem[Ho]{Ho}
M.~Hovey, \textit{Model Categories},
Mathematical Surveys and Monographs \textbf{63}, Amer. Math. Soc., 1999.
%
\bibitem[HMS]{HMS}
D. Husemoller, J.~Moore, and J. Stasheff,
\textit{Differential homological algebra and homogeneous spaces},
J. Pure Appl. Algebra \textbf{5} (1974), 113--185.
%
\bibitem[I]{I}
S. Iyengar,
\textit{Andr\'e-Quillen homology of commutative algebras} ,
Contemp. Math. \textbf{436} (2007),  203--234.
%
\bibitem[J]{J}
J.~F.~Jardine,
\textit{A closed model structure for differential graded algebras},
Fields Inst. Commun. \textbf{17} (1997),  55--58.
%
\bibitem[Ka]{Ka}
M. Kapranov,
\textit{Injective resolutions of BG and derived moduli spaces of local systems},
J. Pure and Appl. Algebra \textbf{155} (2001), 167--179.
%
\bibitem[Ka1]{Ka1}
M. Kapranov, \textit{Noncommutative geometry based on commutator expansions},
J. Reine Angew. Math. \textbf{505} (1998), 73--118.
%
\bibitem[K]{Ke}
B. Keller, \textit{On differential graded categories},
International Congress of Mathematicians, Vol. II,  Eur. Math. Soc., Z\"urich, 2006, pp. 151--190.
%
\bibitem[K1]{K}
B. Keller,
\textit{Introduction to $A$-infinity algebras and modules},
Homology Homotopy Appl. \textbf{3}(1) (2001), 1--35.
%
\bibitem[K2]{K2}
B. Keller, \textit{Derived categories and tilting}, in \textit{Handbook of Tilting Theory}, Cambridge
University Press, 2007, pp. 49--104.
%
\bibitem[K3]{K3}
B. Keller,
\emph{Deformed Calabi-Yau completions} (With an appendix by M. Van den Bergh),
J. Reine Angew. Math. \textbf{654} (2011), 125--180.
%
\bibitem[Kn]{Kn}
A.~Knutson, \textit{Some schemes related to the commuting variety},
J. Alg. Geometry \textbf{14} (2005), 283--294.
%
\bibitem[Ko]{Ko}
M.~Kontsevich,
\textit{Formal (non)commutative symplectic geometry}, The Gelfand
Mathematical Seminars 1990-1992, Birkh\"auser, Boston, 1993, pp. 173--187.
%
\bibitem[KR]{KR}
M.~Kontsevich and A.~Rosenberg,
\textit{Noncommutative smooth spaces}, The Gelfand
Mathematical Seminars 1996-1999, Birkh\"auser, Boston, 2000, pp. 85--108;
{\tt arXiv:math/9812158}.
%
\bibitem[LBW]{LBW}
L. Le Bruyn and G. van de Weyer,
\textit{Formal structures and representation spaces}, J. Algebra \textbf{247} (2002), 616--635.
%
\bibitem[LeB]{LeB}
L. Le Bruyn, \textit{Qurves and quivers}, J. Algebra \textbf{290} (2005), 447--472.
%
\bibitem[LeB1]{LeB1}
L. Le Bruyn, \textit{Noncommutative Geometry and Cayley-smooth Orders},
Chapman \& Hall/CRC, Boca Raton, 2008.
%
\bibitem[L]{L}
J.-L.~Loday, \textit{Cyclic Homology}, Grundl. Math. Wiss. \textbf{301}, 2nd Edition,
Springer-Verlag, Berlin, 1998.
%
\bibitem[LQ]{LQ}
J-L. Loday and D. Quillen,
\textit{Cyclic homology and the Lie algebra homology of matrices},
Comment. Math. Helvetici \textbf{59} (1984), 565--591.
%
\bibitem[LV]{LV}
J.-L.~Loday and B. Vallette, \textit{Algebraic Operads}, Grundl. Math. Wiss. \textbf{346},
Springer, Heidelberg, 2012.
%
\bibitem[ML]{ML}
S. Mac Lane, \textit{Categories for the Working Mathematician}, 2nd edition,
Springer-Verlag, New York, 1998.
%
\bibitem[MT]{MT}
G.~Massuyeau and V.~Turaev, \textit{Quasi-Poisson structures on representation spaces of surfaces}, {\tt arXiv:1205.4898}.
%
\bibitem[M]{M}
H.~J.~Munkholm, \textit{DGA algebras as a Quillen model category. Relations to shm maps},
J. Pure Appl. Alg. \textbf{13} (1978), 221--232.
%
\bibitem[Pi]{Pi}
D.~I.~Piontkovski,
\textit{Graded algebras and their differentially graded extensions}, J. Math. Sci. (N. Y.) \textbf{142}(4)
(2007),  2267–-2301.
%
\bibitem[P]{P}
C.~Procesi, \textit{The invariant theory of $ n \times n$ matrices}, Adv. Math. \textbf{19}
(1976), 306--381.
%
\bibitem[Q1]{Q1}
D. Quillen, \textit{Homotopical Algebra}, Lecture Notes in Math. \textbf{43}, Springer-Verlag,
Berlin, 1967.
%
\bibitem[Q2]{Q2}
D. Quillen, \textit{Rational homotopy theory}, Ann. Math. \textbf{90} (1969), 205--295.
%
\bibitem[Q3]{Q3}
D.~Quillen, \textit{Algebra cochains and cyclic cohomology},
Inst. Hautes Etudes Sci. Publ. Math. \textbf{68} (1989), 139--174.
%
\bibitem[Q4]{Q4}
D.~Quillen, \textit{On the (co-) homology of commutative rings},
Proc. Symp. Pure Math. \textbf{17} (1970), 65--87.
%
\bibitem[Re]{Re}
Ch.~Rezk, \textit{Spaces of algebra structures and cohomology of operads},
Ph.~D. thesis, MIT, 1996.
%
\bibitem[R]{R}
A.~Roig, \textit{Model category structures in bifibred categories},
J. Pure Appl. Alg. \textbf{95} (1994), 203--223.
%
\bibitem[S]{S}
B.~Shipley, \textit{Morita theory in stable homotopy categories},
in \textit{Handbook of Tilting Theory}, Cambridge University Press, 2007, pp. 393--411.
%
\bibitem[SS1]{SS1}
S.~Schwede and B.~Shipley, \textit{Algebras and modules in monoidal model categories},
Proc. London Math. Soc. (3) \textbf{80} (2000), 491--511.
%
\bibitem[SS2]{SS2}
S.~Schwede and B.~Shipley, \textit{Equivalences of monoidal model categories},
Alg. Geom. Topology \textbf{3} (2003), 287--334.
%
\bibitem[TT]{TT} D. Tamarkin and B. Tsygan,
{\it Noncommutative differential calculus, homotopy BV algebras and formality conjectures},
Methods Funct. Anal. Topology \textbf{6}(2) (2000), 85--100.
%
\bibitem[T]{T}
B.~Tsygan, \textit{Homology of matrix Lie algebras over rings and the Hochschild homology},
Russian Math. Surveys \textbf{38}(2) (1983), 198--199.
%
\bibitem[T-TT]{T-TT}
T. Thon-That and T-D. Tran,
{\it Invariant theory of a class of infinite dimensional groups}, J. Lie Theory \textbf{13} (2003), 401--425.
%
\bibitem[TV]{TV}
B. To\"en and M.~Vaqui\'e, \textit{Moduli of objects in DG categories},
Ann. Sci. ENS \textbf{40} (2007), 387--444.
%
\bibitem[VdB]{vdB}
M. Van den Bergh,
\emph{Noncommutative quasi-Hamiltonian spaces},
Comtemp. Math. \textbf{450} (2008),  273--299.
%
\bibitem[VdB1]{vdB1}
M.~Van den Bergh, \textit{Double Poisson algebras}, Trans. Amer. Math. Soc. \textbf{360} (2008), 5711-5769.
%
\bibitem[W]{W}
C.~Weibel, \textit{An Introduction to Homological Algebra},
Cambridge Studies in Advanced Mathematics  \textbf{38}, Cambridge University Press,
Cambridge, 1994
\end{thebibliography}
\end{document}